\documentclass[english,11pt]{mysmfart}

\usepackage[utf8]{inputenc}
\usepackage[a4paper,margin=3cm]{geometry}
\usepackage{amssymb,footnote}
\usepackage{multicol}
\usepackage{multirow}
\usepackage{dsfont}
\usepackage{bm}
\usepackage{comment}
\usepackage{microtype}
\usepackage[colorlinks=true,linkcolor=black,citecolor=black,hypertexnames=false]{hyperref}
\usepackage{caption}
\usepackage[ruled,vlined,boxed]{algorithm2e}
\usepackage{stmaryrd}
\usepackage{color}
\usepackage[all]{xy}
\usepackage{enumerate}
\usepackage{tikz}
\usepackage{url}
\usepackage[english]{babel}

    \usepackage{graphicx}
    \makeatletter
    \def\maxwidth{\ifdim\Gin@nat@width>\linewidth\linewidth
    \else\Gin@nat@width\fi}
    \makeatother
    \let\Oldincludegraphics\includegraphics
    \renewcommand{\includegraphics}[1]{\Oldincludegraphics[width=.8\maxwidth]{#1}}

    \usepackage{adjustbox} 
    \usepackage{textcomp} 
    \AtBeginDocument{%
        \def\PYZsq{\textquotesingle}
    }
    \usepackage{upquote} 
    \usepackage{eurosym} 
    \usepackage[mathletters]{ucs} 
    \usepackage{fancyvrb} 
    \usepackage{grffile} 
    \usepackage{longtable} 
    \usepackage{booktabs}  
    \usepackage[normalem]{ulem} 
    \usepackage{mathrsfs}

    \definecolor{urlcolor}{rgb}{0,.145,.698}
    \definecolor{linkcolor}{rgb}{.71,0.21,0.01}
    \definecolor{citecolor}{rgb}{.12,.54,.11}

    \definecolor{ansi-black}{HTML}{3E424D}
    \definecolor{ansi-black-intense}{HTML}{282C36}
    \definecolor{ansi-red}{HTML}{E75C58}
    \definecolor{ansi-red-intense}{HTML}{B22B31}
    \definecolor{ansi-green}{HTML}{00A250}
    \definecolor{ansi-green-intense}{HTML}{007427}
    \definecolor{ansi-yellow}{HTML}{DDB62B}
    \definecolor{ansi-yellow-intense}{HTML}{B27D12}
    \definecolor{ansi-blue}{HTML}{208FFB}
    \definecolor{ansi-blue-intense}{HTML}{0065CA}
    \definecolor{ansi-magenta}{HTML}{D160C4}
    \definecolor{ansi-magenta-intense}{HTML}{A03196}
    \definecolor{ansi-cyan}{HTML}{60C6C8}
    \definecolor{ansi-cyan-intense}{HTML}{258F8F}
    \definecolor{ansi-white}{HTML}{C5C1B4}
    \definecolor{ansi-white-intense}{HTML}{A1A6B2}
    \definecolor{ansi-default-inverse-fg}{HTML}{FFFFFF}
    \definecolor{ansi-default-inverse-bg}{HTML}{000000}

    \DefineVerbatimEnvironment{Highlighting}{Verbatim}{commandchars=\\\{\}}

    \let\Oldtex\TeX
    \let\Oldlatex\LaTeX
    \renewcommand{\TeX}{\textrm{\Oldtex}}
    \renewcommand{\LaTeX}{\textrm{\Oldlatex}}
    \title{demo}

\makeatletter
\def\PY@reset{\let\PY@it=\relax \let\PY@bf=\relax%
    \let\PY@ul=\relax \let\PY@tc=\relax%
    \let\PY@bc=\relax \let\PY@ff=\relax}
\def\PY@tok#1{\csname PY@tok@#1\endcsname}
\def\PY@toks#1+{\ifx\relax#1\empty\else%
    \PY@tok{#1}\expandafter\PY@toks\fi}
\def\PY@do#1{\PY@bc{\PY@tc{\PY@ul{%
    \PY@it{\PY@bf{\PY@ff{#1}}}}}}}
\def\PY#1#2{\PY@reset\PY@toks#1+\relax+\PY@do{#2}}

\expandafter\def\csname PY@tok@w\endcsname{\def\PY@tc##1{\textcolor[rgb]{0.73,0.73,0.73}{##1}}}
\expandafter\def\csname PY@tok@c\endcsname{\let\PY@it=\textit\def\PY@tc##1{\textcolor[rgb]{0.25,0.50,0.50}{##1}}}
\expandafter\def\csname PY@tok@cp\endcsname{\def\PY@tc##1{\textcolor[rgb]{0.74,0.48,0.00}{##1}}}
\expandafter\def\csname PY@tok@k\endcsname{\let\PY@bf=\textbf\def\PY@tc##1{\textcolor[rgb]{0.00,0.50,0.00}{##1}}}
\expandafter\def\csname PY@tok@kp\endcsname{\def\PY@tc##1{\textcolor[rgb]{0.00,0.50,0.00}{##1}}}
\expandafter\def\csname PY@tok@kt\endcsname{\def\PY@tc##1{\textcolor[rgb]{0.69,0.00,0.25}{##1}}}
\expandafter\def\csname PY@tok@o\endcsname{\def\PY@tc##1{\textcolor[rgb]{0.40,0.40,0.40}{##1}}}
\expandafter\def\csname PY@tok@ow\endcsname{\let\PY@bf=\textbf\def\PY@tc##1{\textcolor[rgb]{0.67,0.13,1.00}{##1}}}
\expandafter\def\csname PY@tok@nb\endcsname{\def\PY@tc##1{\textcolor[rgb]{0.00,0.50,0.00}{##1}}}
\expandafter\def\csname PY@tok@nf\endcsname{\def\PY@tc##1{\textcolor[rgb]{0.00,0.00,1.00}{##1}}}
\expandafter\def\csname PY@tok@nc\endcsname{\let\PY@bf=\textbf\def\PY@tc##1{\textcolor[rgb]{0.00,0.00,1.00}{##1}}}
\expandafter\def\csname PY@tok@nn\endcsname{\let\PY@bf=\textbf\def\PY@tc##1{\textcolor[rgb]{0.00,0.00,1.00}{##1}}}
\expandafter\def\csname PY@tok@ne\endcsname{\let\PY@bf=\textbf\def\PY@tc##1{\textcolor[rgb]{0.82,0.25,0.23}{##1}}}
\expandafter\def\csname PY@tok@nv\endcsname{\def\PY@tc##1{\textcolor[rgb]{0.10,0.09,0.49}{##1}}}
\expandafter\def\csname PY@tok@no\endcsname{\def\PY@tc##1{\textcolor[rgb]{0.53,0.00,0.00}{##1}}}
\expandafter\def\csname PY@tok@nl\endcsname{\def\PY@tc##1{\textcolor[rgb]{0.63,0.63,0.00}{##1}}}
\expandafter\def\csname PY@tok@ni\endcsname{\let\PY@bf=\textbf\def\PY@tc##1{\textcolor[rgb]{0.60,0.60,0.60}{##1}}}
\expandafter\def\csname PY@tok@na\endcsname{\def\PY@tc##1{\textcolor[rgb]{0.49,0.56,0.16}{##1}}}
\expandafter\def\csname PY@tok@nt\endcsname{\let\PY@bf=\textbf\def\PY@tc##1{\textcolor[rgb]{0.00,0.50,0.00}{##1}}}
\expandafter\def\csname PY@tok@nd\endcsname{\def\PY@tc##1{\textcolor[rgb]{0.67,0.13,1.00}{##1}}}
\expandafter\def\csname PY@tok@s\endcsname{\def\PY@tc##1{\textcolor[rgb]{0.73,0.13,0.13}{##1}}}
\expandafter\def\csname PY@tok@sd\endcsname{\let\PY@it=\textit\def\PY@tc##1{\textcolor[rgb]{0.73,0.13,0.13}{##1}}}
\expandafter\def\csname PY@tok@si\endcsname{\let\PY@bf=\textbf\def\PY@tc##1{\textcolor[rgb]{0.73,0.40,0.53}{##1}}}
\expandafter\def\csname PY@tok@se\endcsname{\let\PY@bf=\textbf\def\PY@tc##1{\textcolor[rgb]{0.73,0.40,0.13}{##1}}}
\expandafter\def\csname PY@tok@sr\endcsname{\def\PY@tc##1{\textcolor[rgb]{0.73,0.40,0.53}{##1}}}
\expandafter\def\csname PY@tok@ss\endcsname{\def\PY@tc##1{\textcolor[rgb]{0.10,0.09,0.49}{##1}}}
\expandafter\def\csname PY@tok@sx\endcsname{\def\PY@tc##1{\textcolor[rgb]{0.00,0.50,0.00}{##1}}}
\expandafter\def\csname PY@tok@m\endcsname{\def\PY@tc##1{\textcolor[rgb]{0.40,0.40,0.40}{##1}}}
\expandafter\def\csname PY@tok@gh\endcsname{\let\PY@bf=\textbf\def\PY@tc##1{\textcolor[rgb]{0.00,0.00,0.50}{##1}}}
\expandafter\def\csname PY@tok@gu\endcsname{\let\PY@bf=\textbf\def\PY@tc##1{\textcolor[rgb]{0.50,0.00,0.50}{##1}}}
\expandafter\def\csname PY@tok@gd\endcsname{\def\PY@tc##1{\textcolor[rgb]{0.63,0.00,0.00}{##1}}}
\expandafter\def\csname PY@tok@gi\endcsname{\def\PY@tc##1{\textcolor[rgb]{0.00,0.63,0.00}{##1}}}
\expandafter\def\csname PY@tok@gr\endcsname{\def\PY@tc##1{\textcolor[rgb]{1.00,0.00,0.00}{##1}}}
\expandafter\def\csname PY@tok@ge\endcsname{\let\PY@it=\textit}
\expandafter\def\csname PY@tok@gs\endcsname{\let\PY@bf=\textbf}
\expandafter\def\csname PY@tok@gp\endcsname{\let\PY@bf=\textbf\def\PY@tc##1{\textcolor[rgb]{0.00,0.00,0.50}{##1}}}
\expandafter\def\csname PY@tok@go\endcsname{\def\PY@tc##1{\textcolor[rgb]{0.53,0.53,0.53}{##1}}}
\expandafter\def\csname PY@tok@gt\endcsname{\def\PY@tc##1{\textcolor[rgb]{0.00,0.27,0.87}{##1}}}
\expandafter\def\csname PY@tok@err\endcsname{\def\PY@bc##1{\setlength{\fboxsep}{0pt}\fcolorbox[rgb]{1.00,0.00,0.00}{1,1,1}{\strut ##1}}}
\expandafter\def\csname PY@tok@kc\endcsname{\let\PY@bf=\textbf\def\PY@tc##1{\textcolor[rgb]{0.00,0.50,0.00}{##1}}}
\expandafter\def\csname PY@tok@kd\endcsname{\let\PY@bf=\textbf\def\PY@tc##1{\textcolor[rgb]{0.00,0.50,0.00}{##1}}}
\expandafter\def\csname PY@tok@kn\endcsname{\let\PY@bf=\textbf\def\PY@tc##1{\textcolor[rgb]{0.00,0.50,0.00}{##1}}}
\expandafter\def\csname PY@tok@kr\endcsname{\let\PY@bf=\textbf\def\PY@tc##1{\textcolor[rgb]{0.00,0.50,0.00}{##1}}}
\expandafter\def\csname PY@tok@bp\endcsname{\def\PY@tc##1{\textcolor[rgb]{0.00,0.50,0.00}{##1}}}
\expandafter\def\csname PY@tok@fm\endcsname{\def\PY@tc##1{\textcolor[rgb]{0.00,0.00,1.00}{##1}}}
\expandafter\def\csname PY@tok@vc\endcsname{\def\PY@tc##1{\textcolor[rgb]{0.10,0.09,0.49}{##1}}}
\expandafter\def\csname PY@tok@vg\endcsname{\def\PY@tc##1{\textcolor[rgb]{0.10,0.09,0.49}{##1}}}
\expandafter\def\csname PY@tok@vi\endcsname{\def\PY@tc##1{\textcolor[rgb]{0.10,0.09,0.49}{##1}}}
\expandafter\def\csname PY@tok@vm\endcsname{\def\PY@tc##1{\textcolor[rgb]{0.10,0.09,0.49}{##1}}}
\expandafter\def\csname PY@tok@sa\endcsname{\def\PY@tc##1{\textcolor[rgb]{0.73,0.13,0.13}{##1}}}
\expandafter\def\csname PY@tok@sb\endcsname{\def\PY@tc##1{\textcolor[rgb]{0.73,0.13,0.13}{##1}}}
\expandafter\def\csname PY@tok@sc\endcsname{\def\PY@tc##1{\textcolor[rgb]{0.73,0.13,0.13}{##1}}}
\expandafter\def\csname PY@tok@dl\endcsname{\def\PY@tc##1{\textcolor[rgb]{0.73,0.13,0.13}{##1}}}
\expandafter\def\csname PY@tok@s2\endcsname{\def\PY@tc##1{\textcolor[rgb]{0.73,0.13,0.13}{##1}}}
\expandafter\def\csname PY@tok@sh\endcsname{\def\PY@tc##1{\textcolor[rgb]{0.73,0.13,0.13}{##1}}}
\expandafter\def\csname PY@tok@s1\endcsname{\def\PY@tc##1{\textcolor[rgb]{0.73,0.13,0.13}{##1}}}
\expandafter\def\csname PY@tok@mb\endcsname{\def\PY@tc##1{\textcolor[rgb]{0.40,0.40,0.40}{##1}}}
\expandafter\def\csname PY@tok@mf\endcsname{\def\PY@tc##1{\textcolor[rgb]{0.40,0.40,0.40}{##1}}}
\expandafter\def\csname PY@tok@mh\endcsname{\def\PY@tc##1{\textcolor[rgb]{0.40,0.40,0.40}{##1}}}
\expandafter\def\csname PY@tok@mi\endcsname{\def\PY@tc##1{\textcolor[rgb]{0.40,0.40,0.40}{##1}}}
\expandafter\def\csname PY@tok@il\endcsname{\def\PY@tc##1{\textcolor[rgb]{0.40,0.40,0.40}{##1}}}
\expandafter\def\csname PY@tok@mo\endcsname{\def\PY@tc##1{\textcolor[rgb]{0.40,0.40,0.40}{##1}}}
\expandafter\def\csname PY@tok@ch\endcsname{\let\PY@it=\textit\def\PY@tc##1{\textcolor[rgb]{0.25,0.50,0.50}{##1}}}
\expandafter\def\csname PY@tok@cm\endcsname{\let\PY@it=\textit\def\PY@tc##1{\textcolor[rgb]{0.25,0.50,0.50}{##1}}}
\expandafter\def\csname PY@tok@cpf\endcsname{\let\PY@it=\textit\def\PY@tc##1{\textcolor[rgb]{0.25,0.50,0.50}{##1}}}
\expandafter\def\csname PY@tok@c1\endcsname{\let\PY@it=\textit\def\PY@tc##1{\textcolor[rgb]{0.25,0.50,0.50}{##1}}}
\expandafter\def\csname PY@tok@cs\endcsname{\let\PY@it=\textit\def\PY@tc##1{\textcolor[rgb]{0.25,0.50,0.50}{##1}}}

\def\PYZus{\char`\_}

\def\PYZsh{\char`\#}

\def\PYZsq{\char`\'}

\makeatother

    \definecolor{incolor}{rgb}{0.0, 0.0, 0.5}
    \definecolor{outcolor}{rgb}{0.545, 0.0, 0.0}

\newenvironment{myenumerate}[1][1.]{%
  \vspace{-\parskip}
  \begin{enumerate}[#1]\setlength{\itemsep}{0.2ex}}
 {\end{enumerate}}
\newenvironment{myitemize}{%
  \vspace{-\parskip}
  \begin{itemize}\setlength{\itemsep}{0.2ex}}
 {\end{itemize}}

\newcommand{\ph}{$\vphantom{A^A_A}$}

\newtheorem{thm''}{Theorem}
\newtheorem{conj''}[thm'']{Conjecture}
\newtheorem{cor''}[thm'']{Corollary}

\newtheorem{prop'}{Proposition}[section]
\newtheorem{conj'}[prop']{Conjecture}
\newtheorem{lem'}[prop']{Lemma}
\newtheorem{thm'}[prop']{Theorem}
\newtheorem{cor'}[prop']{Corollary}
\theoremstyle{definition}
\newtheorem{hyp'}[prop']{Hypotheses}
\newtheorem{definit'}[prop']{Definition}
\newtheorem{ex'}[prop']{Example}
\newtheorem{ques'}[prop']{Question}
\newtheorem{rem'}[prop']{Remark}

\theoremstyle{plain}
\newtheorem{prop}{Proposition}[subsection]

\newtheorem{conj}[prop]{Conjecture}
\newtheorem{cor}[prop]{Corollary}
\newtheorem{lem}[prop]{Lemma}
\newtheorem{claim}[prop]{Claim}
\newtheorem{thm}[prop]{Theorem}
\theoremstyle{definition}
\newtheorem{definit}[prop]{Definition}
\newtheorem{ex}[prop]{Example}
\newtheorem{rem}[prop]{Remark}

\def\={\buildrel {\rm d\acute ef}\over =}

\DeclareMathOperator{\Card}{Card}

\DeclareMathOperator{\Spec}{Spec}

\DeclareMathOperator{\Ind}{Ind}

\DeclareMathOperator{\Gal}{Gal}

\DeclareMathOperator{\Sym}{Sym}

\DeclareMathOperator{\GL}{GL}
\newcommand{\nr}{\text{\rm nr}}

\renewcommand{\P}{\mathbb{P}}
\newcommand{\GG}{\mathcal{G}}

\newcommand{\RR}{\mathcal{R}}

\newcommand{\SW}{\mathcal{S}}
\newcommand{\SSS}{\Sigma}
\newcommand{\Sac}{\SSS^{\text{\rm ac}}}
\newcommand{\Scomp}{\SSS^{(h,\gamma,\gamma')}}

\newcommand{\I}{\mathrm{I}}
\newcommand{\II}{\mathrm{II}}
\newcommand{\oF}{{\mathcal O}_F}
\newcommand{\bX}{\mathbb{X}}

\newcommand{\oE}{{\mathcal O}_E}

\newcommand{\oK}{{\mathcal O}_K}

\newcommand{\Z}{{\mathbb Z}}
\newcommand{\N}{{\mathbb N}}
\newcommand{\Q}{{\mathbb Q}}

\newcommand{\Qp}{\Q_{p}}

\newcommand{\Zp}{\Z_{p}}

\newcommand{\F}{\mathbb F}

\newcommand{\Qbar}{\overline\Q}
\newcommand{\Qpbar}{\Qbar_p}

\newcommand{\rhobar}{\overline\rho}

\newcommand{\ttt}{{\rm t}}

\def\DD{{\mathcal D}}

\def\WW{{\mathcal W}}

\newcommand{\gA}{\text{\tt A}}
\newcommand{\gB}{\text{\tt B}}
\newcommand{\gAB}{\text{\tt AB}}
\newcommand{\gO}{\text{\tt O}}

\newcommand{\ew}{\hat w}
\newcommand{\EW}{\hat W}
\newcommand{\ga}{\text{\tt a}}
\newcommand{\gb}{\text{\tt b}}
\newcommand{\comb}{\Delta}
\newcommand{\Fib}{\text{\rm Fib}}

\newcommand{\mut}{\bm{\mu}}
\newcommand{\up}{{
  \begin{tikzpicture}[outer sep=0pt, inner sep=0pt,scale=0.45]
  \clip (-0.12,0) rectangle (0.12,0.5);
  \draw (0,0)--(0,0.5);
  \fill (0,0.5)--(-0.1,0.25)--(0.1,0.25);
  \end{tikzpicture}
}}
\newcommand{\down}{{
  \begin{tikzpicture}[outer sep=0pt, inner sep=0pt,scale=0.45,yscale=-1]
  \clip (-0.15,0) rectangle (0.12,0.5);
  \draw (0,0)--(0,0.5);
  \fill (0,0.5)--(-0.1,0.25)--(0.1,0.25);
  \end{tikzpicture}
}}

\newcommand{\dom}{\text{\rm Dom}}
\newcommand{\GR}{\GG\RR}
\newcommand{\GRs}{\overline{\GR}^{\text{\rm s}}}
\newcommand{\GRts}{\widetilde{\GR}^{\text{\rm s}}}
\newcommand{\Vs}{\mathcal V^{\text{s}}}
\newcommand{\Ws}{\mathcal W^{\text{s}}}
\newcommand{\As}{\mathcal A^{\text{s}}}
\newcommand{\Bs}{\mathcal B^{\text{s}}}
\newcommand{\SEKV}{\text{\rm SEKV}}
\newcommand{\pr}{\text{\rm pr}}
\newcommand{\red}{\text{\rm red}}

\def \II{\mathrm{II}}

\renewcommand{\epsilon}{\varepsilon}

\author[X. Caruso]{Xavier Caruso}
\address{CNRS; IMB,
Université de Bordeaux,
351 cours de la Libération,
33405 Talence, France}
\email{xavier.caruso@normalesup.org}

\author[A. David]{Agnès David}
\address{LMB,
Universit\'e de Franche-Comt\'e,
16 route de Gray,
25030 Besançon Cedex, France;
IRMAR,
Université de Rennes I,
Campus de Beaulieu,
35042 Rennes Cedex, France}
\email{agnes.david@math.cnrs.fr}

\author[A. Mézard]{Ariane Mézard}
\address{DMA,
École Normale Supérieure PSL,
45 rue d'Ulm,
75005 Paris, France}
\email{ariane.mezard@ens.fr}

\title[Combinatorics of Serre weights]
{Combinatorics of Serre weights in the\\
potentially Barsotti--Tate setting}

\begin{document}

\begin{abstract}
Let $F$ be a finite unramified extension of $\Qp$ and $\rhobar$ be an 
absolutely irreducible mod~$p$ $2$-dimensional representation of the 
absolute Galois group of $F$. Let $\ttt$ be a tame inertial type 
of $F$.
We conjecture that the deformation space parametrizing the potentially 
Barsotti--Tate liftings of $\rhobar$ having type $\ttt$ depends only
on the Kisin variety attached to the situation, enriched with its 
canonical embedding into $(\P^1)^f$ and its shape stratification.
We give evidences towards this conjecture by proving that the Kisin
variety determines the cardinality of 
the set of common Serre weights $\DD(\ttt,\rhobar) = \DD(\ttt)
\cap \DD(\rhobar)$.
Besides, we prove that this
dependance is nondecreasing (the smaller is the Kisin variety, the
smaller is the number of common Serre weights) and compatible with
products (if the Kisin variety splits as a product, so does the
number of weights).
\end{abstract}

\maketitle

\tableofcontents

\section*{Introduction}

Let $p>2$ be a prime number.
For a finite extension $K$ of $\Q_p$, we let $\oK$ denote its ring of 
integers, $\pi_K$ an uniformizer and $k_K$ its residue field. Let 
$E$ and $F$ be two extensions of $\Q_p$ with $E$ large enough.
We set $G_F=\Gal(\overline{\Q}_p/F)$ and fix a continuous absolutely 
irreducible 
representation $\rhobar: G_F \to \GL_2(k_E)$ together with a lift
$\psi : G_F \to \oE^\times $ of $\det\rhobar$. Let 
$R^{\psi}(\rhobar)$ denote the universal lifting ring of $\rhobar$ over 
$\oE$ with fixed determinant $\psi$.

Given in addition a Hodge type $\lambda$ and an inertial type $\ttt$, 
Kisin \cite{Ki1} has constructed a quotient 
$R^{\psi}(\lambda,\ttt,\rhobar)$ of $R^\psi(\rhobar)$ whose 
$E$-rational points parametrize the potentially crystalline lifts of 
$\rhobar$ having Hodge type $\lambda$ and inertial type $\ttt$.
The central ingredient in Kisin's argument is the construction of
a scheme $\GR^{\psi}(\lambda,\ttt,\rhobar)$ which is a moduli space 
for the so-called Breuil--Kisin modules. This scheme is equipped with
a morphism to $\Spec R^{\psi}(\rhobar)$ whose schematic image is, by
definition, the spectrum of $R^{\psi}(\lambda,\ttt,\rhobar)$.

These deformation rings $R^{\psi}(\lambda,\ttt,\rhobar)$ play a pivotal 
role in many deep contemporary arithmetical subjects. Understanding 
their structure is then a challenging question, which have been and 
continue to have many outstanding applications.
The Breuil-M\'ezard conjecture~\cite{BM1, EmGe} describes the special fibre of $R^{\psi}(\lambda,
\ttt,\rhobar)$ in terms of the representation theory of $\GL_2(\oF)$. 
In its geometrical form, it predicts the following coincidence of cycles 
in $\Spec R^{\psi}(\rhobar)$:
\begin{equation}
\label{ConjBM}
\Spec\big(R^{\psi}(\lambda,\ttt,\rhobar)/(\pi_E)\big) = 
\sum_{\sigma \in \DD}
\mu_{\lambda,\ttt}(\sigma) Z_{\rhobar}(\sigma).
\end{equation}
We recall briefly what are the terms in the right hand side of 
the above equality. The set $\DD$ over which the sum runs is the 
set of Serre weights, which are by definition the irreductible 
representations of $\GL_2(k_F)$ with coefficients in $k_E$.
The factor $\mu_{\lambda,\ttt}(\sigma)$ is the multiplicity of $\sigma$ 
in the Jordan--H\"older decomposition of a $\GL_2(\oF)$-lattice of
$W_\lambda \otimes\sigma(\ttt)$ where $W_\lambda$ 
is the algebraic representation of $\GL_2(\oF)$ associated to 
$\lambda$ and $\sigma(\ttt)$ denotes the 
Bushnell--Kutzko type associated to $\ttt$ (see~\cite{He}).
Finally, $Z_{\rhobar}(\sigma)$ is a certain cycle in $\Spec 
R^{\psi}(\rhobar)$ which depends only on $\rhobar$ and $\sigma$.

The Breuil--Mézard conjecture has been proved for $F=\Q_p$ by 
Kisin~\cite{Ki1} using the $p$-adic local Langlands correspondence for 
$\GL_2(\Q_p)$ and the (global) Taylor--Wiles--Kisin patching argument 
(without assumption of irreducibility of $\rhobar$). Sander \cite{Sa} and Paškūnas~\cite{Pa} 
gave a purely local alternative proof which has been extended later on
by Hu and 
Tan to nonscalar split residual representations~\cite{HuTa}. The case 
where $\lambda = 0$ (which corresponds to potentially Barsotti--Tate 
deformations) is easier to handle. In this case, the Breuil--Mézard 
conjecture was established by Gee and Kisin~\cite{GeKi} (see also 
\cite[Appendix~C]{CEGS}), who then deduced from it the weight part of 
the Serre's conjecture when $F/\Qp$ is unramified and the 
Buzzard--Diamond--Jarvis conjecture~\cite{BDJ}. Some extensions of the 
Breuil--Mézard conjecture to $3$-dimensional representations have also been 
considered by Le, Le Hung, Levin and Morra~ \cite{LLHLM1,LLHLM2}.
In all cases, the Breuil-M\'ezard conjecture is one of the most
concrete and general statement relating representations of $G_F$
and representations of $p$-adic reductive groups and hence it
paves the way towards a $p$-adic Langlands correspondence beyond 
the case of~$\GL_2(\Qp)$.

In~\cite{CDM1,CDM2}, we addressed the question of the effective 
computation of the deformation ring $R^{\psi}(\lambda,\ttt,\rhobar)$ 
and considered the simplest---but already very rich and 
interesting---case where $F = \mathbb{Q}_{p^f}$, $\lambda = 0$ and $\ttt$ is 
tame. From now, we always make these hypothesis and then 
omit the $\lambda$ in the notations. Our strategy for carrying out the 
computation of $R^{\psi}(\ttt,\rhobar)$ was to come back to Kisin's 
construction and consider the scheme $\GR^{\psi} (\ttt,\rhobar)$. We 
first studied its special fibre---which is known as the \emph{Kisin 
variety} associated to $(\ttt,\rhobar)$---and determined explicit 
equations of it. In order to write them down, we associated to 
$(\ttt,\rhobar)$ a simple combinatorial datum (it is a sequence of 
length $2f$ assuming values in the finite set $\{\gA, \gB, \gAB,
\gO\}$) that we called the \emph{gene}, and gave a totally explicit
recipe for finding the equations of the Kisin variety by looking at
the gene.
Moreover, given a gene $\bX$ satisfying mild assumptions, 
we constructed a rigid space 
$D(\bX)$ and conjectured that:
$$\textstyle 
\text{Spm}\Big(R^{\psi}(\ttt,\rhobar)\big[\frac 1 p\big]\Big) 
\simeq D(\bX)$$
as soon as $\bX$ is the gene of $(\ttt,\rhobar)$ (\cite[\S 5.4]{CDM2}).
In other words, we conjectured that the gene determines the generic
fibre of $R^{\psi}(\ttt,\rhobar)$.
We observed in addition that our conjecture is compatible with the 
computations of~\cite{BM2} (which covers all the generic cases) and 
proved that it holds true when $f = 2$.

\subsection*{Presentation of the results}

We write 
$\overline{\GR}(\ttt,\rhobar)$ for the Kisin variety attached to 
$(\ttt,\rhobar)$; it comes equipped with a canonical embedding into 
$(\P^1)^f$ and a stratification (the so-called \emph{shape 
stratification}). Concretely, the stratification is defined by a
upper semi-continuous \emph{shape} function
$g : \overline{\GR}(\ttt,\rhobar)(\bar\F_p) \to \{\I,\II\}^f$ where 
$\I$ and $\II$ are two new symbols with $\I \leq \II$.
Let $\SEKV$ (for ``Stratified Embedded Kisin Varieties'') be the 
set of subvarieties $\Vs$ of $(\P^1)^n$ (for some varying integer~$n$) 
equipped with a upper semi-continuous shape function $g : \Vs 
(\bar\F_p) \to \{\I,\II\}^n$.
The product of varieties defines a structure of multiplicative monoid
on $\SEKV$. With these notations, we propose the following conjecture.

\begin{conj''}
\label{conj'':product}
There exists a morphism of monoids:
$$R : \SEKV \longrightarrow \big\{\,
\text{\rm complete noetherian $\oE$-algebras}\,\big\}$$
(where the multiplicative structure on the codomain is given
by the completed tensor product) with the property that
$R^\psi(\ttt,\rhobar) \simeq R\big(\hspace{0.2ex}\overline{\GR}(\ttt,\rhobar)\big)$
for all absolutely irreducible Galois representation $\rhobar : G_K \to \GL_2(k_E)$ 
and all nondegenerate\footnote{Nondegeneracy is a mild assumption; we 
refer to \cite[Definition~1.1.1]{CDM2} for the definition.} tame 
inertial type.
\end{conj''}

Conjecture~\ref{conj'':product} can be seen as a refinement of the
conjectures of \cite{CDM2}, in the sense that it concerns the entire
deformation space $R^\psi(\ttt,\rhobar)$ and not only its generic
fibre. Moreover, its formulation is more intrinsic in that it only
involves Kisin varieties and avoids the use of genes.

In the present paper, we provide evidences towards
Conjecture~\ref{conj'':product} \emph{at the level of the special fibre};
our results are then orthogonal to that of \cite{CDM2}, the latter
being uniquely concerned with the generic fibre.
By the Breuil--Mézard conjecture, the special fibre of 
$R^\psi(\ttt,\rhobar)$ is described by the 
quantities $\mu_\ttt(\sigma)$ and $Z_{\rhobar}(\sigma)$
appearing in Eq.~\eqref{ConjBM}. Since in our particular setting,
$\mu_\ttt(\sigma)$ is either $0$ or $1$ and $Z_{\rhobar}(\sigma)$
is conjecturally either empty or isomorphic to $\Spec 
\oE[[\underline T]]$, it is sufficient to work with the following
sets of Serre weights:
\begin{align*}
\DD(\ttt) & = \big\{\, \sigma \in \DD \text{ s.t. } 
\mu_{\ttt}(\sigma) \neq 0 \,\big\}, \\
\DD(\rhobar) & = \big\{\, \sigma \in \DD \text{ s.t. } 
Z_{\rhobar}(\sigma) \neq \emptyset \,\big\}, \\
\DD(\ttt,\rhobar) & = \DD(\ttt) \cap \DD(\rhobar).
\end{align*}
With these notations, the $\sigma$'s in $\DD(\ttt,\rhobar)$ are
exactly those that contribute to the sum in the right hand side
of~\eqref{ConjBM}.
Our main results are summarized by Theorem~\ref{thm'':weightsfactor} 
below which can be considered as a numerical version of
Conjecture~\ref{conj'':product}.

\begin{thm''}
[\emph{cf} Theorem~\ref{thm:weightsfactor}
and Corollary~\ref{cor:monotonygene}]
\label{thm'':weightsfactor}
There exists a nondecreasing morphism of monoids $c : \SEKV \to 
(\N, \times)$ with the property that
$\Card \DD(\ttt,\rhobar) = c\big(\hspace{0.2ex}\overline{\GR}(\ttt,\rhobar)\big)$
for all $\rhobar$ and $\ttt$ as in Conjecture~\ref{conj'':product}.
\end{thm''}

Although the formulation of Theorem~\ref{thm'':weightsfactor} again
does not involve genes, its proof makes a decisive use of this notion.
Precisely, we first attach to each gene $\bX$ a set of so-called 
\emph{combinatorial weights}, denoted by $\WW(\bX)$. These weights 
are simply elements of $\{0,1\}^f$, \emph{i.e.} sequences of length 
$f$ with values in $\{0,1\}$.
The construction of $\WW(\bX)$ is purely combinatorial and elementary.
We then prove the following theorem (for which we underline that the
degeneracy assumption is not required).

\begin{thm''}[\emph{cf} 
Theorems~\ref{thm:main} and~\ref{Athm:main}]
\label{thm'':main}
Let $\rhobar : G_K \to \GL_2(k_E)$ be an absolutely irreducible Galois 
representation and $\ttt$ be a tame inertial type.
If the gene of $(\ttt,\rhobar)$ is $\bX$, there is a canonical
bijection $\WW(\bX) \stackrel\sim\longrightarrow \DD(\ttt,\rhobar)$.
\end{thm''}

Once this has been achieved, we use the results of \cite{CDM2} 
(and complete them in several places) in order to relate the gene
$\bX$ with the Kisin variery $\overline{\GR}(\ttt,\rhobar)$
and eventually show that the cardinality of $\DD(\ttt,\rhobar)$ 
depends only on $\overline{\GR}(\ttt,\rhobar)$.
Finally, the monotony property and the compatibility with products
follow by examining closely the construction of $\WW(\bX)$.

\subsection*{Some perspectives}

Theorem~\ref{thm'':main} is much finer that 
Theorem~\ref{thm'':weightsfactor} in the sense that it not only 
provides a formula for the cardinality of $\DD(\ttt,\rhobar)$ but it 
establishes a bijection between the latter set and an auxiliary set 
of weights.
This observation leads us to our first vague question: can we attach to 
any $\Vs \in \SEKV$ a set of weights $\WW(\Vs)$ 
and, given $(\ttt,\rhobar)$ as before, construct a canonical bijection:
$$\DD(\ttt,\rhobar) \stackrel\sim\longrightarrow
\WW\big(\overline{\GR}^{\psi}(\ttt,\rhobar)\big)$$
in line with Theorem~\ref{thm'':weightsfactor}?
It is worth noticing that the set $\WW(\bX)$ does not make the job 
because it does not only depend on $\overline{\GR}(\ttt,\rhobar)$. 
However, beyond this obvious obstruction, we really want elements of 
$\WW(\Vs)$ to be intrinsic objects with a deep arithmetical meaning 
(\emph{e.g.} modular representations of some group) and not simply 
sequences of integers. A new approach is then definitely needed.

Another striking fact, we think, is that most of the combinatorial 
constructions and arguments we shall develop in this article are 
completely independent of~$p$. Basically, the notion of gene, and
consequently all its derivations, do not involve any prime number.
In particular, 
the fact that the equations of $\overline{\GR}(\ttt,\rhobar)$ are
explicitely given by the gene shows that the Kisin varieties
themselves are mostly independant from~$p$; to be precise, they are
defined over $\F_1$, the field with one element.
This strong uniformity with respect to~$p$ might suggest that,
at least in the particular case we are looking at (\emph{i.e.}
$2$-dimensional potentially Barsotti-Tate representations with
tame inertial type), the
$p$-adic Langlands correspondence could be ``defined over''
a common base for all the $\Zp$'s, \emph{e.g.} $\Z$ or, maybe, the
speculative ring of Witt vectors over $\F_1$~\cite{Co}.
Could we dream of a $1$-adic Langlands correspondence whose purpose
would be to summarize in a nice and uniform way all the combinatorial
parts of the $p$-adic Langlands correspondence?

\subsection*{Organization of the paper}

In \S \ref{sec:Definitions}, we introduce the notations and the 
main objects of this article: Galois representations, inertial type, 
Serre weights and genes. 
In \S \ref{sec: Results}, we construct the set $\WW(\bX)$ and state a
precise version of Theorem~\ref{thm'':main} under a mild 
nondenegaracy assumption on the gene $\bX$. Several results about the
cardinality of $\WW(\bX)$ are discussed.
In \S \ref{sec:End}, we clarify the relationships between genes and
Kisin varieties and prove Theorem~\ref{thm'':weightsfactor}. 
For the convenience of the reader, the proof of 
Theorem~\ref{thm'':main}, which is long and intricated, is postponed
to \S \ref{sec:proof}.
The article is supplemented with two appendices. 
Appendix~\ref{app:degenerate} is of technical nature and aims at 
removing the aforementioned nondegenarcy assumption.
In Appendix \ref{sec:algo}, we design efficient algorithms for counting 
and enumerating weights. An implement of these algorithms in SageMath 
is also presented.

\paragraph*{Grants}

The three authors are supported by the ANR project CLap--CLap
(ANR-18-CE40-0026-01).

\section{Representations, types, weights and genes}
\label{sec:Definitions}

The aim of this introductory section is to present the panel of 
objects we work with in this article. These are Galois 
representations, inertial types (\S \ref{ssec:reptypes}), Serre 
weights (\S \ref{ssec:serreweights}) and genes (\S
\ref{ssec:defgene}-\ref{ssec:abstractgene}).
We recall the classification of $2$-dimensional irreducible modular 
Galois representations and tame inertial types and the definition by
Breuil and Paškūnas~\cite{BP} of the set of Serre weights associated
to these objects (\S \ref{ssec:serreweights}). 
We recall, from \cite{CDM2}, the definition of the gene and give
its first properties.

Throughout this paper, we fix a prime number $p > 2$ together with
an algebraic closure $\Qpbar$ of $\Qp$. All the algebraic extensions 
we consider in this article live inside $\Qpbar$.
We consider an integer $f \geq 2$ and set $q = p^f$. We let $F$ be 
the unique unramified extension of $\Qp$ of degree $f$. We denote
its ring of integers by $\oF$; it is a local ring with maximal
ideal $(p)$ and residue field of cardinality $q$.
Let $G_F = \Gal(\Qpbar/F)$ be the absolute Galois group of $F$.

We fix a finite extension $E$ of $\Qp$ and denote its ring of 
integers $\oE$ and its residue field by $k_E$. This extension $E$
is the field of coefficients of our representations. 
Actually, if it had been possible, it would have been easier to 
work with representations with coefficients in $\Qpbar$; however, 
finiteness assumptions are needed in several places, which discards
this option. It is the reason why we introduce $E$; nevertheless, we 
constantly allow ourselves to enlarge $E$ if needed.

\subsection{Representations and types}
\label{ssec:reptypes}

The first objects of interest we study in this paper are
\emph{Galois representations} of $G_F$.
Precisely, we consider a continuous $2$-dimensional absolutely
irreducible 
representation $\rhobar : G_F \to \GL_2(k_E)$.
Such a representation actually admits a very concrete description.
Indeed, let $F'$ be the unique unramified extension of $F$ of
degree $2$ and set $G_{F'} = \Gal(\Qpbar/F')$. Assuming that $E$
contains $F'$, there exists an integer $h$ and an element $\theta
\in k_E^\times$ such that $\rhobar$ takes the form:
\begin{equation}
\label{eq:rhobar}
\rhobar \,\simeq\, 
\Ind_{G_{F'}}^{G_F} \Big( \omega_{2f}^h \cdot \nr'(\theta)\Big)
\end{equation}
where $\omega_{2f}$ is the fundamental character of $G_{F'}$ of 
level $2f$
and $\nr'(\theta)$ denotes the unique unramified character of
$G_{F'}$ sending the arithmetic Frobenius to $\theta$.
In what precedes, $h$ must be not divisible by $q{+}1$ (in order
to guarantee that $\rhobar$ is absolutely irreducible). Moreover, it is 
uniquely determined modulo $q^2{-}1$ and modulo the transformation
$h \mapsto qh$.

The second objects we are interested in are \emph{inertial
types}. Let $I_F$ denote the inertial subgroup of $G_F$.
By definition, an \emph{inertial type} (or simply a \emph{type}) is 
a representation of $I_F$ which admits a prolongation to $G_F$.
In the present article, we only work with $2$-dimensional
tame types, that are types which factor through the tame inertia.
Like Galois representations, those types admit very concrete 
descriptions as they all take the form:
\begin{equation}
\label{eq:ttt}
\ttt = \omega_f^\gamma \oplus \omega_f^{\gamma'}
\end{equation}
where $\omega_f$ denotes (the restriction to $I_F$ of) the fundamental 
character of $G_F$ of level $f$ and $\gamma$ and $\gamma'$ are integers, 
which are uniquely defined modulo $q{-}1$.
 
To those data, one can associate deformation spaces. For this, we fix in 
addition a character $\psi : G_F \to \oE^\times$ lifting $\det\rhobar$.
We let $R^\psi(\rhobar)$ denote the universal deformation ring
parametrizing the liftings of $\rhobar$ having determinant $\psi$.
From \cite{Ki3}, we know that there exists a unique reduced quotient 
$R^\psi(\ttt, \rhobar)$ of $R^\psi(\rhobar)$ whose $E$-rational
points parametrize the representations $\rho: G_F \to \GL_2(E)$ of
determinant $\psi$ and semi-simplification $\rhobar$, exhibiting in
addition the two following extra properties:

\begin{myenumerate}[(1)]
\item $\rho$ is potentially crystalline of type $\ttt$;
\item the Hodge--Tate weights of $\rho$ are $\{0, 1\}$ for each
embedding $F \hookrightarrow E$.
\end{myenumerate}

In order to have a chance that $R^\psi(\ttt, \rhobar)$ does
not vanish, one has to impose a compatibility condition on the
determinants. It reads:
\begin{equation}
\label{eq:cohdet}
\det \rhobar_{|I_F} = \epsilon \cdot \det \ttt
\end{equation}
where $\epsilon$ denotes the cyclotomic character mod $p$. 
Coming back to the explicit descriptions of $\rhobar$ and $t$, this
condition translates to a numerical congruence connecting the 
parameters $h$, $\gamma$ and $\gamma'$, namely:
\begin{equation}
\label{eq:coherence}
h \equiv \gamma + \gamma' + \frac{q-1}{p-1} \pmod{q-1}.
\end{equation}

\begin{definit}
\label{def:triple}
A \emph{coherent triple} is a triple 
$$(h, \gamma, \gamma') \,\,\in \,\,
\Z/(q^2{-}1)\Z \,\times\,
\Z/(q{-}1)\Z \,\times\,
\Z/(q{-}1)\Z$$ 
such that $h$ is not divisible by $q{+}1$ and the congruence
\eqref{eq:coherence} holds.
\end{definit}

A coherent triple then encodes a pair $(\rhobar_{|I_F}, \ttt)$ 
with $\rhobar$ absolutely irreducible and condition~\eqref{eq:cohdet} satisfied.
We underline that this encoding is surjective but not injective. 
Precisely, two triples 
$(h_1, \gamma_1, \gamma'_1)$ and
$(h_2, \gamma_2, \gamma'_2)$ correspond to isomorphic $(\rhobar_{|I_F}, 
\ttt)$ when $h_1 \equiv q h_2 \pmod{q^2{-}1}$ or $(\gamma_1,
\gamma'_1) = (\gamma'_2,\gamma_2)$.
In what follows, we often work with coherent triples instead
of pairs $(\ttt, \rhobar)$ because they are more suitable for the 
combinatorial constructions we develop in this paper.

\subsection{Serre weights}
\label{ssec:serreweights}

In this subsection, we recall the notation for Serre weights and the 
definitions of the sets of Serre weights associated to an irreducible
$2$-dimension representation $\rhobar$ and a tame inertial type
$\ttt$.

\subsubsection{Definition of Serre weights}
A Serre weight is an (isomorphism class of an) absolutely irreducible 
$k_E$-representation of the group $\GL_2(k_F)$. Again, one has a concrete description of Serre
weights. Before giving it, we need to introduce further notations.
First, we let $\tau_0$ denote the inclusion $F 
\hookrightarrow E$
and, for each integer $i$, we set $\tau_i = \tau_0^{p^{-i}}$. The
$\tau_i$'s then exhaust all the field embeddings from $F$ into $E$.
Second, given a positive integer $r$ and an embedding
$\tau : F \to E$, we let $(\Sym^r k_E^2)^{\tau}$ denote
the space of homogeneous bivariate polynomials over $k_E$ of degree 
$r$ and endow it with the action of $\GL_2(k_F)$ defined by:
$$\left(\begin{matrix} a & b \\ c & d \end{matrix}\right) \cdot
P(X,Y) = 
P\big(\tau(a)X + \tau(b)Y,\,\, \tau(c)X + \tau(d)Y)\big).$$
With these notations, one proves that any Serre weight takes the following 
form:
\begin{prop}
Any Serre weight can be written as
\begin{equation}
\label{eq:weight}
\sigma_{s, \underline r} \,\simeq\,
(\tau_0\circ\det{}^{s}) \,\otimes\,
\bigotimes_{i=0}^{f-1}\, (\Sym^{r_i}k_E^2)^{\tau_i}
\end{equation}
where $s$ is an integer modulo $q{-}1$ and the $r_i$'s are 
integers in the range $\{0, 1, \ldots, p{-}1\}$ with
$\underline r = (r_0, \ldots, r_{f-1}) \neq (p{-}1, \ldots,
p{-}1)$. Moreover, in the above writing, both $s$ and $\underline r$ 
are uniquely determined.
\end{prop}

\paragraph*{Serre weights associated to 
Galois representation and inertial type}

To each Galois representation $\rhobar$ (resp. inertial type $\ttt$) as in \S \ref{ssec:reptypes},
one associates a set of Serre weights, which is denoted 
by $\DD(\rhobar)$ (resp. by $\DD(\ttt)$). 
The construction of these sets is given by a combinatorial recipe
that we recall now.

Let $\rhobar$ be the $2$-dimensional absolutely irreducible Galois 
representation associated with the parameters $h$ and $\theta$ thanks
to Eq.~\eqref{eq:rhobar}. We consider the following equation:
\begin{equation}\label{equahSerre}
h \equiv \sum_{i=0}^{f-1}(-1)^{\varepsilon_i}p^i(1+r_i)
\pmod {q+1}
\end{equation}
whose unknowns are $\underline \varepsilon = (\varepsilon_0,
\ldots, \varepsilon_{f-1})$ and $\underline r = (r_0, \ldots, 
r_{f-1})$.
Precisely we seek solutions of this equation with $\varepsilon_i \in 
\{0, 1\}$ and $r_i \in \{0, \ldots, p{-}1\}$ for all $i$.
To each such solution, one associates the Serre weight
$\sigma_{s, \underline r}$ where $s$ is defined by:
\begin{equation}
\label{eq:reps}
s \equiv 
\frac 1 {q+1} \left( h -
   \sum_{i=0}^{f-1} (-1)^{\varepsilon_i} p^i(1+r_i)\right) -
\sum_{i=0}^{f-1} \varepsilon_i p^i(1+r_i) \pmod{q-1}.
\end{equation}
\begin{definit}
The set $\DD(\rhobar)$ is defined as the set of Serre 
weights associated to solutions of (\ref{equahSerre}) .
\end{definit} 

 Let $\ttt = \omega_f^\gamma \oplus
\omega_f^{\gamma'}$ be an inertial type and let $c_0, \ldots,
c_{f-1} \in \{0, 1, \ldots, p{-}1\}$ be the integers defined by
the following congruence:
$$\gamma - \gamma' \equiv \sum_{i=0}^{f-1} c_i p^i \pmod {q-1}.$$
When $\gamma \equiv \gamma' \pmod{q-1}$, then $c_0 = 
\cdots = c_{f-1} = 0$. Given $\varepsilon'_0, \ldots, \varepsilon'_{f-1}
\in \{0, 1\}$, we form the tuple $\underline r = (r_0, \ldots, 
r_{f-1})$ defined by the rules summarized in the following table 
(see \cite{BP,Da}):
\begin{equation}
\label{eq:typetable}
\raisebox{-0.5\height}{%
\begin{tikzpicture}[xscale=2.5, yscale=0.6]
\draw (0,1)--(2,1);
\draw (-1,0)--(2,0);
\draw (-1,-1)--(2,-1);
\draw (-1,-2)--(2,-2);
\draw (-1,0)--(-1,-2);
\draw (0,1)--(0,-2);
\draw (1,1)--(1,-2);
\draw (2,1)--(2,-2);
\node at (-0.5,0.5) { $r_i$ };
\node at (0.5,0.5) { $\varepsilon'_{i-1} = 0$ };
\node at (1.5,0.5) { $\varepsilon'_{i-1} = 1$ };
\node at (-0.5,-0.5) { $\varepsilon'_i = 0$ };
\node at (-0.5,-1.5) { $\varepsilon'_i = 1$ };

\node at (0.5,-0.5) { \ph $c_i$ };
\node at (0.5,-1.5) { \ph $p - 2 - c_i$ };
\node at (1.5,-0.5) { \ph $c_i - 1$ };
\node at (1.5,-1.5) { \ph $p - 1 - c_i$ };
\end{tikzpicture}}
\end{equation}

\noindent
where we have set $\varepsilon'_{-1} = \varepsilon'_{f-1}$.
When all the $r_i$'s lie in $\{0, \ldots, p{-}1\}$,
we further form the Serre weights $\sigma_{s, \underline r}$ where 
$s \in \Z/(q{-}1)\Z$ is given by the formula:
\begin{equation}
\label{eq:typetwist}
s \equiv \gamma' + \frac 1 2 \left( \varepsilon'_{f-1} (q-1) + 
\sum_{i=0}^{f-1} (c_i - r_i) \: p^i\right) \pmod{q-1}.
\end{equation}
If one of the $r_i$'s falls outside the interval
$[0, p{-}1]$, the procedure fails and does not produce a Serre
weight. 
\begin{definit}
The set $\DD(\ttt)$ is the collection of all the Serre
weights obtained by Table  (\ref{eq:typetable}) when the family $(\varepsilon'_i)_
{i \in \Z/f\Z}$ varies in $\{0,1\}^f$.
\end{definit}

\begin{lem}
\label{lem:alts}
With the above notations, the integer $s$ is alternatively given by:
$$s \equiv \gamma' + 
         \sum_{i=0}^{f-1} \varepsilon'_i \: (p-1-r_i) \: p^i
         \pmod{q-1}.$$
Moreover, for any integer $i_0 \in \{0, \ldots, f{-}1\}$, we have the 
formula:
$$s \equiv \gamma' + 
         \frac 1 2 \left( \varepsilon'_{i_0} (q-1) + 
           \sum_{i=0}^{f-1} \lambda_{i+i_0} (c_i - r_i)  p^i\right) 
         \pmod{q-1}$$
where, by definition, $\lambda_j = q$ if $j < f{-}1$ and $\lambda_j = 1$
otherwise.
\end{lem}

\begin{proof}
It follows by inspection that
$c_i - r_i = \varepsilon'_i\: (p-2-2r_i) - \varepsilon'_{i-1}$ for 
all $i$. Injecting this relation and reorganizing terms, we obtain:
$$\sum_{i=0}^{f-1} (c_i - r_i)  \: p^i
= \varepsilon'_{f-1} (1-q) \, + \,
  2 \sum_{i=0}^{f-1} \varepsilon'_i \: (p - 1 - r_i)  \: p^i.$$
Plugging this in Eq.~\eqref{eq:typetwist}, we get the first part of
the lemma. The second part is proved similarly.
\end{proof}

\begin{rem}
\label{rem:uniqueepsp}
Lemma~\ref{lem:alts} shows that the datum of a type $\ttt$ and 
a weight $\sigma_{s, \underline r} \in \DD(\ttt)$ uniquely determines
the $\varepsilon'_i$'s since any single $\varepsilon'_{i_0}$ can be 
recovered from $s$, $\underline r$ and the $c_i$'s thanks to second
formula of Lemma~\ref{lem:alts}.
\end{rem}

\paragraph*{Common weights}

Given $\rhobar$ and $\ttt$ as above, we define the set of common weights as
\begin{definit} The set of common Serre weights of $\rhobar$ and $\ttt$ is
$$\DD(\ttt, \rhobar)
= \DD(\rhobar) \cap \DD(\ttt).$$ 
\end{definit}
If we fix in addition a character
$\psi : G_F \to \oE^\times$ lifting $\det\rhobar$,
the Breuil--Mézard conjecture (which is a theorem in this setting, \cite{GeKi}) 
relates the set $\DD(\ttt,\rhobar)$ to the special fibre 
of the deformation space $\Spec R^\psi(\ttt, \rhobar)$. More
precisely, its geometric version states that there is an equality
of cycles in $\Spec (k_E \otimes_{\oE} R^\psi(\rhobar))$:
$$\Spec \big(k_E \otimes_{\oE} R^\psi(\ttt, \rhobar)\big)
= \sum_{\sigma \in \DD(\ttt,\rhobar)} Z(\sigma)$$
where $Z(\sigma)$ depends only on $\sigma$. In our context, it 
is moreover expected that the $Z(\sigma)$'s are all smooth but,
to the best of our knowledge, this has not been proved yet.
In this article, we give an explicit combinatorial 
description of the set $\DD(\ttt, \rhobar)$ (Theorem~\ref{thm:main}) and, consequently,
shed new lights on the description of special fibre of
$R^\psi(\ttt, \rhobar)$.

\subsection{The gene of $(h, \gamma, \gamma')$}
\label{ssec:defgene}

In \cite{CDM2}, with the perspective of writing down explicit
equations of the deformation rings $R^\psi(\ttt,\rhobar)$,
we have associated to each pair $(\ttt, \rhobar)$ a combinatorial
data that we called the \emph{gene}.
We have conjectured (see \cite[Conjecture~5.1.6]{CDM2}) that the gene 
determines the generic fibre of $R^\psi(\ttt,\rhobar)$ (by an 
explicit recipe) and, as a first step in this direction, we have shown 
that it actually determines the Kisin variety associated to $(\ttt, 
\rhobar)$ (see \S \ref{sssec:cdm2} for more details).

In this subsection, we recall the definition of the gene and
prove some additional combinatorial facts about it.
We fix a coherent triple $(h, \gamma,\gamma')$ and
consider the integers $h_0, \ldots, h_{f-1}$ in
$\{0, \ldots, p{-}1\}$ uniquely defined by the congruence:
\begin{equation}
h \equiv 1 + \sum_{i=0}^{f-1} h_i p^{f-1-i} \pmod{q+1}.
\end{equation}
We set $h_i = p - 1 - h_{i-f}$ for $i \in \{f, f{+}1, \ldots,
2f{-}1\}$ and $h_i = h_{i \text{ mod } 2f}$ for all $i \in \Z$.
By construction, the sequence $(h_i)_{i \in \Z}$ is periodic with
period $2f$.

\begin{definit}\label{defalphai}
Let $\nu=p^{f-1}+\cdots+p$. For $i \in \Z$, we define:
\begin{myenumerate}[(1)]
\item the integer $\alpha_i$ as the unique integer in $\{0, \ldots,
q{-}2\}$ satisfying the congruence:
$$\alpha_i \equiv
\left\lfloor \frac{p^i h}{q+1} \right\rfloor - p^i \gamma' \pmod {q-1}$$
\item the symbol $X_i \in \{\gA, \gB, \gAB, \gO\}$ by:
$$\begin{array}{rcll}
X_i & = & \gA & \text{if }
\alpha_i \in \big[0, \, \frac 1 p \: \nu + \iota_{i+f} \big[
\medskip \\
& = & \gAB & \text{if }
\alpha_i \in \big[\frac 1 p \: \nu + \iota_{i+f},
\, \frac {p-1} p \: \nu - \iota_i \big] \medskip \\
& = & \gB & \text{if }
\alpha_i \in \, \big]\frac {p-1} p \: \nu -
\iota_i, \, \nu\big] \medskip \\
& = & \gO & \text{if }
\alpha_i \in \, \big]\nu, e \big[
\end{array}$$
where $\iota_i= 1$ if $h_i = p-1$ and $\iota_i = 0$ otherwise.
\end{myenumerate}
The sequence $\bX = (X_i)_{i \in \Z}$ is 
called the \emph{gene} of $(h, \gamma, \gamma')$.
\end{definit}

As we did in \cite{CDM2}, we will always picture a gene on 
two rows, writing down the $f$ first symbols $X_0, \ldots,
X_{f-1}$ on the top row and the others on the bottom one
(this is justified by the fact that the pair $(X_i, X_{i+f})$
plays a quite important role).
For example, the gene of length~$7$
$$\bX = 
(\ldots, X_0, \ldots, X_{13}, \ldots) = 
(\ldots, \gO, \gA, \gB, \gA, \gAB, \gO, \gA,
 \gB, \gA, \gAB, \gO, \gO, \gB, \gAB, \ldots)$$
is drawn as follows:

\medskip

\noindent\hfill%
\begin{tikzpicture}[scale=0.8]
\draw [fill=yellow!20] (-0.5,-0.5) rectangle (6.5,1.5);
\node at (0, 0) { $\gB$ };
\node at (0, 1) { $\gO$ };
\node at (1, 0) { $\gA$ };
\node at (1, 1) { $\gA$ };
\node at (2, 0) { $\gAB$ };
\node at (2, 1) { $\gB$ };
\node at (3, 0) { $\gO$ };
\node at (3, 1) { $\gA$ };
\node at (4, 0) { $\gO$ };
\node at (4, 1) { $\gAB$ };
\node at (5, 0) { $\gB$ };
\node at (5, 1) { $\gO$ };
\node at (6, 0) { $\gAB$ };
\node at (6, 1) { $\gA$ };
\end{tikzpicture}%
\hfill\null

The next proposition is proved in \cite[Lemme~2.1.4]{CDM2}.

\begin{prop}
\label{prop:gene}
Let $\bX = (X_i)_{i\in\Z}$ be the gene of $(h, \gamma,
\gamma')$. Then:
\begin{myenumerate}[(i)]
\item if $X_i = \gAB$ for some integer $i$, then $X_{i+1} = \gO$;
\item if $X_i = \gO$ for some integer $i$, then $X_{i-1} \in \{\gAB, \gO\}$.
\end{myenumerate}
\end{prop}

Given a coherent triple $(h, \gamma, \gamma')$, we let
$v_0, \ldots, v_{2f-1}$ be the unique integers in the range $[0,p{-}1]$
such that:
\begin{equation}
\label{eq:vi}
h - (q{+}1)\gamma' \equiv 
p^{2f-1} v_0 + p^{2f-2} v_1 + \cdots + p v_{2f-2} + v_{2f-1}
\pmod{q^2 - 1}.
\end{equation}
We extend the definition of the $v_i$'s to all indices $i \in \Z$
by letting $v_i = v_{i \mod 2f}$.
It turns out that the $v_i$'s are closely related to the gene of 
$(h,\gamma,\gamma')$. The next lemma, which we use repeatedly 
in this article, makes these relationships precise.

\begin{lem}
\label{lem:vi}
Let $\bX = (X_i)_{i\in\Z}$ be the gene of $(h, \gamma,
\gamma')$ and $(v_i)_{i\in \Z}$ defined by Eq.~(\ref{eq:vi}). 
For all integers $i \in \Z$, the
following holds:
\begin{myenumerate}[(i)]
\item[(a)] if $X_i = \gA$, then $v_i = 0$;
\item[(ab)] if $X_i = \gAB$, then $v_i = 0$;
\item[(b)] if $X_i = \gB$, then $v_i = 1$;
\item[(o)] if $X_i = \gO$, then $v_i \geq 1$.
\end{myenumerate}
\end{lem}

\begin{proof}
Write $v = h - (q{+}1)\gamma'$.
A direct computation shows that $p^i v \equiv a_i q + a_{i+f}$
where the $a_i$'s are defined by
$a_i = p^{f-1} v_i + p^{f-2} v_{i+1} + \cdots + v_{i+f-1}$.
Moreover, we have:
$$\begin{array}{r@{\hspace{0.5ex}}l@{\quad}l}
\displaystyle \alpha_i 
   = \left\lfloor\frac{p^i v}{q+1} \right\rfloor \text{ mod } (q{-}1)
   = \left\lfloor a_i + \frac{a_{i+f} - a_i}{q+1} \right\rfloor 
 & {}= a_i & \text{if } a_{i+f} \geq a_i \\
 & {}= a_i-1 & \text{otherwise.}
\end{array}$$
Let us now assume that $X_i = \gA$. By definition, we then have
$$\alpha_i \leq 1 + p + p^2 + \cdots + p^{f-2} < p^{f-1} - 1.$$
Thus
$a_i \leq \alpha_i + 1 < p^{f-1}$ and, coming back to the definition 
of the $a_i$'s, we deduce $v_i = 0$. We have proved~(a).
Similarly, if $X_i = \gO$, we obtain $a_i \geq p + p^2 + \cdots
+ p^{f-1}$ and then deduce $v_i \geq 1$, which proves~(o).

Let us now consider the case where $X_i = \gAB$. Then
$\alpha_i \leq p^{f-1} - 1$ and so $a_i \leq p^{f-1}$. We deduce
from this that $v_i = 0$ except maybe in the very special case
where $a_i = p^{f-1}$. However, if this happens, we deduce in
addition that $v_{i+1} = \cdots = v_{i+f-1} = 0$. This implies in
particular that $X_{i+1} \neq \gO$, which is a contradiction.
Consequently, $v_i = 0$ in all cases and we have proved~(ab).

Finally, if $X_i = \gB$, we get the estimation:
$$p^{f-1} - 1 \leq a_i \leq p + p^2 + \cdots + p^{f-1}.$$
Therefore $v_i = 1$ except maybe in the particular case where
$a_i = p^{f-1} - 1$. However, in this case, we also have 
$v_{i+1} = \cdots = v_{i+f-1} = p{-}1$. From~(a) and~(ab), we deduce
that $X_{i+1}$ cannot be $\gA$ nor $\gAB$. It cannot also be $\gB$
by what we have just done. Therefore $X_{i+1}$ must be $\gO$, which 
contradicts the fact that $X_i = \gB$.
\end{proof}

\begin{cor}
\label{cor:gene}
Let $\bX = (X_i)_{i\in\Z}$ be the gene associated to a coherent triple $(h, \gamma,
\gamma')$. Then
there exists an integer $i$ such that $X_i = \gO$ or $X_i \neq X_{i+f}$.
\end{cor}

\begin{proof}
We argue by contradiction and assume that $X_i = X_{i+f} \in \{\gA, 
\gB\}$ for all $i$. Note that it is safe to assume $X_i \neq \gAB$ 
because otherwise, we would have $X_{i+1} = \gO$, which we exclude.
Let $v_0, \ldots, v_{2f}$ be the integers introduced above.
From Lemma~\ref{lem:vi}, we know
that $v_i = 0$ if $X_i = \gA$, and $v_i = 1$ if $X_i = \gB$. From
our assumption, we then derive that $v_i = v_{i+f}$ for all $i$.
Plugging this in the congruence~(\ref{eq:vi}), we find that $h$ is 
divisible by $q{+}1$, 
which contradicts the fact that $\rhobar$ is absolutely irreducible.
\end{proof}

\begin{rem}
We prove in \S \ref{ssec:algogene} several refined versions 
of Lemma~\ref{lem:vi} and Corollary~\ref{cor:gene} (see for
instance Proposition~\ref{prop:vi}).
\end{rem}

\subsection{Abstract genes}\label{ssec:abstractgene}

For the purpose of this article, it is relevant to introduce an 
abstract definition of a gene; it is the aim of this subsection.
Our definition is simply obtained by gathering the properties 
enlightened in Proposition~\ref{prop:gene} and 
Corollary~\ref{cor:gene}.

\begin{definit}
\label{def:gene}
A \emph{gene} of length $f$ is a $(2f)$-periodic sequence 
$\bX = (X_i)_{i \in \Z}$ assuming
values in the finite set $\{\gA,\gB,\gAB,\gO\}$ and satisfying the
following conditions:
\begin{myenumerate}[(i)]
\item if $X_i = \gAB$ for some integer $i$, then $X_{i+1} = \gO$;
\item if $X_i = \gO$ for some integer $i$, then $X_{i-1} \in \{\gAB, \gO\}$;
\item there exists an integer $i$ such that $X_i = \gO$ or $X_i \neq X_{i+f}$.
\end{myenumerate}
\vspace{-\parskip}
A \emph{gene} $\bX = (X_i)_{i \in \Z}$ is said \emph{degenerate} if 
$X_i \neq \gO$ for all $i$; other it is \emph{nondegenerate}.
\end{definit}

\begin{rem}
In \cite[Lemme~2.1.6]{CDM2}, the case where $\{X_i, X_{i+f}\} = 
\{\gA,\gB\}$ for all index $i$ is also excluded because it corresponds 
to a type with $\gamma = \gamma'$. In the current paper, we do not want 
to discard such types; it is the reason why we have not incorporated 
this condition in Definition~\ref{def:gene}.
\end{rem}

\begin{rem}
By \cite[Lemme~2.1.5]{CDM2}, we know that such degenerate genes 
correspond to very special representations, which are called
degenerate in \emph{loc. cit.} (see \cite[Definition~1.1.1]{CDM2}).
The terminology is then coherent.
\end{rem}

The next proposition shows that Definition~\ref{def:gene} exactly captures the genes we are interested in.

\begin{prop}
We assume $p > 3$. Given a gene $\bX$, 
there exists a coherent triple $(h, \gamma, \gamma')$ 
whose associated gene is $\bX$.
\end{prop}

\begin{proof}
Let $(v_i)_{i \in \Z}$ be a sequence of integers such that:
$$\begin{array}{r@{\hspace{0.5ex}}l@{\quad}l}
v_i &= 0 & \text {if } X_i \in \{\gA, \gAB\}, \\
v_i &= 1 & \text {if } X_i = \gB, \\
2 \leq v_i &\leq p{-}1 & \text {if } X_i = \gO
\end{array}$$
and $v_i \neq v_{i+f}$ for at least one index $i$. 
The condition~(iii) of Definition~\ref{def:gene} guarantees that
such a sequence always exists.

We set $h = p^{2f-1} v_0 + p^{2f-2} v_1 + \cdots + v_{2f-1}$ and
$\gamma' = 0$. We let $\gamma$ be an integer for which the 
compatibility relation \eqref{eq:coherence} holds. Let us first
check that $h$ is not divisible by $q{+}1$. A simple calculation
shows that $q{+}1$ divides $h$ if and only if
$$\sum_{i=0}^{f-1} p^{f-1-i} v_{i+f} 
\equiv \sum_{i=0}^{f-1} p^{f-1-i} v_i
\pmod{q{+}1}.$$
Since the $v_i$'s are between $0$ and $p{-}1$, this can only occur
if $v_i = v_{i+f}$ for all~$i$, which does not hold by construction.

Let $(\alpha'_i)_{i \in \Z}$ and $(v'_i)_{i \in \Z}$ be the 
sequences of integers associated to $(h, \gamma,\gamma')$ (see
\S \ref{ssec:defgene}) and $\bX' = (X'_i)_{i \in \Z}$
be the corresponding gene. It is clear from the construction 
of $h$ and $\gamma'$ that $v'_i = v_i$ for all $i$. We have to show 
that this implies $\bX = \bX'$.

For this, we consider an integer $i$.
 If $X_i = \gO$, then $v'_i > 1$
and Lemma~\ref{lem:vi} ensures that $X'_i = \gO$ as well.
If $X_i = \gAB$, then $X_{i+1} = \gO$.
By what we have done before, we find $X'_{i+1} = \gO$ as well. Hence 
$X'_i \in \{\gAB, \gO\}$. Since $v'_i = 0$,
Lemma~\ref{lem:vi} then guarantees that $X'_i$ cannot be $\gO$. 
Therefore $X'_i = \gAB$.
If $X_i = \gB$,
Lemma~\ref{lem:vi} shows that $X'_i \in \{\gB,\gO\}$, which is 
\emph{a priori} not enough to conclude. However, coming back to the
proof of Lemma~\ref{lem:vi} and setting
$$a'_i = p^{f-1}v'_i + p^{f-2} v'_{i+1} + \cdots + v'_{i+f-1}$$
we find $a'_i \leq p + p^2 + \cdots + p^{f-1}$. Thus
$\alpha'_i \leq p + p^2 + \cdots + p^{f-1}$ as well, which discards
the option $X'_i = \gO$.
Finally, if $X_i = \gA$, applying Lemma~\ref{lem:vi}, we
obtain $X'_i \in \{\gA, \gAB\}$. But $X'_i = \gAB$  would
imply that $X'_{i+1} = \gO$, which is impossible because $X_{i+1}
\in \{\gA, \gB, \gAB\}$.
\end{proof}

To conclude this section, we recall the definition of viability.

\begin{definit}
\label{def:viable}
Let $\bX = (X_i)_{i \in \Z}$ be a gene. We say that $\bX$ is
\emph{not viable} if there exists $i$ such that $X_i = X_{i+1}
= \gO$. It is \emph{viable} otherwise.
\end{definit}

It has been shown in \cite[Proposition 4.1.3]{CDM2} that the gene of 
$(h, \gamma, \gamma')$ is viable if and only if the corresponding 
deformation space is not zero.

\section{Combinatorial weights of a gene}
\label{sec: Results}

In this section, we first associate to a gene $\bX$ a 
set of \emph{combinatorial weights}, denoted by $\WW(\bX)$. For this, 
we decompose the gene $\bX$ into fragments and work on each fragment 
separately. The set $\WW(\bX)$ is defined as the product of sets of 
weights associated to each fragment.
Then, we give several properties of $\WW(\bX)$; we prove in particular 
that $\WW(\bX)$ is not empty if and only if the gene is viable (see 
Theorem \ref{thm:empty}) and give general upper bounds on the 
cardinality of $\WW(\bX)$ in terms of Fibonacci numbers (see 
Theorem~\ref{thm:fibo}).
Finally, we formulate a precise statement of Theorem~\ref{thm'':main}
of the introduction.

For the convenience of the readers, we always assume in this section 
that the gene $\bX$ is nondegenerate. 
The constructions in the general case follows the same pattern but
are much more technical; they are presented in 
Appendix~\ref{app:degenerate}.

\subsection{Construction of $\WW(\bX)$}
\label{ssec:weightgene} \label{ssec:withO} 

We begin with the definition of the combinatorial weights.

\begin{definit}
\label{def:combweight}
A \emph{combinatorial weight} of length $f$ is a $f$-periodic
sequence assuming values in $\{0, 1\}$.
\end{definit}

\noindent
By definition, a combinatorial weight of length $f$ is determined by 
its values on $\{0, 1, \ldots, f{-}1\}$. We then can alternatively
think of it as a subset of $\{0, 1, \ldots, f{-}1\}$. In any case,
there are exactly $2^f$ combinatorial weights of length $f$.

When the gene $\bX$ is not viable,
we just set $\WW(\bX) = \emptyset$. The rest of this subsection is 
devoted to the construction of $\WW(\bX)$ when $\bX$ is viable.
From now on, we pick a nondegenerate gene $\bX$ (in the sense of 
Definition~\ref{def:gene}) and assume that it is viable.

To start with, we decompose $\bX$ into fragments $\underline F$ by 
cutting vertically the gene before each occurrence of $\gO$.
Before giving formal definitions, it is enlighting to visualize
the constructions on a simple example.

\begin{ex}
\label{ex:fragments}
The following gene:

\medskip

\noindent\hfill%
\begin{tikzpicture}[scale=0.8]
\draw [fill=yellow!20] (-0.5,-0.5) rectangle (6.5,1.5);
\node at (0, 0) { $\gB$ };
\node at (0, 1) { $\gO$ };
\node at (1, 0) { $\gA$ };
\node at (1, 1) { $\gA$ };
\node at (2, 0) { $\gAB$ };
\node at (2, 1) { $\gB$ };
\node at (3, 0) { $\gO$ };
\node at (3, 1) { $\gA$ };
\node at (4, 0) { $\gO$ };
\node at (4, 1) { $\gAB$ };
\node at (5, 0) { $\gB$ };
\node at (5, 1) { $\gO$ };
\node at (6, 0) { $\gAB$ };
\node at (6, 1) { $\gA$ };
\end{tikzpicture}%
\hfill\null

\medskip

\noindent
has four fragments $\underline F$, which are:

\medskip

\noindent\hfill%
\begin{tikzpicture}[scale=0.8]
\begin{scope}
\node at (0, 0) { \ph $\gB$ };
\node at (0, 1) { \ph $\gO$ };
\node at (1, 0) { \ph $\gA$ };
\node at (1, 1) { \ph $\gA$ };
\node at (2, 0) { \ph $\gAB$ };
\node at (2, 1) { \ph $\gB$ };
\node at (3.5, 0.5) { ; };
\end{scope}
\begin{scope}[xshift=2cm]
\node at (3, 0) { \ph $\gO$ };
\node at (3, 1) { \ph $\gA$ };
\node at (4.5, 0.5) { ; };
\end{scope}
\begin{scope}[xshift=4cm]
\node at (4, 0) { \ph $\gO$ };
\node at (4, 1) { \ph $\gAB$ };
\node at (5.5, 0.5) { ; };
\end{scope}
\begin{scope}[xshift=6cm]
\node at (5, 0) { \ph $\gB$ };
\node at (5, 1) { \ph $\gO$ };
\node at (6, 0) { \ph $\gAB$ };
\node at (6, 1) { \ph $\gA$ };
\end{scope}
\end{tikzpicture}%
\hfill\null
\end{ex}

Formally, a \emph{fragment} can be defined as follows.

\begin{definit}
\label{defi:fragment}
A \emph{fragment} $\underline F$
of length $\ell$ is a tuple $(F_0, 
F_1, \ldots, F_{\ell-1})$ where each $F_i$ is a pair $(F^\up_i, 
F^\down_i) \in \{\gA, \gB, \gAB, \gO\}^2$ satisfying the following 
requirements:

\begin{myitemize}
\item[(L)] $F^\up_0 = \gO$ or $F^\down_0 = \gO$, 
but $(F^\up_0, F^\down_0) \neq (\gO, \gO)$,

\smallskip

\item[(C)]
$F^\up_i \neq \gO$ and $F^\down_i \neq \gO$ for $i > 0$,

\noindent
$F^\up_i \neq \gAB$ and $F^\down_i \neq \gAB$ for $i < \ell{-}1$,

\smallskip

\item[(R)]
if $\ell > 1$: $F^\up_{\ell-1} = \gAB$ or $F^\down_{\ell-1} = \gAB$, 
but $(F^\up_{\ell-1}, F^\down_{\ell-1}) \neq (\gAB, \gAB)$.
\end{myitemize}
\end{definit}

To a fragment $\underline F$ of length $\ell$, we associate a 
set $\WW(\underline F)$ of fragmentary combinatorial weights, which 
are tuples $(w_0, \ldots, w_{\ell-1})$ in $\{0,1\}^\ell$.
In order to handle smoothly corner cases, it is convenient to introduce
the equivalence relation $\sim$ on $\{\gA, \gB, \gAB, \gO\}$ whose two
equivalence classes are $\{\gA, \gAB\}$ and
$\{\gB, \gO\}$.

\begin{definit}
\label{def:fragmentweight}
Let $\underline F = (F_0, \ldots, F_{\ell-1}$) be a fragment of
length $\ell$ and write $F_i = (F^\up_i, F^\down_i)$.
We define three sequences 
$(W^{(\gb,\gb)}_i)_{0 \leq i < \ell}$,
$(W^{(\ga,\gb)}_i)_{0 \leq i < \ell}$,
$(W^{(\gb,\ga)}_i)_{0 \leq i < \ell}$ by:

\medskip

$\begin{array}{cr@{\hspace{0.5ex}}l@{\qquad}l}
\bullet
 & W^{(\gb,\gb)}_0
 & = \emptyset
 & \text{if } \ell = 1 \text{ and } \big(F^\up_0 \in \{\gA, \gB\}
                         \text{ or } F^\down_0 \in \{\gA, \gB\}\big) \smallskip \\
&& = \{1\}
 & \text{otherwise}
\end{array}$

\smallskip

$\begin{array}{cr@{\hspace{0.5ex}}l@{\qquad}l}
\bullet
 & W^{(\ga,\gb)}_0
 & = \emptyset
 & \text{if } F^\down_0 = \gO \smallskip \\
&& = \{0\}
 & \text{otherwise}
\end{array}$

\smallskip

$\begin{array}{cr@{\hspace{0.5ex}}l@{\qquad}l}
\bullet
 & W^{(\gb,\ga)}_0
 & = \emptyset
 & \text{if } F^\up_0 = \gO \smallskip \\
&& = \{0\}
 & \text{otherwise}
\end{array}$

\medskip

\noindent
and the following recurrence formulas (for $1 \leq i \leq \ell-1$):

\medskip

$\begin{array}{cr@{\hspace{0.5ex}}l@{\qquad}l}
\bullet
 & W^{(\gb,\gb)}_i 
 & = \big(W^{(\ga,\gb)}_{i-1} \cup W^{(\gb,\ga)}_{i-1}\big) 
     \times \{1\}
 & \text{if } F^\up_{i-1} \sim F^\down_{i-1} \smallskip \\
&& = W^{(\gb,\gb)}_{i-1}
    \times \{1\}
 & \text{otherwise}
\end{array}$

\smallskip

$\begin{array}{cr@{\hspace{0.5ex}}l@{\qquad}l}
\bullet
 & W^{(\ga,\gb)}_i 
 & = W^{(\ga,\gb)}_{i-1}
     \times \{0\}
 & \text{if } F^\up_i \sim F^\up_{i-1} \smallskip \\
&& = \big(W^{(\gb,\ga)}_{i-1} \cup W^{(\gb,\gb)}_{i-1}\big) 
    \times \{0\}
 & \text{otherwise}
\end{array}$

\smallskip

$\begin{array}{cr@{\hspace{0.5ex}}l@{\qquad}l}
\bullet
 & W^{(\gb,\ga)}_i 
 & = W^{(\gb,\ga)}_{i-1}
     \times \{0\}
 & \text{if } F^\down_i \sim F^\down_{i-1} \smallskip \\
&& = \big(W^{(\ga,\gb)}_{i-1} \cup W^{(\gb,\gb)}_{i-1}\big) 
    \times \{0\}
 & \text{otherwise.}
\end{array}$

\medskip

\noindent
We then set:

\medskip

$\begin{array}{r@{\hspace{0.5ex}}l@{\qquad}l}
\WW(\underline F)
 & = W^{(\gb,\gb)}_{\ell-1} \cup W^{(\ga,\gb)}_{\ell-1}
 & \text{if } F^\down_{\ell-1} = \gAB \smallskip \\
 & = W^{(\gb,\gb)}_{\ell-1} \cup W^{(\gb,\ga)}_{\ell-1}
 & \text{if } F^\up_{\ell-1} = \gAB \smallskip \\
 & = W^{(\gb,\gb)}_0 \cup W^{(\ga,\gb)}_0 \cup W^{(\gb,\ga)}_0
 & \text{otherwise.}
\end{array}$

\medskip

\noindent
Notice that the last case can only show up when $\ell = 1$.
\end{definit}

\begin{ex}
\label{ex:fragmentweights}
Let us compute the set of combinatorial weights associated to
the fragments of Example~\ref{ex:fragments}. For the first one,
following the definitions, we get:

\medskip

$\begin{array}{cr@{\hspace{0.5ex}}l}
[\:i{=}0\:]:
& W^{(\gb,\gb)}_0 
 & = \{1\} \quad ; \quad
  W^{(\ga,\gb)}_0 
   = \{0\} \quad ; \quad
  W^{(\gb,\ga)}_0 
   = \emptyset \medskip \\{}
[\:i{=}1\:]:
& W^{(\gb,\gb)}_1 
 & = \big(W^{(\ga,\gb)}_0 \cup W^{(\gb,\ga)}_0\big) \times \{1\}
   = \big\{(0,1)\big\} \smallskip \\
& W^{(\ga,\gb)}_1 
 & = \big(W^{(\gb,\ga)}_0 \cup W^{(\gb,\gb)}_0\big) \times \{0\}
   = \big\{(1,0)\big\} \smallskip \\
& W^{(\gb,\ga)}_1
 & = \big(W^{(\ga,\gb)}_0 \cup W^{(\gb,\gb)}_0\big) \times \{0\}
   = \big\{(0,0), (1,0)\big\} \medskip \\{}
[\:i{=}2\:]:
& W^{(\gb,\gb)}_2 
 & = \big(W^{(\ga,\gb)}_1 \cup W^{(\gb,\ga)}_1\big) \times \{1\}
   = \big\{(0,0,1), (1,0,1)\big\} \smallskip \\
& W^{(\ga,\gb)}_2 
 & = \big(W^{(\gb,\ga)}_1 \cup W^{(\gb,\gb)}_1\big) \times \{0\}
   = \big\{(0,0,0), (1,0,0), (0,1,0)\big\} \smallskip \\
& W^{(\gb,\ga)}_2 
 & = W^{(\gb,\ga)}_0 \times \{0\}
   = \big\{(0,0,0), (1,0,0)\big\}
\end{array}$

\medskip

\noindent
and therefore:
$$\WW\Big(
\raisebox{-0.5cm}{%
\begin{tikzpicture}[scale=0.8,yscale=0.6]
\begin{scope}
\node at (0, 0) { \ph $\gB$ };
\node at (0, 1) { \ph $\gO$ };
\node at (1, 0) { \ph $\gA$ };
\node at (1, 1) { \ph $\gA$ };
\node at (2, 0) { \ph $\gAB$ };
\node at (2, 1) { \ph $\gB$ };
\end{scope}
\end{tikzpicture}}\Big)
= W^{(\gb,\gb)}_2 \cup W^{(\ga,\gb)}_2
= \big\{(0,0,1), (1,0,1), (0,0,0), (1,0,0), (0,1,0)\big\}.$$

\noindent
Similarly, we obtain:
$$\WW\Big(
\raisebox{-0.5cm}{%
\begin{tikzpicture}[scale=0.8,yscale=0.6]
\begin{scope}
\node at (0, 0) { \ph $\gO$ };
\node at (0, 1) { \ph $\gA$ };
\end{scope}
\end{tikzpicture}}\Big)
= \big\{0\big\}
\quad ; \quad
\WW\Big(
\raisebox{-0.5cm}{%
\begin{tikzpicture}[scale=0.8,yscale=0.6]
\begin{scope}
\node at (0, 0) { \ph $\gO$ };
\node at (0, 1) { \ph $\gAB$ };
\end{scope}
\end{tikzpicture}}\Big)
= \big\{0, 1\big\}
\quad ; \quad
\WW\Big(
\raisebox{-0.5cm}{%
\begin{tikzpicture}[scale=0.8,yscale=0.6]
\begin{scope}
\node at (0, 0) { \ph $\gB$ };
\node at (0, 1) { \ph $\gO$ };
\node at (1, 0) { \ph $\gAB$ };
\node at (1, 1) { \ph $\gA$ };
\end{scope}
\end{tikzpicture}}\Big)
= \big\{(0,1), (1,0)\big\}.$$
\end{ex}

We now come to the definition of the set of combinatorial weights
of a gene.

\begin{definit}
\label{def:geneweight}
We set:
$$\WW(\bX) \,=\, \prod_{\underline F}\,\WW(\underline F)$$
where the product runs over all fragments $\underline F$ of
$\bX$ (and the coordinates of the fragmentary combinatorial weights
go at the corresponding positions).
\end{definit}

\begin{ex}
\label{ex:geneweights}
After  Example~\ref{ex:fragmentweights}, the gene of Example~\ref{ex:fragments} has $5 \times 1 \times
2 \times 2 = 20$ combinatorial weights, which are:
$$\begin{array}{l@{\,\,}l@{\,\,}l@{\,\,}l}
  (0,0,1,0,0,0,1), & (0,0,1,0,0,1,0), 
& (0,0,1,0,1,0,1), & (0,0,1,0,1,1,0), \smallskip \\
  (1,0,1,0,0,0,1), & (1,0,1,0,0,1,0),
& (1,0,1,0,1,0,1), & (1,0,1,0,1,1,0), \smallskip \\
  (0,0,0,0,0,0,1), & (0,0,0,0,0,1,0),
& (0,0,0,0,1,0,1), & (0,0,0,0,1,1,0), \smallskip \\
  (1,0,0,0,0,0,1), & (1,0,0,0,0,1,0),
& (1,0,0,0,1,0,1), & (1,0,0,0,1,1,0), \smallskip \\
  (0,1,0,0,0,0,1), & (0,1,0,0,0,1,0), 
& (0,1,0,0,1,0,1), & (0,1,0,0,1,1,0).
\end{array}$$
\end{ex}


\subsection{On the number of combinatorial weights}
\label{ssec:count}

The viability of the gene and the non-emptyness of the set of 
combinatorial weights are closely related.

\begin{thm}
\label{thm:empty}
Let $\bX$ be a nondegenerate gene. Then $\WW(\bX)$ is not empty 
if and only if $\bX$ is viable.
\end{thm}

\begin{proof}
If $\bX$ is not viable, the set $\WW(\bX)$ is empty by definition. It 
is then enough to prove that $\WW(\bX)$ is not empty as soon as $\bX$ 
is viable, which amounts to proving that $\WW(\underline F) \neq
\emptyset$ for any fragment $\underline F$.
The latter assertion is easily checked by hand when $\underline F$ has 
length $1$ by examining all cases.
We now assume that $\underline F$ has length $\ell \geq 2$.
An easy induction on $i$ shows that
$W_i^{(\gb,\gb)}$ and $W_i^{(\ga,\gb)} \cup W_i^{(\gb,\ga)}$ are
always not empty for $i \in \{0, \ldots, \ell{-}1\}$. Since
$\WW(\underline F)$ contains $W_{\ell-1}^{(\gb,\gb)}$, it is 
also not empty.
\end{proof}

Beyond the emptiness characterization of Theorem~\ref{thm:empty}, one 
can count the number of combinatorial weights of $\WW(\bX)$ without 
having to write them down all explicitly. The cardinality of $\WW(\bX)$ 
is the product of the cardinality of $\WW(\underline F)$ for 
$\underline F$ the fragments of $\bX$. Moreover, it turns out that the 
cardinality of $\WW(\underline F)$ are given by simple recursive 
formulas.
Here is the key lemma allowing for deriving them.

\begin{lem}
\label{lem:inclweight}
Let $\underline F$ be a fragment of length $\ell$ and let
$(W^{(\gb,\gb)}_i)_{0 \leq i < \ell}$,
$(W^{(\ga,\gb)}_i)_{0 \leq i < \ell}$,
$(W^{(\gb,\ga)}_i)_{0 \leq i < \ell}$ be the sequences defined
in Definition~\ref{def:fragmentweight}. Then, for all $i \in
\{0, \ldots, \ell-1\}$, we have:
\begin{myenumerate}[(i)]
\item $W^{(\ga,\gb)}_i \cap W^{(\gb,\gb)}_i = \emptyset$,
\item $W^{(\gb,\ga)}_i \cap W^{(\gb,\gb)}_i = \emptyset$,
\item $W^{(\ga,\gb)}_i \subset W^{(\gb,\ga)}_i$ or
$W^{(\gb,\ga)}_i \subset W^{(\ga,\gb)}_i$.
\end{myenumerate}
\end{lem}

\begin{proof}
The two first assertions are obvious since the last coordinate
of an element of $W^{(\ga,\gb)}_i$ (resp. $W^{(\gb,\ga)}_i$)
is always $0$ whereas the last coordinate of an element of
$W^{(\gb,\gb)}_i$ is always $1$.

We prove the assertion~(iii) by induction on $i$. For $i = 0$,
it follows immediately from the definition. Let us now assume
that~(iii) holds for the index $i{-}1$. We distinguish between
four cases.

If $F^\up_i \sim F^\up_{i-1}$
and $F^\down_i \sim F^\down_{i-1}$, then:
$$W^{(\ga,\gb)}_i = W^{(\ga,\gb)}_{i-1} \times \{0\}
\quad;\quad
W^{(\gb,\ga)}_i = W^{(\gb,\ga)}_{i-1} \times \{0\}$$
and the conclusion follows directly by induction.

If $F^\up_i \not\sim F^\up_{i-1}$
and $F^\down_i \sim F^\down_{i-1}$, then:
$$W^{(\gb,\ga)}_i = W^{(\gb,\ga)}_{i-1} \times \{0\}
\subset \big(W^{(\gb,\ga)}_{i-1} \cup W^{(\gb,\gb)}_{i-1}\big) \times \{0\}
= W^{(\ga,\gb)}_i.$$
If $F^\up_i \sim F^\up_{i-1}$
and $F^\down_i \not\sim F^\down_{i-1}$, then:
$$W^{(\ga,\gb)}_i = W^{(\ga,\gb)}_{i-1} \times \{0\}
\subset \big(W^{(\ga,\gb)}_{i-1} \cup W^{(\gb,\gb)}_{i-1}\big) \times \{0\}
= W^{(\gb,\ga)}_i.$$
If $F^\up_i \not\sim F^\up_{i-1}$
and $F^\down_i \not\sim F^\down_{i-1}$, then:
$$W^{(\ga,\gb)}_i 
 = \big(W^{(\ga,\gb)}_{i-1} \cup W^{(\gb,\gb)}_{i-1}\big) \times \{0\}
\quad ; \quad
W^{(\gb,\ga)}_i 
 = \big(W^{(\gb,\ga)}_{i-1} \cup W^{(\gb,\gb)}_{i-1}\big) \times \{0\}$$
and the conclusion again follows by induction (the order is reversed).
\end{proof}

For $0 \leq i < \ell$ and $\square \in \{(\ga,\gb), (\gb,\ga), 
(\gb,\gb)\}$, we set $c^\square_i = \Card W^\square_i$. 
Definition~\ref{def:fragmentweight} and Lemma~\ref{lem:inclweight} 
together show that the $c^\square_i$'s are subject to the following 
recurrence relations:

$\begin{array}{cr@{\hspace{0.5ex}}l@{\qquad}l}
\bullet
 & c^{(\gb,\gb)}_i 
 & = \max\big(c^{(\ga,\gb)}_{i-1},  c^{(\gb,\ga)}_{i-1}\big) 
 & \text{if } F^\up_{i-1} \sim F^\down_{i-1} \smallskip \\
&& = c^{(\gb,\gb)}_{i-1}
 & \text{otherwise} \medskip \\
\bullet
 & c^{(\ga,\gb)}_i 
 & = c^{(\ga,\gb)}_{i-1}
 & \text{if } F^\up_i \sim F^\up_{i-1} \smallskip \\
&& = c^{(\gb,\ga)}_{i-1} + c^{(\gb,\gb)}_{i-1}
 & \text{otherwise} \medskip \\
\bullet
 & c^{(\gb,\ga)}_i 
 & = c^{(\gb,\ga)}_{i-1}
 & \text{if } F^\down_i \sim F^\down_{i-1} \smallskip \\
&& = c^{(\ga,\gb)}_{i-1} + c^{(\gb,\gb)}_{i-1}
 & \text{otherwise}
\end{array}$

\medskip

\noindent
and that the cardinality we are looking for is finally given by:

\begin{cor}
\label{cor:countwithO} \label{cor:cardcomb}
The cardinality of the set of combinatorial weights associated to a 
fragment ${\underline F}$ is given by:

\medskip

\noindent\hfill%
$\begin{array}{r@{\hspace{0.5ex}}l@{\qquad}l}
\Card \WW(\underline F)
 & = c^{(\gb,\gb)}_{\ell-1} + c^{(\ga,\gb)}_{\ell-1}
 & \text{if } F^\down_{\ell-1} = \gAB, \smallskip \\
 & = c^{(\gb,\gb)}_{\ell-1} + c^{(\gb,\ga)}_{\ell-1}
 & \text{if } F^\up_{\ell-1} = \gAB.
\end{array}$%
\hfill\null
\end{cor}

Beyond the fact that these recursive formulas are well suited for
a direct and simple computation of $\Card \WW(\underline F)$, they
have interesting corollaries.
In order to state them, let us call
$(\Fib_i)_{i \geq 0}$ the Fibonacci sequence defined by
$\Fib_0 = 0$, $\Fib_1 = 1$ and $\Fib_i = \Fib_{i-1} + \Fib_{i-2}$ 
for $i \geq 2$.

\begin{thm}
\label{thm:fibo}
Let $\underline F$ be a fragment of length $\ell \geq 2$.
Then $\Card \WW(\underline F) \leq \Fib_{\ell+2}$ and equality
occurs if and only if $\underline F$ is the following fragment:

\begin{tikzpicture}[yscale=0.5,xscale=0.8]
\node[right] at (-4,1) { \ph if $\ell$ is even: };
\node at (0, 1) { \ph $\gO$ };
\node at (0, 0) { \ph $\gB$ };
\node at (1, 1) { \ph $\gA$ };
\node at (1, 0) { \ph $\gA$ };
\node at (2, 1) { \ph $\gB$ };
\node at (2, 0) { \ph $\gB$ };
\node at (3, 0.6) { \ph $\ldots$ };
\node at (4, 1) { \ph $\gA$ };
\node at (4, 0) { \ph $\gA$ };
\node at (5, 1) { \ph $\gB$ };
\node at (5, 0) { \ph $\gB$ };
\node at (6, 1) { \ph $\gAB$ };
\node at (6, 0) { \ph $\gA$ };
\end{tikzpicture}

\begin{tikzpicture}[yscale=0.5,xscale=0.8]
\node[right] at (-4,1) { \ph if $\ell$ is odd: };
\node at (0, 1) { \ph $\gO$ };
\node at (0, 0) { \ph $\gB$ };
\node at (1, 1) { \ph $\gA$ };
\node at (1, 0) { \ph $\gA$ };
\node at (2, 1) { \ph $\gB$ };
\node at (2, 0) { \ph $\gB$ };
\node at (3, 0.6) { \ph $\ldots$ };
\node at (4, 1) { \ph $\gA$ };
\node at (4, 0) { \ph $\gA$ };
\node at (5, 1) { \ph $\gB$ };
\node at (5, 0) { \ph $\gB$ };
\node at (6, 1) { \ph $\gA$ };
\node at (6, 0) { \ph $\gA$ };
\node at (7, 1) { \ph $\gB$ };
\node at (7, 0) { \ph $\gAB$ };
\end{tikzpicture}

\noindent
or the fragment obtained from the above one by interverting the 
letters $\gA$ and $\gB$.
\end{thm}

\begin{proof}
We keep the notations introduced above the statement of the theorem.
We are going to prove  by induction on $i$ that:
$$c_i^{(\gb,\gb)} \leq \Fib_{i+1},
\quad c_i^{(\ga,\gb)} \leq \Fib_{i+2}
\quad \text{and} \quad 
c_i^{(\gb,\ga)} \leq \Fib_{i+2}.$$
The statement is clearly true for $i = 0$. For $i > 0$, we have:
$$\begin{array}{r@{\hspace{0.5ex}}l@{\qquad}l}
c^{(\gb,\gb)}_i 
 & = \max\big(c^{(\ga,\gb)}_{i-1},  c^{(\gb,\ga)}_{i-1}\big) 
   \leq \Fib_{i+1}
 & \text{if } F^\up_{i-1} \sim F^\down_{i-1} \smallskip \\
 & = c^{(\gb,\gb)}_{i-1}
   \leq \Fib_i \leq \Fib_{i+1}
 & \text{otherwise} \medskip \\
c^{(\ga,\gb)}_i 
 & = c^{(\ga,\gb)}_{i-1} 
   \leq \Fib_{i+1} \leq \Fib_{i+2}
 & \text{if } F^\up_i \sim F^\up_{i-1} \smallskip \\
 & = c^{(\gb,\ga)}_{i-1} + c^{(\gb,\gb)}_{i-1}
   \leq \Fib_{i+1} + \Fib_i = \Fib_{i+2}   
 & \text{otherwise} \medskip \\
c^{(\gb,\ga)}_i 
 & = c^{(\gb,\ga)}_{i-1} 
   \leq \Fib_{i+1} \leq \Fib_{i+2}
 & \text{if } F^\down_i \sim F^\down_{i-1} \smallskip \\
 & = c^{(\ga,\gb)}_{i-1} + c^{(\gb,\gb)}_{i-1}
   \leq \Fib_{i+1} + \Fib_i = \Fib_{i+2}   
 & \text{otherwise.}
\end{array}$$
We finally conclude that:
$$\begin{array}{r@{\hspace{0.5ex}}l@{\qquad}l}
\Card \WW(\underline F)
 & = c^{(\gb,\gb)}_{\ell-1} + c^{(\ga,\gb)}_{\ell-1}
   \leq \Fib_\ell + \Fib_{\ell+1} = \Fib_{\ell+2}
 & \text{if } F^\down_{\ell-1} = \gAB \smallskip \\
 & = c^{(\gb,\gb)}_{\ell-1} + c^{(\gb,\ga)}_{\ell-1}
   \leq \Fib_\ell + \Fib_{\ell+1} = \Fib_{\ell+2}
 & \text{if } F^\up_{\ell-1} = \gAB.
\end{array}$$
The first assertion of the theorem is then proved.

Let us now study the case of equality.
It holds if and only all inequalities encoutered along the way are
equalities. Let us assume first that $F^\down_{\ell-1} = \gAB$.
We then derive that $c^{(\gb,\gb)}_{\ell-1}$ and $c^{(\ga,\gb)}_{\ell-1}$ 
must be equal to $\Fib_\ell$ and $\Fib_{\ell+1}$ respectively.
Noticing that $\Fib_i < \Fib_{i+1}$ as soon as $i \geq 2$, these
equalities imply by  descending induction that $c^{(\gb,\gb)}_i = 
\Fib_i$ for all $i \in \{0, \ldots, \ell-1\}$ and that:
\begin{myitemize}
\item if $\ell{-}i$ is odd, then
$c^{(\ga,\gb)}_i = \Fib_{i+2}$ and $F^\up_i \not\sim F^\up_{i-1}$,
\item if $\ell{-}i$ is even, then
$c^{(\gb,\ga)}_i = \Fib_{i+2}$ and $F^\down_i \not\sim F^\down_{i-1}$.
\end{myitemize}
Similarly, the equality $c^{(\gb,\gb)}_i = \Fib_i$ is only possible
if $F^\up_{i-1} \sim F^\down_{i-1}$ for all $i \geq 2$.
The two possible shapes of $\underline F$ given in the statement
of the theorem follow from these observations.

Finally, the case where $F^\up_{\ell-1} = \gAB$ is treated similarly.
\end{proof}

In Theorem~\ref{thm:fibo}, we have intentionally discarded the 
fragments with $\gO$  of length $1$. It is actually not difficult to handle
these fragments by hand.
Precisely, there are exactly six such segments, which are:

\noindent\hfill%
\begin{tikzpicture}[scale=0.8,yscale=0.6]
\begin{scope}
\node at (0, 0) { \ph $\gO$ };
\node at (0, 1) { \ph $\gA$ };
\node at (1, 0.5) { ; };
\end{scope}
\begin{scope}[xshift=2cm]
\node at (0, 0) { \ph $\gA$ };
\node at (0, 1) { \ph $\gO$ };
\node at (1, 0.5) { ; };
\end{scope}
\begin{scope}[xshift=4cm]
\node at (0, 0) { \ph $\gO$ };
\node at (0, 1) { \ph $\gB$ };
\node at (1, 0.5) { ; };
\end{scope}
\begin{scope}[xshift=6cm]
\node at (0, 0) { \ph $\gB$ };
\node at (0, 1) { \ph $\gO$ };
\node at (1, 0.5) { ; };
\end{scope}
\begin{scope}[xshift=8cm]
\node at (0, 0) { \ph $\gO$ };
\node at (0, 1) { \ph $\gAB$ };
\node at (1, 0.5) { ; };
\end{scope}
\begin{scope}[xshift=10cm]
\node at (0, 0) { \ph $\gAB$ };
\node at (0, 1) { \ph $\gO$ };
\end{scope}
\end{tikzpicture}%
\hfill\null

\noindent
Their number of fragmentary combinatorial weights are $1$, $1$, $1$,
$1$, $2$ and
$2$ respectively. The upper bound of Theorem~\ref{thm:fibo}
is then correct for the four first fragments (with equality) but it 
is not for the two last ones.\\

\begin{rem}
Numerical experimentations show that almost---but not all---integers 
between $1$ and $\Fib_{\ell+2}$ can show up as the number of
fragmentary combinatorial weights of a fragment of length $\ell$.
Exceptions are large integers which are rather close to the 
upper bound $\Fib_{\ell+2}$; for example, when $\ell =
10$ (so that $\Fib_{\ell+2} = 144$), they are $113$, $114$, $118$, 
$120$, $126$, $127$, $130$ and all the integers in the interval
$[132, 142]$. We notice that $\Fib_{\ell+2} - 1$ always appear
for the following fragment:

\medskip

\begin{tikzpicture}[yscale=0.5,xscale=0.8]
\node[right] at (-4,1) { \ph if $\ell$ is even: };
\node at (0, 1) { \ph $\gO$ };
\node at (0, 0) { \ph $\gB$ };
\node at (1, 1) { \ph $\gA$ };
\node at (1, 0) { \ph $\gA$ };
\node at (2, 1) { \ph $\gB$ };
\node at (2, 0) { \ph $\gB$ };
\node at (3, 0.6) { \ph $\ldots$ };
\node at (4, 1) { \ph $\gA$ };
\node at (4, 0) { \ph $\gA$ };
\node at (5, 1) { \ph $\gB$ };
\node at (5, 0) { \ph $\gB$ };
\node at (6, 1) { \ph $\gA$ };
\node at (6, 0) { \ph $\gAB$ };
\end{tikzpicture}

\begin{tikzpicture}[yscale=0.5,xscale=0.8]
\node[right] at (-4,1) { \ph if $\ell$ is odd: };
\node at (0, 1) { \ph $\gO$ };
\node at (0, 0) { \ph $\gB$ };
\node at (1, 1) { \ph $\gA$ };
\node at (1, 0) { \ph $\gA$ };
\node at (2, 1) { \ph $\gB$ };
\node at (2, 0) { \ph $\gB$ };
\node at (3, 0.6) { \ph $\ldots$ };
\node at (4, 1) { \ph $\gA$ };
\node at (4, 0) { \ph $\gA$ };
\node at (5, 1) { \ph $\gB$ };
\node at (5, 0) { \ph $\gB$ };
\node at (6, 1) { \ph $\gA$ };
\node at (6, 0) { \ph $\gA$ };
\node at (7, 1) { \ph $\gAB$ };
\node at (7, 0) { \ph $\gB$ };
\end{tikzpicture}
\end{rem}

\subsection{From combinatorial weights to Serre weights}
 \label{ssec:recipe} \label{sec:combserre}

In this subsection, we explain how combinatorial weights defined in \S 
\ref{ssec:weightgene} are related to Serre weights. We give a precise 
statement of Theorem~\ref{thm'':main} of the introduction which 
provides an explicit bijection between the set of common weights of 
$\ttt$ and $\rhobar$ and the set of combinatorial weights of their 
gene.

We consider an absolutely irreducible representation 
$$\rhobar = \Ind_{G_{F'}}^{G_F}\big(\omega_{2f}^h \otimes \nr'(\theta)\big)$$
together with a tamely ramified Galois type $\ttt = \omega_f^\gamma 
\oplus \omega_f^{\gamma'}$. We let $\bX = (X_i)_{i \in \Z}$ denote 
the gene associated to $(h,\gamma,\gamma')$. As in \S \ref{ssec:serreweights},
we define the integers $c_0, \ldots, c_{f-1}$ in $\{0, \ldots, p{-}1\}$
by the relation
$\gamma' - \gamma \equiv \sum_{i=0}^{f-1} c_i p^i \pmod{q-1}$.
As in \S \ref{ssec:defgene}, we
introduce the $(2f)$-periodic sequence $(v_i)_{i \in \Z}$ defined
by the fact that $0 \leq v_i \leq p{-}1$ for all $i$ and the 
congruence \eqref{eq:vi}, namely:
$$h - (q{+}1)\gamma' \equiv 
p^{2f-1} v_0 + p^{2f-2} v_1 + \cdots + p v_{2f-2} + v_{2f-1}
\pmod{q^2 - 1}.$$
We define an equivalence
relation $\sim$ on $\{\gA, \gB, \gAB, \gO\}$ by $\gA \sim \gAB$ and 
$\gB \sim \gO$. For $i$ in $\Z$, we set
$\delta_i = 1$ if $X_i \sim X_{i+f}$ and $\delta_i = 0$ otherwise.
This defines a $f$-periodic sequence $(\delta_i)_{i \in \Z}$
with values in $\{0, 1\}$.

After these preparations, we are ready to explain the recipe to 
construct a Serre weight $\SW(\underline w)$ from the datum of 
a combinatorial weight $\underline w = (w_i)_{i \in \Z}$ in $\WW(\bX)$.
We continue to assume that $\bX$ is nondegenerate.
For this, we need to describe its parameters $s$ and $\underline r
= (r_0, \ldots, r_{f-1})$.
For $i$ in $\{0, \ldots, f{-}1\}$, we let $r_{f-i-1}$ be the integer 
defined in the table of Figure~\ref{fig:ri}.
\begin{figure}
\noindent\hfill%
\begin{tikzpicture}[xscale=4, yscale=0.8]
\draw (0,1)--(2,1);
\draw (-0.8,0)--(2,0);
\draw (-0.8,-1)--(2,-1);
\draw (-0.8,-2)--(2,-2);
\draw (-0.8,-3)--(2,-3);
\draw (-0.8,0)--(-0.8,-3);
\draw (0,1)--(0,-3);
\draw (1,1)--(1,-3);
\draw (2,1)--(2,-3);
\node at (-0.4,0.5) { $r_{f-1-i}$ };
\node at (0.5,0.5) { $w_{i-1} = \delta_{i-1}$ };
\node at (1.5,0.5) { $w_{i-1} \neq \delta_{i-1}$ };
\node at (-0.4,-0.5) { \ph $X_i = \gO$ };
\node at (-0.4,-1.5) { \ph $X_{i+f} = \gO$ };
\node at (-0.4,-2.5) { \ph otherwise };

\node at (0.5,-0.5) { \ph $v_i - 1 - w_i$ };
\node at (0.5,-1.5) { \ph $v_{i+f} - 1 - w_i$ };
\node at (0.5,-2.5) { \ph $w_i \cdot (p-1)$ };

\node at (1.5,-0.5) { \ph $p - 1 - v_i + w_i$ };
\node at (1.5,-1.5) { \ph $p - 1 - v_{i+f} + w_i$ };
\node at (1.5,-2.5) { \ph $p - 2 + w_i$ };

\end{tikzpicture}%
\hfill\null

\caption{Table giving the values of $r_{f-1-i}$}
\label{fig:ri}
\end{figure}
Defining $s$ is a bit more painful. Eq.~\eqref{eq:typetwist}
tells us that we have to choose $s$ with the property that:
$$2s \equiv \gamma + \gamma' - \sum_{i=0}^{f-1} r_i p^i
  \pmod{q-1}.$$
However, this leaves us with two possibilities 
between them we have to decide. In order to do so, we distinguish
between two cases:
\begin{myenumerate}[(1)]
\item
if there exists an integer $i_0$ such that $c_{i_0} \neq
\frac{p-1} 2$, we set $\varepsilon'_{i_0} = 0$ if $r_{i_0} \in
\{c_{i_0}, c_{i_0}-1\}$ and $\varepsilon'_{i_0} = 1$ otherwise;
\item
if $c_i = \frac{p-1} 2$ for all $i$,
we consider an integer $i_0$ such that $X_{f-1-i_0} = \gO$ and set
$\varepsilon'_{i_0} = 0$ if $w_{f-2-i_0} = \delta_{f-2-i_0}$ and
$\varepsilon'_{i_0} = 1$ otherwise.
\end{myenumerate}
Then, for this particular choice of $i_0$, we define:
$$s = \gamma' + 
      \frac 1 2 \left( \varepsilon'_{i_0} (q-1) + 
        \sum_{i=0}^{f-1} \lambda_{i+i_0} (c_i - r_i)  p^i\right)
  \in \Z/(q{-}1)\Z$$
where, by definition, $\lambda_j = q$ if $j < f{-}1$ and $\lambda_j 
= 1$ otherwise.

\begin{thm}
\label{thm:main}
If $\bX$ is nondegenerate, the construction
$\SW$ induces a bijection
$\WW(\bX) \stackrel\sim\longrightarrow \DD(\ttt,\rhobar)$.
\end{thm}

The proof of Theorem~\ref{thm:main} is entirely of combinatorial
nature; however, it is not easy and requires the introduction of
several new combinatorial objects. In order to avoid interrupting
the ongoing discussion, we prefer deriving just consequences of
\ref{thm:main} and postponing its proof until \S \ref{sec:proof}.

\begin{ex}
\label{ex:rtweights}
We take $p = 5$, $f = 7$, and:
\begin{align*}
h & = (3 + 4p + 3p^3 + 4p^4 + 4p^5 + 3p^6) (q+1) +{} \\
  & \hspace{10ex} (1 + 3p + p^2 + 4p^3 + 4p^4 + 4p^5 + p^6) \\
\gamma & = 3 + 4p + p^2 + p^3 + 4p^4 + 3p^5 + 3p^6 \\
\gamma' & = 3 + p + 3p^3 + 3p^4 + 4p^5 + 4p^6.
\end{align*}
A straightforward computation shows that the gene associated to
these data is the one considered in Example~\ref{ex:fragments}
(see also Example~\ref{ex:algogene}).
Its combinatorial weights are then those enumerated at the end of 
Example~\ref{ex:geneweights}. 
Besides, the $\delta_i$'s are given $(\delta_0, \ldots, \delta_7) =
(1, 1, 0, 0, 0, 1, 1)$.
Moreover, we find 
$(v_0, \ldots, v_{13}) = (4, 0, 1, 0, 0, 3, 0, 1, 0, 0, 4, 2, 1, 0)$.

Applying the recipe described above, we find that the Serre weights
associated to the combinatorial weights listed in 
Example~\ref{ex:geneweights} are respectively:

\medskip

\setlength{\parskip}{0pt}

\noindent\hfill%
$\begin{array}{r@{\hspace{0.5ex}}l@{\quad;\quad}r@{\hspace{0.5ex}}l}
(\tau_0 \circ \det{}^{77758}) & \otimes \Sym^{[4, 2, 1, 0, 4, 3, 3]} k_E^2 &
(\tau_0 \circ \det{}^{140262}) & \otimes \Sym^{[0, 1, 1, 0, 4, 3, 0]} k_E^2 \smallskip \\
(\tau_0 \circ \det{}^{77773}) & \otimes \Sym^{[4, 1, 0, 0, 4, 3, 3]} k_E^2 &
(\tau_0 \circ \det{}^{140272}) & \otimes \Sym^{[0, 2, 0, 0, 4, 3, 0]} k_E^2 \smallskip \\
\end{array}$%
\hfill\null

\noindent\hfill%
$\begin{array}{r@{\hspace{0.5ex}}l@{\quad;\quad}r@{\hspace{0.5ex}}l}
(\tau_0 \circ \det{}^{90258}) & \otimes \Sym^{[4, 2, 1, 0, 4, 0, 2]} k_E^2 &
(\tau_0 \circ \det{}^{137137}) & \otimes \Sym^{[0, 1, 1, 0, 4, 0, 1]} k_E^2 \smallskip \\
(\tau_0 \circ \det{}^{90273}) & \otimes \Sym^{[4, 1, 0, 0, 4, 0, 2]} k_E^2 &
(\tau_0 \circ \det{}^{137147}) & \otimes \Sym^{[0, 2, 0, 0, 4, 0, 1]} k_E^2 \smallskip \\
\end{array}$%
\hfill\null

\noindent\hfill%
$\begin{array}{r@{\hspace{0.5ex}}l@{\quad;\quad}r@{\hspace{0.5ex}}l}
(\tau_0 \circ \det{}^{77883}) & \otimes \Sym^{[4, 2, 1, 3, 3, 3, 3]} k_E^2 &
(\tau_0 \circ \det{}^{140387}) & \otimes \Sym^{[0, 1, 1, 3, 3, 3, 0]} k_E^2 \smallskip \\
(\tau_0 \circ \det{}^{77898}) & \otimes \Sym^{[4, 1, 0, 3, 3, 3, 3]} k_E^2 &
(\tau_0 \circ \det{}^{140397}) & \otimes \Sym^{[0, 2, 0, 3, 3, 3, 0]} k_E^2 \smallskip \\
\end{array}$%
\hfill\null

\noindent\hfill%
$\begin{array}{r@{\hspace{0.5ex}}l@{\quad;\quad}r@{\hspace{0.5ex}}l}
(\tau_0 \circ \det{}^{90383}) & \otimes \Sym^{[4, 2, 1, 3, 3, 0, 2]} k_E^2 &
(\tau_0 \circ \det{}^{137262}) & \otimes \Sym^{[0, 1, 1, 3, 3, 0, 1]} k_E^2 \smallskip \\
(\tau_0 \circ \det{}^{90398}) & \otimes \Sym^{[4, 1, 0, 3, 3, 0, 2]} k_E^2 &
(\tau_0 \circ \det{}^{137272}) & \otimes \Sym^{[0, 2, 0, 3, 3, 0, 1]} k_E^2 \smallskip \\
\end{array}$%
\hfill\null

\noindent\hfill%
$\begin{array}{r@{\hspace{0.5ex}}l@{\quad;\quad}r@{\hspace{0.5ex}}l}
(\tau_0 \circ \det{}^{77258}) & \otimes \Sym^{[4, 2, 1, 3, 0, 4, 3]} k_E^2 &
(\tau_0 \circ \det{}^{139762}) & \otimes \Sym^{[0, 1, 1, 3, 0, 4, 0]} k_E^2 \smallskip \\
(\tau_0 \circ \det{}^{77273}) & \otimes \Sym^{[4, 1, 0, 3, 0, 4, 3]} k_E^2 &
(\tau_0 \circ \det{}^{139772}) & \otimes \Sym^{[0, 2, 0, 3, 0, 4, 0]} k_E^2
\end{array}$%
\hfill\null

\setlength{\parskip}{3mm}

\noindent
where $\Sym^{[r_0, r_1, \ldots, r_6]} k_E^2$ is a shortcut for
$\bigotimes_{i=0}^6 (\Sym^{r_i} k_E^2)^{\tau_i}$.

Note that, in this particular example, $\DD(\rhobar)$ has 
cardinality $96$ and $\DD(\ttt)$ has cardinality $60$.
If one wants to calculate $\DD(\ttt, \rhobar)$, it is then 
much faster to use the previous techniques than to compute the 
intersection naively. We will make the latest assertion more 
rigourous in~\S \ref{sec:algo}.
\end{ex}

\section{Applications to deformations spaces}
\label{sec:End}

The aim of this section is to relate Theorem~\ref{thm:main}
to the results of \cite{CDM2}, with the objective to provide
new evidences supporting the conjectural description of the deformations 
rings $R^\psi(\ttt, \rhobar)$ we made in \cite[\S 5]{CDM2}.

Indeed, these conjectures suggest that the deformation ring 
$R^\psi(\ttt,\rhobar)$ is determined by the gene associated to 
$(\ttt,\rhobar)$ and, more precisely, that it can be obtained as a
completed tensor product of some rings associated to each fragment
of the gene.
A way to support this conjectural fragmentation is to consider its 
implications in terms of special fibres. Following the Breuil--Mézard 
conjecture, proved in \cite{CEGS} in the considered case, the number 
of irreducible components of the special fiber of a deformation ring 
$R^\psi(\ttt,\rhobar)$ is closely related to the number of common weights.
If our conjectures are true, the number of common weights should reflect
the fragmentation phenomena.  
Theorem \ref{thm:main} is a first enlighting result in this perspective
as it shows that $\DD(\ttt,\rhobar)$ splits canonically as a direct
product indexed by the fragments of the associated gene.

In this section, we do a thorough study of these phenomena.
Precisely, our intention is to go beyond the fragmentation of the 
gene and to relate (the cardinality of) $\DD(\ttt,\rhobar)$ to a more 
intrinsic object attached to the situation, namely the Kisin variety.
We will notably show that the Kisin variety equipped with some extra 
structures (which are its canonical embedding into $(\P^1)^f$ and its 
so-called shape stratification) entirely determines $\Card 
\DD(\ttt,\rhobar)$ and that the latter appears as a product as soon
as the Kisin variety itself splits a direct product
(Theorem~\ref{thm:weightsfactor}).
This will allow us to formulate refinements of the conjectures of
\cite{CDM2} (Conjecture~\ref{conj:rpsi}) and give convincing evidences
towards them.

\subsection{From genes to Kisin varieties}
\label{ssec:genesKV}

Given $\psi$, $\ttt$ and $\rhobar$ as in the previous sections, we 
recall from \S\ref{ssec:reptypes} that Kisin constructed in \cite{Ki3}
a ring $R^\psi(\ttt, \rhobar)$ parametrizing the potentially 
Barsotti--Tate deformations of $\rhobar$ with Hodge--Tate weights
$\{0,1\}$ at all embeddings, determinant $\psi$ and inertial type
$\ttt$.
The key ingredient in Kisin's argument is the construction of an 
auxiliary scheme $\GR^\psi(\ttt,\rhobar)$ parametrizing the 
Breuil--Kisin modules of type $\ttt$ inside the étale $\varphi$-module 
associated to $\rhobar$ by Fontaine and Wintenberger's theory\footnote{We 
refer to \cite[\S 2]{CDM2} for a more detailed exposition of Kisin's 
construction in the setting of this article.}.
This scheme comes equipped with a canonical morphism
$\GR^\psi(\ttt,\rhobar) \to \Spec R^{\psi}(\rhobar)$
whose schematic image is, by definition, the spectrum of the 
deformation ring $R^\psi(\ttt, \rhobar)$ we want to construct.
One then gets for free a morphism:
$$\GR^\psi(\ttt,\rhobar)\longrightarrow 
\Spec R^\psi(\ttt,\rhobar)$$
which is sometimes considered as a partial resolution of singularities.
Indeed, Kisin notably proves that the above morphism induces an 
isomorphism in generic fibre.
The special fibre of $\GR^\psi(\ttt,\rhobar)$, namely
$$\overline{\GR}^{\psi}(\ttt,\rhobar) = 
\Spec k_E\times_{\Spec R^\psi(\ttt,\rhobar)} 
\GR^\psi(\ttt,\rhobar).$$
is the so-called \emph{Kisin variety}; it is not smooth in general but 
its structure is much easier to describe than that of the special fibre 
of the deformation ring $R^\psi(\ttt,\rhobar)$ itself.

\subsubsection{Review of previous results}
\label{sssec:cdm2}

The first important result of~\cite{CDM2} is Theorem~2.2.1 which gives 
an entirely explicit description of the Kisin variety 
$\overline{\GR}^{\psi}(\ttt,\rhobar)$ in terms of the gene $\bX$ 
associated to $(\ttt,\rhobar)$.
Let us recall briefly what it is when $\bX = (X_i)_{i \in \Z}$ 
is nondegenerate (which is the only case we will work with later on).

By definition, we say that the letter $\gA$ (resp. $\gB$) is 
\emph{dominant} at position $i$ if it appears more often that $\gB$ 
(resp. $\gA$) in $\{X_i, X_{i+f}\}$ or, in case of equality, if $\gA$ 
(resp. $\gB$) is dominant at position $i{+}1$. The fact that $\bX$ 
is nondegenerate ensures that the definition is not circular
and, consequently, that each position has a well-defined dominant
letter; we will denote it by $\dom_i(\bX)$ in what follows.
For $i$ in $\Z$, we set $\lambda_i = 1$ if $X_i = \dom_i(\bX)$
and $\lambda_i = 0$ otherwise. Then,
the variety $\overline{\GR}^{\psi}(\ttt,\rhobar)$ is
isomorphic to the closed subscheme of $(\P^1)^f$, with projective
coordinates $[x_i\:{:}\:x_{i+f}]$ on the $i$-th factor ($0 \leq i < f$),
defined by the following equations: \smallskip \\
\null\hspace{1em}$\bullet$
for $0 \leq i < 2f$, 
if $X_i = \gO$, then $x_i = 0$; \smallskip \\
\null\hspace{1em}$\bullet$
for $0 \leq i < f$, 
if $\dom_i(\bX) = \dom_{i+1}(\bX)$, then
$\lambda_i x_i x_{i+1+f} = \lambda_{i+f} x_{i+f} x_{i+1}$.

\noindent
Besides, the Kisin variety is equipped with a so-called \emph{shape 
stratification} which can be read off on the gene too (see
\cite[Proposition~5.2.5]{CDM2}). 
This stratification is materialized by the datum of a shape function
in the sense of the following definition.

\begin{definit}
Let $\mathcal V$ be a subvariety of $(\P^1)^f$.
A \emph{shape function} of $\mathcal V$ is a upper
semi-continuous function $g : \mathcal V(\bar \F_p) \to \{\I, \II\}^f$ 
where $\I$ and $\II$ are two new symbols subject to the inequality 
$\I < \II$.

\vspace{-\parskip}

We say that $\mathcal V$ is \emph{shape-stratified} if it is equipped 
with a shape function.
\end{definit}

In what follows, we use the notation $\GRs(\bX)$ to denote the 
Kisin variety associated to $\bX$, \emph{viewed as a closed subscheme 
of $(\P^1)^f$ and equipped with its shape function}. In particular, 
equalities between various $\GRs(\bX)$ will always mean equalities 
inside $(\P^1)^f$ and equalities of the shape functions.

For a fixed gene $\bX$,
let $S$ be the set of indices $i$ in $\{0, \ldots, f{-}1\}$ for which 
the $i$-th projection map $\pr_i : (\P^1)^f \to \P^1$ is constant 
(necessarily 
equal to $[0\:{:}\:1]$ or $[1\:{:}\:0]$) on $\GRs(\bX)$. Since $\bX$
is nondegenerate, it is clear that $S$ is nonempty. Without loss of 
generality, one may further assume that $0 \in S$. Write $S = \{i_0, 
\ldots, i_{r-1}\}$ with $0 = i_0 < i_1 < \cdots < i_{r-1}$ and set $i_r 
= f$. By \cite[\S 5.2.3]{CDM2}, $\GRs(\bX)$ splits as a direct product: 
\begin{equation}
\label{eq:prodGR}
\GRs(\bX) = \Vs_0 \times \Vs_1 \times \cdots \times \Vs_{r-1}
\end{equation} 
where $\Vs_j$ is a shape-stratified closed subscheme of 
$(\P^1)^{i_{j+1}-i_j}$ corresponding to the portion of the gene located 
between the columns $i_j$ (included) and $i_{j+1}$ (excluded). We 
insist on the fact that \eqref{eq:prodGR} respects the stratification 
and the embedding into $(\P^1)^f$.

After this result, we conjectured (under mild assumptions on the 
inertial type) that the generic fibre $D^{\psi}(\ttt,\rhobar)$ (viewed 
as a rigid space) of $R^\psi(\ttt,\rhobar)$ splits as a direct product
$$D^{\psi}(\ttt,\rhobar)= D(\Vs_1) \times \cdots \times D(\Vs_{r-1})$$
where $D(\Vs_j)$ is a rigid space which depends only on $\Vs_j$ (see 
\cite[Conjecture~5.2.7]{CDM2})
and we gave a geometrical construction (in terms of blow-ups and formal
completions) of a candidate for $D(\Vs_j)$.
Unfortunately, regarding the deformation ring $R^\psi(\ttt, \rhobar)$ 
itself, we cannot expect it to be simply the ring of power-bounded 
functions on $D^\psi(\ttt,\rhobar)$ in general.
Indeed, in \cite[\S 5.3.2]{CDM2}, we exhibited an example where, 
conjecturally:
$$\textstyle 
D^{\psi}(\ttt,\rhobar) \simeq
\text{Spm}\Big(\frac{\oE[[T_1,T_2,U, V]]}{UV + p^2}[1/p]\Big)$$
but where the set of Serre weights $\DD(\ttt,\rhobar)$ has 
cardinality~$3$. By the Breuil--Mézard conjecture, we cannot then
have $R^\psi(\ttt,\rhobar) = \frac{\oE[[T_1,T_2,U, V]]}{UV + p^2}$
in this case.

Nevertheless, we propose the following weaker conjecture which
looks plausible.

\begin{conj}~
\begin{myenumerate}[(1)]
\item We have:
$$R^\psi(\ttt,\rhobar) \simeq
R(\Vs_0) \:\hat\otimes_{\oE}\: R(\Vs_1) \:\hat\otimes_{\oE}\: \cdots 
\:\hat\otimes_{\oE}\: R(\Vs_{r-1})$$
where $R(\Vs_j)$ is a local noetherian complete $\oE$-algebra depending
only on $\Vs_j$.
\item
The deformation ring $R^\psi(\ttt,\rhobar)$ is the ring of 
power-bounded functions on $D^\psi(\ttt,\rhobar)$
if there is no index $i$ with $X_i = X_{i+f} = \dom_{i+1}(\bX)$.
\end{myenumerate}
\label{conj:rpsi}
\end{conj}

In what follows, we establish several results relating 
$\DD(\ttt,\rhobar)$ and $\GRs(\bX)$ which, combined with the 
Breuil--Mézard conjecture, will eventually give some evidences towards 
the above conjecture (see Theorems~\ref{thm:weightsfactor},
\ref{thm:monotonyfragment} and \ref{thm:cross}).

\subsubsection{Fragmentation}
\label{sssec:fragmentation}

In order to define the factor varieties $\Vs_j$, we have divided
the gene $\bX$ into parts.
Of course, this division is related to the fragmentation
introduced in~\S \ref{ssec:withO} but one needs to be careful that
they are not exactly the same.
In order to clarify the relationships between these two splittings, 
we remember that, by definition, $[x_i\:{:}\:x_{i+f}]$ is constantly 
equal to $[0\:{:}\:1]$ on the Kisin variety as soon as $X_i = \gO$.
Hence each fragment in the sense of Definition~\ref{defi:fragment}
corresponds to a slice of indices~$j$ but, in full generality, this
slice can have cardinality larger than $1$.

In order to study in more details this phenomenon, it is convenient to 
associate a shape-stratified Kisin variety $\GRs(\underline F) \subset 
(\P^1)^\ell$ to any fragment $\underline F$ of length~$\ell$. This can 
be done simply by copying the rules we detailed previously in the case 
of genes. Concretely, if $[x_i\:{:}\:y_i]$ are the coordinates on the
$i$-th factor, $\GRs(\underline F)$ is defined by the equations:
\smallskip\\
\null\hspace{1em}$\bullet$
$x_0 = 0$ if $F_0^\up = \gO$ (resp. $y_0 = 0$ if $F_0^\down = \gO$),
\smallskip\\
\null\hspace{1em}$\bullet$
for $0 \leq i < \ell-1$, 
if $\dom_i(\underline F) = \dom_{i+1}(\underline F)$,
then
$\lambda_i x_i y_{i+1} = \mu_i y_i x_{i+1}$.
\smallskip\\
where $\lambda_i$ (resp. $\mu_i$) is $1$ if $F_i^\up = 
\dom_i(\underline F)$ (resp. $F_i^\down = 
\dom_i(\underline F)$) and $0$ otherwise.

The Kisin variety of a gene is obviously equal to the product of
the Kisin varieties of its fragments.
It may happen that the Kisin variety of a given
fragment splits further as a product of smaller Kisin varieties. For 
instance, if $\underline F$ is a fragment of the form:

\hfill%
\begin{tikzpicture}[scale=0.8]
\begin{scope}
\node[left,scale=0.7] at (-1, 1.9) { \ph $i$: };
\node[scale=0.7] at (0, 1.9) { \ph $0$ };
\node[scale=0.7] at (1, 1.9) { \ph $1$ };
\node[scale=0.7] at (3, 1.9) { \ph $n{-}2$ };
\node[scale=0.7] at (4, 1.9) { \ph $n{-}1$ };
\node[scale=0.7] at (6, 1.9) { \ph $\ell{-}1$ };
\node[left] at (-1, 0) { \ph $F_i^\down$: };
\node[left] at (-1, 1) { \ph $F_i^\up$: };
\node at (0, 0) { \ph $\gA$ };
\node at (0, 1) { \ph $\gO$ };
\node at (1, 0) { \ph $\gA$ };
\node at (1, 1) { \ph $\star$ };
\node at (3, 0) { \ph $\gA$ };
\node at (4, 0) { \ph $\star$ };
\node at (6, 1) { \ph $\star$ };
\node at (6, 0) { \ph $\star$ };
\draw[thick,dotted] (1.5,0)--(2.5,0);
\draw[thick,dotted] (4.5,0)--(5.5,0);
\draw[thick,dotted] (1.5,1)--(5.5,1);
\end{scope}
\end{tikzpicture}%
\hfill\null

\noindent
where the letter $\gA$ is dominant in all positions between $0$ and 
$n{-}1$, the rules giving the equations of the Kisin variety imply that 
$\pr_0, \ldots, pr_{n-1} : \GRs(\underline F) \to \P^1$
are all constant equal to $[0\:{:}\:1]$.
The next lemma shows that all the examples are of this type.

\begin{lem}
\label{lem:prfrag}
Let $\underline F$ be a fragment of length $\ell > 1$ such that
$F_0^\up = \gO$ and $F_0^\down = \gA$.
Let $n$ be the largest integer for which $F_i^\down = \gA$ for $i < n-1$ 
and $\gA$ is dominant in all positions $i < n$.
Then, the map $\pr_i : \GRs(\underline F) \to \P^1$ is constant 
if $i < n$ and surjective if $i \geq n$.
\end{lem}

\begin{proof}
We have already seen that $\pr_i$ is constant equal to 
$[0:1]$ for $i < n$.
Let us now prove that $\pr_n$ is surjective.
For this, the key is to observe that if we are given $[x_i:y_i] \in 
\P^1$, there always exists $[x_{i+1}:y_{i+1}] \in \P^1$ satisfying the 
equation $\lambda_i x_i y_{i+1} = \mu_i y_i x_{i+1}$. 

Now we argue as follows.
If $\gB$ is dominant at position~$n$, then
there is no equation relating $[x_{n-1}:y_{n-1}]$ and
$[x_n:y_n]$. We can then choose $[x_n:y_n]$ arbitrarily and complete
the family into an actual point of the Kisin variety by using
repeatedly the observation we made above. 
On the contrary, if $\gA$ is dominant at position~$n$, it follows
from the definition of $n$ that $F_{n-1}^\down = \gB$ and then that
$\mu_{n-1} = 0$. 
The equation $\lambda_{n-1} x_{n-1} y_n = \mu_{n-1} y_{n-1} x_n$ is 
then trivially satisfied since $x_{n-1} = 0$ as well. Therefore, as in the
first case, we can choose $[x_n:y_n]$ arbitrarily and conclude as
before.

The surjectivity of $\pr_i$ for $i > n$ follows by induction on~$i$
using a similar argument.
\end{proof}

Lemma~\ref{lem:prfrag} tells us that we have a decomposition of the
form:
$$\GRs(\underline F) = \big\{[0\:{:}\:1]\big\}^n \times \Vs$$
where $\Vs$ is some shape-stratified subvariety of $(\P^1)^{\ell-n}$ and
the shape of the prefactor $\big\{[0\:{:}\:1]\big\}^n$ is
$(\I, \ldots, \I)$.
It is actually possible to go further and make the factor $\Vs$
explicit.
Let us first consider the case where $F_{n-1}^\down = \gA$ which is
the easiest one. Under this assumption, one can check that
$\Vs = \GRs(\underline F')$ where $\underline F'$ is the fragment:

\hfill%
\begin{tikzpicture}[yscale=0.7,xscale=0.9]
\begin{scope}
\node at (-0.1, 0) { \ph $\gA$ };
\node at (-0.1, 1) { \ph $\gO$ };
\node at (0.8, 0) { \ph $F_n^\down$ };
\node at (0.8, 1) { \ph $F_n^\up$ };
\node at (2, 0) { \ph $F_{n+1}^\down$ };
\node at (2, 1) { \ph $F_{n+1}^\up$ };
\node at (4, 0) { \ph $F_{\ell-1}^\down$ };
\node at (4, 1) { \ph $F_{\ell-1}^\up$ };
\draw[thick,dotted] (2.7,0)--(3.3,0);
\draw[thick,dotted] (2.7,1)--(3.3,1);
\end{scope}
\end{tikzpicture}%
\hfill\null

When $F_{n-1}^\down = \gB$, it is still true that $\Vs =
\GRs(\underline F')$ as subvarieties of $(\P^1)^{\ell-n}$ but the
shape stratification does not agree.
Actually, in this situation, $\Vs$ does not appear as the Kisin
variety of a smaller fragment. However we can write:
$$\GRs(\underline F) = \big\{[0\:{:}\:1]\big\}^{n-1} \times 
\GRs(\underline F'')$$
where $\underline F''$ is the following fragment:

\hfill%
\begin{tikzpicture}[yscale=0.7,xscale=0.9]
\begin{scope}
\node at (-1, 0) { \ph $\gA$ };
\node at (-1, 1) { \ph $\gO$ };
\node at (-0.1, 0) { \ph $\gB$ };
\node at (-0.1, 1) { \ph $\gA$ };
\node at (0.8, 0) { \ph $F_n^\down$ };
\node at (0.8, 1) { \ph $F_n^\up$ };
\node at (2, 0) { \ph $F_{n+1}^\down$ };
\node at (2, 1) { \ph $F_{n+1}^\up$ };
\node at (4, 0) { \ph $F_{\ell-1}^\down$ };
\node at (4, 1) { \ph $F_{\ell-1}^\up$ };
\draw[thick,dotted] (2.7,0)--(3.3,0);
\draw[thick,dotted] (2.7,1)--(3.3,1);
\end{scope}
\end{tikzpicture}%
\hfill\null

\noindent
(Notice that, in this case, we necessarily have $F_{n-1}^\up = \gA$
because otherwise $\gB$ would be dominant in position $n{-}1$, which
contradicts the definition of $n$.)

\begin{rem}
In what precedes, we have assumed that $F_0^\up = \gO$ and 
$F_0^\down = \gA$. Of course, there are similar statements where $\gA$
is remplaced by $\gB$ and where the roles of $F_0^\up$ and $F_0^\down$
are exchanged. All fragments of length at least $2$ are covered by 
these variants.
\end{rem}

\subsubsection{Reduced fragments}
\label{sssec:reduced}

After \cite[Theorem 2.2.1]{CDM2}, the Kisin variety of a pair
$(\ttt, \rhobar)$ is entirely determined by the gene. Conversely,
one may wonder if the Kisin variety determines the gene.
Lemma~\ref{lem:prfrag} shows that this too naive question has a 
negative answer; indeed, if $n$ is the integer defined
in the statement of this lemma, the values of $F_i^\up$ for $1 \leq i
< n$ do not have any influence on $\GRs(\underline F)$.
Besides, one easily checks that if $\underline F$ is any segment and
$\underline F'$ is the segment deduced from $\underline F$ by 
flipping the letters $\gA$ and $\gB$, we have $\GRs(\underline F) = 
\GRs(\underline F')$.
Nonetheless, beyond these two ``trivial'' obstructions, one can
prove injectivity results about $\GRs$.

\begin{definit}
\label{def:reduced}
Let $\underline F$ be a fragment of length $\ell$.\\
We say that $\underline F$ is \emph{top-reduced} if 
$F_0^\up = \gO$, $F_0^\down \in \{\gA, \gAB\}$ and:
\begin{myenumerate}[(i)]
\item either $\ell = 1$,
\item or $\ell > 1$ and $\gB$ is dominant at position $1$,
\item or $\ell > 2$, $F_1^\up = \gA$, $F_1^\down = \gB$
and $\gA$ is dominant at position $2$.
\end{myenumerate}

\vspace{-\parskip}

\noindent
We say that $\underline F$ is \emph{bottom-reduced} if the 
fragment deduced from $\underline F$ by swapping its top row
and its bottom row is top-reduced.

\vspace{-\parskip}

\noindent
We say that $\underline F$ is \emph{reduced} if it is either
top-reduced or bottom-reduced.

If $\underline F$ is reduced, we set $\GRts(\underline F) =
\GRs(\underline F)$ in cases~(i) and~(ii) whereas, in case~(iii),
we define $\GRts(\underline F)$ by the equality
$\GRs(\underline F) = \big\{[0:1]\} \times \GRts(\underline F).$
\end{definit}

It is easy to check that any fragment $\underline F$ can be related 
to a reduced fragment $\underline F^\red$ by flipping letters and/or 
performing the transformation presented at the end of \S 
\ref{sssec:fragmentation}. We then have the relation:
$$\GRs(\underline F) = \big\{[0\:{:}\:1]\big\}^n \times 
\GRts(\underline F^\red)$$
where $n$ is the integer defined in Lemma~\ref{lem:prfrag}.

\begin{prop}
\label{prop:GRsinj}
The function $\GRts$ is injective on the set of reduced fragments.
\end{prop}

\begin{proof}
Let $\underline F$ be a reduced fragment of length $\ell > 1$.
It is easy to see on the Kisin variety $\GRs(\underline F)$ if 
$\underline F$ is top-reduced or bottom-reduced since the function
$\pr_1 : \GRs(\underline F) \to \P^1$ is constant equal to 
$[0:1]$ in the first case and constant equal to $[1:0]$ in the
second case.

We then assume that $F_0^\up = \gO$, $F_0^\down = \gA$
and first consider the case where $\dom_1(\underline F) = \gB$.
Let $i$ be in $\{1, \ldots, \ell{-}1\}$. From the proof of 
Lemma~\ref{lem:prfrag}, we deduce that, for $0 < i < \ell{-}1$, the 
map 
$\pi_i = (\pr_i, \pr_{i+1}) : \GRs(\underline F) \to \P^1 \times \P^1$
is surjective if and only if $\dom_i(\underline F) \neq \dom_{i+1}
(\underline F)$. 
Therefore, the values of $\dom_i(\underline F)$ can be reconstructed
from the datum of $\GRs(\underline F)$. When
$\dom_i(\underline F) \neq \dom_{i+1} (\underline F)$, this is enough
to reconstruct $F_i^\up$ and $F_i^\down$ because we need to have
$F_i^\up = F_i^\down = \dom_i(\underline F)$. On the contrary, when
$\dom_i(\underline F) = \dom_{i+1}(\underline F)$, we note that the 
image of $\pi_i$ is:
\begin{myitemize}
\item
the subscheme of equation $xt = 0$ if 
$F_i^\up = \dom_i(\underline F)$ and $F_i^\down \neq \dom_i(\underline F)$,
\item
the subscheme of equation $yz = 0$ if 
$F_i^\up \neq \dom_i(\underline F)$ and $F_i^\down = \dom_i(\underline F)$,
\item
the subscheme of equation $xt = yz$ if 
$F_i^\up = F_i^\down = \dom_i(\underline F)$,
\end{myitemize}
where $[x:y]$ are the coordinates on the first copy of $\P^1$ and
$[z:t]$ are the coordinates on the second copy. (We notice that the
case $F_i^\up = F_i^\down \neq \dom_i(\underline F)$ cannot occur.)
As a consequence, one can reconstruct $F_i^\up$ and $F_i^\down$ in
this case as well.

In a similar fashion, we prove that $\underline F$ is determined by 
$\GRs(\underline F)$ when $\ell > 2$, $F_1^\up = \gA$, $F_1^\down = \gB$ 
and $\gA$ is dominant at position $2$. It then remains to check that we
can see on $\GRs(\underline F)$ if we are in the case~(i), (ii) or~(iii)
of Definition~\ref{def:reduced}. Recognizing case~(i) is trivial.
Finally, distinguishing between the cases~(ii) and~(iii) can be done 
by looking at the shape stratification (see \cite[Proposition~5.2.5]{CDM2}).
\end{proof}

\begin{cor}
\label{cor:reduced}
Given a fragment $\underline F$, there exists a unique nonnegative 
integer~$n$ and a unique reduced fragment $\underline F^\red$ such 
that 
$$\GRs(\underline F) = \big\{[0\:{:}\:1]\big\}^n \times 
\GRts(\underline F^\red)$$
where the shape of the prefactor $\big\{[0\:{:}\:1]\big\}^n$ is
$(\I, \ldots, \I)$.
\end{cor}

\begin{proof}
Existence have been already discussed.
Unicity follows from Proposition~\ref{prop:GRsinj} after having noticed 
that there is a unique way to obtain $n$ and $\underline F'$ for 
$\underline F$ by flipping letters and performing the 
transformation of \S \ref{sssec:fragmentation}.
\end{proof}

Another byproduct of the proof of Proposition~\ref{prop:GRsinj} is the 
following proposition which underlines the particular interest of the 
factorisation given by Eq.~\eqref{eq:prodGR}.

\begin{prop}
\label{prop:maxdecomp}
Let $\bX$ be a gene. If $\GRs(\bX)$ splits as a product
$\As \times \Bs$, then there exists an integer $k$ such that
\begin{align*}
\As & = \Vs_0 \times \Vs_1 \times \cdots \Vs_{k-1} \\
\text{and}\quad
\Bs & = \Vs_k \times \Vs_{k+1} \times \cdots \Vs_{r-1}
\end{align*}
where the $\Vs_j$'s are those defined by Eq.~\eqref{eq:prodGR}.
\end{prop}

\begin{proof}
It is enough to prove that if $\underline F$ is a reduced
fragment, then the variety $\GRts(\underline F)$ cannot be
decomposed as a product. For this, using the notations of the
proof of Proposition~\ref{prop:GRsinj}, it suffices to show that
the image of each function $\pi_i$ cannot be written as a product
of shape-stratified varieties (the shape stratification descends to the
image of $\pi_i$ thanks to \cite[Proposition~5.2.5]{CDM2}).
When $\dom_i(\underline F) = \dom_{i+1}(\underline F)$, this
directly follows for the explicit equations of the image of $\pi_i$
we have obtained.
On the contrary, when $\dom_i(\underline F) = \dom_{i+1}(\underline F)$,
we use \cite[Proposition~5.2.5]{CDM2} which teaches us that the
subvariety of $(\P^1)^2$ of equation $xz = yt$ is a stratum of the
shape stratification. Hence $\pi_i$ cannot be decomposed as a 
product of shape-stratified varieties of dimension $1$ and the proof 
of Proposition~\ref{prop:maxdecomp} is complete.
\end{proof}

\subsection{Weights and Kisin variety}

In this subsection, we study the relationships between the Kisin variety 
of a gene $\bX$ and its set of combinatorial weights. We are
particularly interested in the cardinality of $\WW(\bX)$. Indeed, 
this numerical invariant has a huge arithmetical meaning since the 
Breuil--Mézard conjecture relates it directly to the number of 
irreducible components of the special fibres of deformation rings.

\subsubsection{Behaviour of weights under gene transformations}

In \S \ref{ssec:genesKV}, we have seen a couple of transformations of a 
fragment $\underline F$ that preserve the associated Kisin varieties. 
To begin with, we would like to study how these transformations affect the set of 
fragmentary combinatorial weights $\WW(\underline F)$ of $\underline F$ 
and its cardinality.
Eventually, we aim at proving the following theorem which, in some 
sense, can be seen as a numerical version of the first part of 
Conjecture~\ref{conj:rpsi} and then provides some support to it.

\begin{thm}
\label{thm:weightsfactor}
Let $\bX$ be a gene containing at least an instance of the
letter~$\gO$. Then the cardinality of $\WW(\bX)$ depends only
on the Kisin variety $\GRs(\bX)$. More precisely, if 
$$\GRs(\bX) = \Vs_0 \times \Vs_1 \times \cdots \times \Vs_{r-1}$$
is the canonical decomposition of $\GRs(\bX)$ given by
Eq.~\eqref{eq:prodGR}, we have:
$$\Card \WW(\bX) = c(\Vs_0) \cdot c(\Vs_1)\:\cdots\:c(\Vs_{r-1})$$
where $c(\Vs_j)$ is an integer depending only on $\Vs_j$ (as
suggested by the notation).
\end{thm}

\begin{rem}
\label{rem:weightsfactor}
After Proposition~\ref{prop:maxdecomp}, we see that
Theorem~\ref{thm:weightsfactor} can be rephrased as follows.
Let $\SEKV_0$ be the set of all shape-stratified subvarieties of 
$(\P^1)^n$
(for varying~$n$) that can be written as a product of the form
$\GRts(\underline F_0) \times \cdots \times \GRts(\underline F_{r-1})$
for some reduced fragments $\underline F_0, \ldots, \underline
F_{r-1}$.
Theorem~\ref{thm:weightsfactor} then tells that there exists a
\emph{multiplicative} function $c : \SEKV_0 \to \N$ such that
$\Card \WW(\bX) = c(\hspace{0.2ex}\GRs(\bX)\big)$ for all
gene $\bX$ with $X_0 = \gO$.
Moreover, it will follow from Theorem~\ref{thm:monotonyfragment} of 
\S \ref{sssec:monotony} below that $c$ is nondecreasing.
Furthermore, we note that the function $c$ can be prolonged to the
set $\SEKV$ defined in the introduction while preserving the previous
properties; for instance, for $\Ws \in \SEKV$, one can define $c(\Ws)$ 
as the maximum of the $c(\Vs)$'s for $\Vs \in \SEKV_0$ with $\Vs\subset
\Ws$.
\end{rem}

The easiest transformation considered in \S \ref{ssec:genesKV}
consists in flipping the letters $\gA$ and~$\gB$. 
If $\underline F$ is a fragment, we denote by $\underline F^\tau$ 
the fragment obtained by performing this transformation. We have already 
said that $\GRs(\underline F) = \GRs(\underline F^\tau)$. Similarly,
it turns out that the sets $\WW(\underline F)$ and $\WW(\underline
F^\tau)$ are closely related. Precisely, if $\underline w
= (w_0, \ldots, w_{\ell-1}) \in \{0, 1\}^\ell$ is a fragmentary
combinatorial weight, we set
$\underline w^\tau = (1{-}w_0, w_1, \ldots, w_{\ell-1})$.
Similarly, if $W$ is a subset of $\{0, 1\}^\ell$, we define $W^\tau$
as the subset of $\{0, 1\}^\ell$ obtained by applying the
transformation $\underline w \mapsto \underline w^\tau$ to each
element of $W$.

\begin{lem}
\label{lem:flipletters}
For any fragment $\underline F$ of length $\ell > 1$,
we have $\WW(\underline F^\tau) = \WW(\underline F)^\tau$.
\end{lem}

\begin{proof}
We directly check from the definition (see 
Definition~\ref{def:fragmentweight}) that the set of associated 
fragmentary combinatorial weights is unaffected if we exchange the
top row and the bottom row of a fragment. Hence, we may assume without
loss of generality that $\underline F$ is top-reduced, \emph{i.e.}
$F_0^\up = \gO$ and $F_0^\down = \gA$.
Coming back to the definition of $\WW(\underline F)$, we find:

\medskip

$\begin{array}{cr@{\hspace{0.5ex}}ll}
\bullet &
W_1^{(\gb,\gb)}(\underline F) & = \big\{(0,0)\big\} \smallskip \\
\bullet &
W_1^{(\ga,\gb)}(\underline F) 
  & = \big\{(1,0)\big\} & \text{if } F_1^\up \sim \gA \smallskip \\
 && = \big\{(0,0)\big\} & \text{otherwise} \smallskip \\
\bullet &
W_1^{(\gb,\ga)}(\underline F) 
  & = \big\{(0,0), (1,0)\big\} & \text{if } F_1^\up \sim \gA \smallskip \\
 && = \emptyset & \text{otherwise} \smallskip \\
\end{array}$

\noindent
and similarly:

\medskip

$\begin{array}{cr@{\hspace{0.5ex}}ll}
\bullet &
W_1^{(\gb,\gb)}(\underline F^\tau) & = \big\{(1,0)\big\} \smallskip \\
\bullet &
W_1^{(\ga,\gb)}(\underline F^\tau) 
  & = \big\{(0,0)\big\} & \text{if } F_1^\up \sim \gA \smallskip \\
 && = \big\{(1,0)\big\} & \text{otherwise} \smallskip \\
\bullet &
W_1^{(\gb,\ga)}(\underline F^\tau) 
  & = \big\{(0,0), (1,0)\big\} & \text{if } F_1^\up \sim \gA \smallskip \\
 && = \emptyset & \text{otherwise.} \smallskip \\
\end{array}$

\noindent
Therefore, one checks that $W_1^{\square}(\underline F^\tau)
= W_1^{\square}(\underline F)^\tau$ for any $\square \in
\{(\ga,\gb), (\gb,\ga), (\gb,\gb)\}$. By induction, this equality
extends to all indices $i$ between $1$ and $\ell{-}1$ and the
proposition follows.
\end{proof}

\begin{rem}
If $\underline F$ is a fragment of length $1$, one can verify by hand
that $\WW(\underline F^\tau) = \WW(\underline F)$, \emph{i.e.} the
weights are not twisted in this case.
In particular, in all cases, we conclude that $\WW(\underline F)$
and $\WW(\underline F^\tau)$ have the same cardinality.
\end{rem}

We now focus on the transformation reported at the end of \S
\ref{sssec:fragmentation}.

\begin{lem}
\label{lem:reduce}
Let $\underline F$ be a fragment of length $\ell > 1$.
We assume that $F_0^\up = \gO$ and $F_0^\down = F_1^\down = \cdots
= F_{n-1}^\down = \gA$ for some integer $n \leq \ell$.
Let $\underline F'$ be the following truncated fragment:

\hfill%
\begin{tikzpicture}[yscale=0.7,xscale=0.9]
\begin{scope}
\node at (-0.1, 0) { \ph $\gA$ };
\node at (-0.1, 1) { \ph $\gO$ };
\node at (0.8, 0) { \ph $F_n^\down$ };
\node at (0.8, 1) { \ph $F_n^\up$ };
\node at (2, 0) { \ph $F_{n+1}^\down$ };
\node at (2, 1) { \ph $F_{n+1}^\up$ };
\node at (4, 0) { \ph $F_{\ell-1}^\down$ };
\node at (4, 1) { \ph $F_{\ell-1}^\up$ };
\draw[thick,dotted] (2.7,0)--(3.3,0);
\draw[thick,dotted] (2.7,1)--(3.3,1);
\end{scope}
\end{tikzpicture}%
\hfill\null

\noindent
Then $\Card \WW(\underline F) = \Card \WW(\underline F')$.
\end{lem}

\begin{proof}
For $\square \in \{(\ga,\gb), (\gb,\ga), (\gb,\gb)\}$,
set $c_i^\square = \Card W_i^\square(\underline F)$ 
and ${c'}_i^\square = \Card W_i^\square(\underline F')$.
Using the recursive formulas of \S \ref{ssec:count}, one 
checks by induction on $i$ that 
$c_i^{(\ga,\gb)} = 1$, $c_i^{(\gb,\ga)} = 0$, $c_i^{(\gb,\gb)} = 1$
for $i < n$ and then, that:
$$c_n^{(\ga,\gb)} = {c'}_1^{(\ga,\gb)} = 1, \quad
  c_n^{(\gb,\gb)} = {c'}_1^{(\gb,\gb)}, \quad
  c_n^{(\gb,\gb)} = {c'}_1^{(\gb,\gb)} = 1.$$
(Note that the common value of $c_n^{(\gb,\gb)}$ and ${c'}_1^{(\gb,\gb)}$
can be either $0$ or $2$ depending on $F_n^\down$.)
By a second induction, we finally find that
$$c_i^{(\ga,\gb)} = {c'}_{i-n+1}^{(\ga,\gb)}, \quad
  c_i^{(\gb,\gb)} = {c'}_{i-n+1}^{(\gb,\gb)}, \quad
  c_i^{(\gb,\gb)} = {c'}_{i-n+1}^{(\gb,\gb)}$$
for all $i < \ell$.
After this, the lemma follows from Corollary~\ref{cor:countwithO}.
\end{proof}

\begin{proof}[Proof of Theorem~\ref{thm:weightsfactor}]
If $\Vs$ is a shape-stratified subvariety of $(\P^1)^\ell$ of the form
$\GRts(\underline F)$ for some reduced fragment $\underline F$, let
us set $c(\Vs) =  \Card \WW(\underline F)$. 
Proposition~\ref{prop:GRsinj} ensures that this definition is
nonambiguous. 
Let $\bullet$ be the Kisin variety of the unique reduced fragment of 
length $1$ (namely $F_0^\up = \gO$ and $F_0^\down = \gA$). Concretely
it is the subvariety $\big\{[0\:{:}\:1]\big\}$ of $\P^1$ with shape 
function equal to $\I$. Moreover we obtain from the definition that
$c(\bullet) = 1$.

We fix a gene $\bX$ with $X_0 = \gO$.
The Kisin variety associated to $\bX$ decomposes as:
$$\GRs(\bX) = \prod_{\underline F} \, \GRs(\underline F)$$
where the product runs over all the fragments of $\bX$. Besides,
by Corollary~\ref{cor:reduced}, for each such fragment $\underline
F$, we have a second decomposition:
$$\GRs(\underline F) = \bullet^n \times \GRts(\underline F^\red)$$
where $n$ is an integer and $\underline F^\red$ is a reduced fragment
canonically attached to $\underline F$.
Putting together all the above decompositions, we end up with the 
decomposition of Eq.~\eqref{eq:prodGR}.
Moreover, we deduce from Lemmas~\ref{lem:flipletters} 
and~\ref{lem:reduce} that $\underline F$ and $\underline F^\red$ 
share the same number of weights. We can then write:
$$\Card \WW(\bX) 
 = \prod_{\underline F} \Card \WW(\underline F)
 = \prod_{\underline F} \Card \WW(\underline F^\red)
 = \prod_{\underline F} c\big(\GRts(\underline F^\red)\big).$$
As $c(\bullet) = 1$, we obtain the product formula of
Theorem~\ref{thm:weightsfactor}.
\end{proof}

\subsubsection{A monotony result}
\label{sssec:monotony}

In the construction of \cite[\S 5.4]{CDM2},
the candidate $D(\Vs)$ for the rigid space associated to a Kisin 
variety $\Vs$ is obtained by taking the 
formal completing of a certain space along $\Vs$ (which naturally
appears as a subscheme of the special fibre).
Hence, one may expect to some extent that the deformation ring
$R^\psi(\ttt,\rhobar)$ becomes more intricated as the underlying
Kisin variety gets larger.
Since, by the Breuil--Mézard conjecture, the number of common
Serre weights of $\ttt$ and $\rhobar$ is a direct measure of the
complexity of $R^\psi(\ttt,\rhobar)$, one may expect that this
number increases when the Kisin variery gets larger.
It turns out that this rough intuition is indeed correct as shown by 
the next theorem.

\vspace{\parskip}

\begin{thm}
\begin{myenumerate}[(1)]
\item
Let $\underline F$ and $\underline F'$ be two fragments 
such that $\GRs(\underline F) \subset \GRs(\underline F')$. 
Then $\Card \WW(\underline F) \leq \Card \WW(\underline F')$.

\item
Let $\underline F_1$, $\underline F_2$ and $\underline F'$ be 
three fragments such that
$\GRs(\underline F_1) \times \GRs(\underline F_2) \subset 
\GRs(\underline F')$. Then
$\Card \WW(\underline F_1) \cdot
\Card \WW(\underline F_2) \leq \Card \WW(\underline F')$.
\end{myenumerate}
\label{thm:monotonyfragment}
\end{thm}

In order to prove the theorem, we need some preparatory results.
In what follows,
for $\square \in \{(\ga,\gb), (\gb,\ga), (\gb,\gb)\}$, we will 
denote the cardinality of $\Card W_i^\square(\underline F)$ by
$c_i^\square$ and that of $\Card W_i^\square(\underline F')$ by
${c'}_i^\square$ as we already did in the proof of Lemma~\ref{lem:reduce}.

\begin{lem}
\label{lem:ineqci}
Let $\underline F$ be a fragment of length $\ell$.

\begin{myenumerate}[(1)]
\item
For all $i$ in $\{0, \ldots, \ell{-}1\}$, we have
$\big|c_i^{(\ga,\gb)} - c_i^{(\gb,\ga)}\big| \leq c_i^{(\gb,\gb)}$.
\item
We assume that there exists $s$ and $t$ with $0 < s < t \leq \ell$ and
$F_s^\up = \gA$, $F_s^\down = \gB$, $\{F_i^\up, F_i^\down\} = \{\gA,
\gB\}$ for $s < i < t$ and $F_t^\up = F_t^\down \in \{\gA,\gAB\}$.
Then there exist nonnegative integers $n$ and $m$ such that:
\begin{align*}
c_t^{(\ga,\gb)} & = c_s^{(\ga,\gb)} + n \cdot c_s^{(\gb,\gb)}, \\
c_t^{(\gb,\ga)} & = c_s^{(\ga,\gb)} + m \cdot c_s^{(\gb,\gb)}, \\
\text{and} \quad
c_t^{(\gb,\gb)} & = c_s^{(\gb,\gb)}.
\end{align*}
\end{myenumerate}
\end{lem}

\begin{proof}
It is an easy checking using the recursive formulas of
\S \ref{ssec:count}.
\end{proof}

\begin{proof}[Proof of Theorem~\ref{thm:monotonyfragment}]
We first prove the statement~(1) of the theorem under the additional 
assumption that there exists an integer $s$ such that:
\begin{myitemize}
\item
for $0 \leq i < s{-}1$, we have
$F_i^\up = {F'_i}^\up$ and $F_i^\down = {F'_i}^\down$,
\item
$F_{s-1}^\up = \gA$, $F_{s-1}^\down = \gB$
and ${F'}^\up_{\hspace{-1ex}s-1} = \gA$,
${F'}^\down_{\hspace{-1ex}s-1} = \gA$,
\item $\gA$ is dominant in $\underline F$ at position $s$,
\item
for $s \leq i < \ell$, we have
$F_i^\up = \tau\big({F'_i}^\up\big)$ and
$F_i^\down = \tau\big({F'_i}^\down\big)$
\end{myitemize}
where $\tau$ denotes the transposition exchanging the letters
$\gA$ and $\gB$.
In this particular situation, the equations defining $\GRs(\underline
F)$ are exactly those defining $\GRs(\underline F')$ plus the
equation $x_{s-1} y_s = 0$.
Moreover the shape functions agree on the smallest variety.
Therefore, the inclusion $\GRs(\underline F) \subset \GRs(\underline F')$
holds and we have to prove that 
$\Card \WW(\underline F) \leq \Card \WW(\underline F')$.

Since the fragments $\underline F$ and $\underline F'$ agree up to
position $s{-}2$, we have $c_i^\square = {c'_i}^\square$ for all $i
< s{-}1$ and all $\square \in \{(\ga,\gb), (\gb,\ga), (\gb,\gb)\}$.
Besides, one also checks that
$c_{s-1}^{(\ga,\gb)} = {c'}^{(\ga,\gb)}_{\hspace{-1ex}s-1}$ and
$c_{s-1}^{(\gb,\gb)} = {c'}^{(\gb,\gb)}_{\hspace{-1ex}s-1}$. Now, 
we use the hypothesis that $\gA$ is dominant in $\underline F$ at
position $s$. It implies that $\gA \in \{F_s^\up, F_s^\down\}$. 
Let us first assume that $F_s^\up = \gA$ and $F_s^\down = \gB$.
In this case, we find:

\medskip

\hspace{2em}%
$\begin{array}{r@{\hspace{0.5ex}}l}
c_s^{(\ga,\gb)} 
 & = c_{s-1}^{(\ga,\gb)}, \smallskip \\
c_s^{(\gb,\gb)} 
 & = c_{s-1}^{(\gb,\gb)},
\end{array}$

\smallskip

\hspace{2em}%
$\begin{array}{r@{\hspace{0.5ex}}l}
{c'_s}^{(\ga,\gb)} 
 & = {c'}^{(\ga,\gb)}_{\hspace{-1ex}s-1}
   + {c'}^{(\gb,\gb)}_{\hspace{-1ex}s-1}
   = c_{s-1}^{(\ga,\gb)} + c_{s-1}^{(\gb,\gb)}, \smallskip \\
{c'_s}^{(\gb,\gb)} 
 & = \max\big({c'}^{(\ga,\gb)}_{\hspace{-1ex}s-1},
     {c'}^{(\gb,\ga)}_{\hspace{-1ex}s-1}\big)
   = \max\big(c_{s-1}^{(\ga,\gb)} + c_{s-1}^{(\gb,\ga)}\big).
\end{array}$

\medskip

From the first part of Lemma~\ref{lem:ineqci}, we deduce that
$c_s^{(\ga,\gb)} \leq {c'_s}^{(\ga,\gb)}$ and $c_s^{(\gb,\gb)}
\leq {c'_s}^{(\gb,\gb)}$. Using again that $\gA$ is dominant in 
$\underline F$ at position $s$, we deduce that there exists a
position $t > s$ with the property that 
$F_t^\up = F_t^\down \in \{\gA,\gAB\}$. We can then apply the
second part of Lemma~\ref{lem:ineqci} with $\underline F$ and
$\tau(\underline F')$ and conclude that 
$c_t^\square \leq {c'_t}^\square$ for all $\square \in
\{(\ga,\gb), (\gb,\ga), (\gb,\gb)\}$. By induction, we finally
find that $c_i^\square \leq {c'_i}^\square$ for all $i \geq t$ and,
in particular, for $i = \ell{-}1$. Corollary~\ref{cor:countwithO} 
allows us to conclude in this case.
The other cases where $(F_s^\up, F_s^\down)$ is equal to $(\gA,\gA)$ 
or $(\gB, \gA)$ are handled similarly.

We now consider general $\underline F$ and $\underline F'$. The 
assumption $\GRs(\underline F) \subset \GRs(\underline F')$ means that 
the set of equations defining $\GRs(\underline F')$ is included in the one  
defining $\GRs(\underline F)$. Besides, the particular choices of 
fragments we have considered earlier correspond exactly to the case 
where \emph{one} equation is removed (up to exchanging rows and 
flipping letters).
One can then go from $\underline F$ to $\underline F'$ by applying
these particular transformations and removing the equations one by one.
The first part of Theorem~\ref{thm:monotonyfragment} then follows by 
transitivity of inequality.

The second part is proved in a similar fashion and left to the
reader.
\end{proof}

\begin{cor}
\label{cor:monotonygene}
Let $\bX$ and $\bX'$ be two genes such that
$\GRs(\bX) \subset \GRs(\bX')$. Then
$\Card \WW(\bX) \subset \Card \WW(\bX')$.
\end{cor}

\begin{proof}
It follows from the assumption that if the $i$-th projection
map $\pr_i : (\P^1)^f \to \P^1$ is constant on $\GRs(\bX')$, it
needs to be constant on $\GRs(\bX)$ as well.
We can then conclude by applying Theorem~\ref{thm:monotonyfragment} to 
each reduced fragment $\underline F'$ of $\bX'$.
\end{proof}

\subsection{Effect of crosses}

We now discuss the second part of Conjecture~\ref{conj:rpsi}.
Following \cite{CDM2}, when a fragment $\underline F$ and integer $i$
are such that $F_i^\up = F_i^\down = \dom_{i+1}(\underline F)$, we
will say that $\underline F$ exhibits a \emph{cross} at position $i$.
It turns out that the presence of crosses has only a very limited
impact on the associate candidate $D(-)$.
Precisely, if $\underline F$ is a fragment exhibiting a cross at 
position $i$ and $\underline F^{(i)}$ is the fragment obtained by 
deleting the $i$-th column in $\underline F$, we have by construction:
$$D(\underline F) \,\simeq\, D(\underline F^{(i)})\:[[T]].$$
where $T$ is a new variable.
Nevertheless, we cannot expect that $R(\underline F) 
\simeq R(\underline F^{(i)})[[T]]$ because, thanks to the
Breuil--Mézard conjecture, this would imply that $\underline F$ and 
$\underline F^{(i)}$ have the same number of weights, which is not
true in general.

In what follows, we prove that when a cross is deleted, 
the number of weights always decreases and, in most cases, it strictly 
decreases. Hence, it is minimal when the fragment does not exhibit any 
crosses. This observation supports Conjecture~\ref{conj:rpsi} as the 
subring of power-bounded functions is the integral model of 
$D(\underline F)$ with the smallest Hilbert--Samuel multiplicity.

\begin{thm}
\label{thm:cross}
Let $\underline F$ be a fragment on length $\ell$ and $i$ be index
at which $\underline F$ exhibits a cross.
We set $\delta_{i-1} = 1$ if $F^\up_{i-1} \sim F^\down_{i-1}$, and 
$\delta_{i-1} = 0$ otherwise.
\begin{myenumerate}[(i)]
\item For any combinatorial weight $\underline w = (w_0, \ldots,
w_{\ell-1})$ in $\WW(\underline F)$, we have $w_i = 0$.
\item For $w_0, \ldots, w_{\ell-2} \in \{0, 1\}$, we have
$(w_0, \ldots, w_{\ell-2}) \in \WW(\underline F^{(i)})$ if and only
if $(w_0, \ldots, w_{i-1}, 0, w_i, \ldots, w_{\ell-2}) \in
\WW(\underline F)$ and $(w_{i-1}, w_i) \neq (\delta_{i-1},1)$.
\end{myenumerate}
\end{thm}

\begin{proof}
Without loss of generality, we may assume
$F_i^\up = F_i^\down = \dom_{i+1}(\underline F) = \gA$.
First of all, we observe that it follows directly from 
Definition~\ref{def:fragmentweight} that the last coordinate of all 
weights in $W_i^{(\ga,\gb)}(\underline F)$ and in 
$W_i^{(\gb,\ga)}(\underline F)$ is $0$. Besides, we have:
\begin{align*}
W_{i+1}^{(\ga,\gb)}(\underline F) 
 & = W_i^{(\ga,\gb)}(\underline F) \times \{0\} 
   \quad \text{if } F_{i+1}^\up = \gA, \\
W_{i+1}^{(\gb,\ga)}(\underline F) 
 & = W_i^{(\gb,\ga)}(\underline F) \times \{0\} 
   \quad \text{if } F_{i+1}^\down = \gA, \\
W_{i+1}^{(\gb,\gb)}(\underline F) 
 & = \big(W_i^{(\ga,\gb)}(\underline F) \cup
          W_i^{(\gb,\ga)}(\underline F)\big) \times \{1\}.
\end{align*}
Hence, the $i$-th coordinate of the elements of 
$W_{i+1}^{(\gb,\gb)}(\underline F)$ is always equal to $0$ and similarly,
the $i$-th coordinate of the elements of $W_{i+1}^{(\ga,\gb)}(\underline F)$
(resp. of $W_{i+1}^{(\gb,\ga)}(\underline F)$) is $0$ provided that
$F_{i+1}^\up = \gA$ (resp. $F_{i+1}^\down = \gA$).
Moreover, a variant of Lemma~\ref{lem:ineqci}.(2) shows that the map 
$(w_0, \ldots, w_{\ell-1}) \mapsto (w_0, \ldots, w_i)$ takes 
$\WW(\underline F)$ to
$$
\tilde W_{i+1}^{(\ga,\gb)}(\underline F) \,\cup\,
\tilde W_{i+1}^{(\gb,\ga)}(\underline F) \,\cup\,
W_{i+1}^{(\gb,\gb)}(\underline F)$$
where $\tilde W_{i+1}^{(\ga,\gb)}(\underline F) = 
W_{i+1}^{(\ga,\gb)}(\underline F)$ if $F_{i+1}^\up = \gA$ and $\emptyset$
otherwise, and similarly $\tilde W_{i+1}^{(\gb,\ga)}(\underline F) = 
W_{i+1}^{(\gb,\ga)}(\underline F)$ if $F_{i+1}^\down = \gA$ and $\emptyset$
otherwise.
As a consequence, we conclude that the $i$-th coordinate of the
elements of $\WW(\underline F)$ are all $0$, which proves~(i).

We now move to~(ii).
In order to save space, we only consider the case where $F_{i-1}^\up
= F_{i+1}^\up = \gA$ and $F_{i-1}^\down = F_{i+1}^\down = \gB$ (the
other cases are treated similarly). Since $\underline F$ and
$\underline F^{(i)}$ agree up to position $i{-}1$, we certainly
have $W_{i-1}^\square(\underline F) = W_{i-1}^\square(\underline 
F^{(i)})$ for all $\square \in \{(\ga,\gb), (\gb,\ga), (\gb,\gb)\}$,
so we can use the shorter notation $W_{i-1}^\square$ to refer
to this set. A direct computation then gives:
\begin{align*}
W_{i+1}^{(\ga,\gb)}(\underline F)
 & = W_{i-1}^{(\ga,\gb)}\times \{(0,0)\} \\
W_{i+1}^{(\gb,\gb)}(\underline F)
 & = \big(W_{i-1}^{(\ga,\gb)} \,\cup\, W_{i-1}^{(\gb,\gb)}\big) 
     \times \{(0,1)\} \\
W_i^{(\ga,\gb)}(\underline F^{(i)})
 & = W_{i-1}^{(\ga,\gb)}\times \{0\} \\
W_i^{(\gb,\gb)}(\underline F^{(i)})
 & = W_{i-1}^{(\gb,\gb)} \times \{1\}
\end{align*}
from what the assertion~(ii) follows.
\end{proof}

The second part of Theorem~\ref{thm:cross} tells us that there
is an injection
$$\begin{array}{rcl}
\iota_i : \qquad
\WW(\underline F^{(i)}) & \hookrightarrow & \WW(\underline F)
\smallskip \\
(w_0, \ldots, w_{\ell-2}) & \mapsto &
(w_0, \ldots, w_{i-1}, 0, w_i, \ldots w_{\ell-2}).
\end{array}$$
For the cardinalities, this implies that
$\Card \WW(\underline F^{(i)}) \leq \Card \WW(\underline F)$,
\emph{i.e.} the number of weights decreases when a cross is removed
as we claimed earlier.
In general, $\iota_i$ is not surjective, however.
The next proposition shows that it actually occurs quite rarely.

\begin{prop}
\label{prop:iotasurj}
We keep the above notations and assume that $\underline F$ is 
reduced (see Definition~\ref{def:reduced}).
The mapping $\iota_i$ is a bijection if and only if $\underline F$
exhibits another cross at position $i{-}1$ or at position $i{+}1$
(or both).
\end{prop}

\begin{proof}
Throughout the proof, we assume that $\underline F$ is top-reduced,
the bottom-reduced case being totally similar.

If $\underline F$ exhibits a cross at position $i{+}1$, it follows from 
Theorem~\ref{thm:cross}.(i) that the $(i{+}1)$-th coordinate of the 
weights of $\underline F$ is always $0$. The surjectivity of $\iota_i$ 
then follows from Theorem~\ref{thm:cross}.(ii). A similar argument shows 
that $\iota_i$ is surjective if $\underline F$ exhibits a cross at 
position $i{-}1$.

We now assume that there is no crosses at position $i{-}1$ and 
$i{+}1$ and we want to prove that $\iota_i$ is not surjective. For
this, it is enough to check that
$\Card \WW(\underline F) > \Card \WW(\underline F^{(i)})$.
We first remark that, given that $\underline F$ is top-reduced, we
have $c_0^{(\ga,\gb)} = 1$ and $c_0^{(\gb,\gb)} = 1$. By induction,
this implies that $c_j^{(\ga,\gb)}$ and $c_j^{(\gb,\gb)}$ 
are both strictly positive for all $j$.
Regarding $c_j^{(\gb,\ga)}$, the only possibility to make it vanish is 
to have $F_0^\down = F_1^\down = \cdots = F_j^\down = \gA$.

After these preparations, 
we need to distinguish between several cases according to the values 
of $F_j^\up$, $F_j^\down$ for $j \in \{i{-}1, i, i{+}1\}$.
As in the proof of Theorem~\ref{thm:cross}, we only treat the case 
where $F_{i-1}^\up = F_{i+1}^\up = F_i^\down = F_{i+1}^\down = \gA$ 
and $F_{i-1}^\down = F_{i+1}^\down = \gB$, the other ones being similar.
In this case, we have $c_{i+1}^{(\gb,\gb)}(\underline F) = 
c_i^{(\gb,\gb)}(\underline F^{(i)}) + c_{i-1}^{(\ga,\gb)}(\underline F)
> c_i^{(\gb,\gb)}(\underline F^{(i)})$.
Using Lemma~\ref{lem:ineqci}.(2), we deduce that there exists $t > i$ such that
$c_{t+1}^\square(\underline F) > c_t^\square(\underline F^{(i)})$ for
$\square \in \{(\ga,\gb), (\gb,\ga), (\gb,\gb)\}$. By induction, these
inequalities continue to hold for all $j \geq t$, which eventually
implies that $\Card \WW(\underline F) > \Card \WW(\underline F^{(i)})$
as wanted.
\end{proof}

\begin{rem}
The hypothesis that $\underline F$ is reduced is necessary 
because there might exist crosses in the initial ``reducible''
part of a fragment and, of course, those crosses have no
influence on the number of weights since it is the case of the 
whole reducible part.
\end{rem}

In a more crude language, Proposition~\ref{prop:iotasurj} says
that the effect of two (or more) consecutive crosses is the same
than that a unique cross. On the contrary, a unique cross cannot
be deleted innocently.
This suggests that the second part of Conjecture~\ref{conj:rpsi}
could be an equivalence provided that we restrict ourselves to
the crosses located in the nonreducible parts of the gene.

\section{The bijection between $\WW(\bX)$ and $\DD(\ttt,\rhobar)$}
\label{sec:proof}

This section is devoted to the proof of Theorem~\ref{thm:main}. The 
general plan of the proof goes as follows. Given a coherent triple 
$(h,\gamma,\gamma')$ parametrizing the datum of $\rhobar$ and $\ttt$, 
we start by constructing in \S \ref{ssec:desccommon} a 
\emph{surjection} $\SW_\SSS : \Sac \cap \Scomp \to \DD(\ttt,\rhobar)$ 
where $\Sac$ and $\Scomp$ are two sets of $(2f)$-periodic sequences of 
integers. It then remains to study the default of injectivity of 
$\SW_\SSS$. For this, we first reinterpret in \S \ref{ssec:ew} the 
elements in the intersection $\Sac \cap \Scomp$ thanks to a new 
combinatorial datum, that we call \emph{enriched weights}. 
We finally relate in \S \ref{ssec:ewtow} those enriched weights to the 
combinatorial weights we have already introduced in \S 
\ref{ssec:weightgene} and measure the default of injectivity of 
$\SW_\SSS$ at this level.

\subsection{Description of common weights}
\label{ssec:desccommon}

In what follows, we constantly work with $(2f)$-periodic sequences 
of integers. We denote their set by $\SSS$ and introduce two definitions
on these sequences.

\begin{definit}
\label{def:active}
A pair of integers $(x,y)$ is \emph{active} if 
the three following conditions are fulfilled:
\begin{myenumerate}[(i)]
\item $0 \leq x, y \leq p$;
\item $x \in \{0, p\}$ or $y \in \{0, p\}$;
\item $x \neq y$.
\end{myenumerate}

\noindent
A sequence $\underline \sigma = (\sigma_i)_{i \in \Z}$ in $\SSS$
is \emph{active} if $(\sigma_i, \sigma_{i+f})$ is active
for all $i$ in $\Z$.
We denote $\Sac$ the subset of $\SSS$ consisting of
active sequences.
\end{definit}

\begin{definit}
\label{def:compatible}
We say that 
$\underline \sigma = (\sigma_i)_{i \in \Z}$ in $\SSS$
is \emph{$(h,\gamma,\gamma')$-compatible} if, for all $i$ in $\Z$, we have:
\begin{equation}
\label{eq:Snc}
\sum_{i=0}^{2f-1} \sigma_i \: p^{2f-1-i} \equiv
h - (q{+}1)\gamma' \pmod{q^2-1}.
\end{equation}
We let $\Scomp$ denote the set of $(h,\gamma,\gamma')$-compatible sequences.
\end{definit}

\subsubsection{From sequences to weights}
\label{sssec:seqweight}

Active sequences are closely related to Serre weights. Indeed,
to any active sequence $\underline \sigma = (\sigma_i)_{i \in \Z}$, 
we associates the tuple
$$\SW_\SSS(\underline\sigma) = (s, \underline r) \,\in\,
\big(\Q/(q{-}1)\Z\big) \times \{0, \ldots, p{-}1\}^f$$
defined as follows.
Given an index $i \in \{0, \ldots, f{-}1\}$, we set:
\begin{equation}
\label{eq:risigma}
\begin{array}{r@{\hspace{0.5ex}}l@{\hspace{3ex}}l}
r_{f-1-i} & = \sigma_i - 1         & \text{if } \sigma_{i+f} = 0, \\
r_{f-1-i} & = p - 1 - \sigma_i     & \text{if } \sigma_{i+f} = p, \\
r_{f-1-i} & = \sigma_{i+f} - 1     & \text{if } \sigma_i = 0, \\
r_{f-1-i} & = p - 1 - \sigma_{i+f} & \text{if } \sigma_i = p
\end{array}
\end{equation}
and $\underline r = (r_0, \ldots, r_{f-1})$.
We observe that the activity condition ensures that the above
definition covers all cases and is not ambiguous. We further 
set $\varepsilon_{f-1-i} = 0$ if $\sigma_i < \sigma_{i+f}$
and $\varepsilon_{f-1-i} = 1$ otherwise. We finally put:
\begin{equation}
\label{eq:ssigma}
s =
  \frac 1{q+1}
  \left(h - \sum_{i=0}^{f-1} (-1)^{\varepsilon_i} (1+r_i) p^i\right)
- \sum_{i=0}^{f-1} \varepsilon_i (1+r_i) p^i
\,\in\, \Q/(q{-}1)\Z.
\end{equation}
When $s$ is an integer, we will slightly abuse notations and 
identify the pair $(s, \underline r)$ with the corresponding Serre
weight without further discussion.

\begin{prop}
\label{prop:SW}
The function $\SW_\SSS$ induces a mapping
$\Sac \cap \Scomp \to \DD(\ttt, \rhobar)$.
\end{prop}

\begin{proof}
We consider $\underline\sigma = (\sigma_i)_{i \in \Z}$ in $\Sac \cap 
\Scomp$ and write $\SW_\SSS(\underline \sigma) = (s, \underline r)$ 
with $\underline r = (r_0, \ldots, r_{f-1})$. By definition, the 
$r_i$'s and $s$ are given by the formulas~\eqref{eq:risigma} 
and~\eqref{eq:ssigma} where we recall that $\varepsilon_{f-1-i} = 0$ if 
$\sigma_i < \sigma_{i+f}$ and $\varepsilon_{f-1-i} = 1$ otherwise.
It follows from the fact that $\underline \sigma$ is $(h,\gamma,
\gamma')$-compatible that:
\begin{equation}
\label{eq:SW:h}
h \equiv 
\sum_{i=0}^{2f-1} \sigma_i \: p^{2f-1-i} \equiv
\sum_{i=0}^{f-1} (\sigma_{i+f} - \sigma_i) \: p^{f-1-i} 
\pmod{q+1}.
\end{equation}
Observe that if $\sigma_i < \sigma_{i+f}$,
the activity condition implies that $\sigma_i = 0$ or $\sigma_{i+f} 
= p$. In the first case, $r_{f-1-i} = \sigma_{i+f} - 1$, from what
we derive $\sigma_{i+f} - \sigma_i = 1 + r_{f-1-i}$. A similar
calculation shows that the last equality holds in the second case
as well. If $\sigma_i > \sigma_{i+f}$, one finds
$\sigma_i - \sigma_{i+f} = 1 + r_{f-1-i}$. Therefore, in all cases,
we have the relation:
\begin{equation}
\label{eq:SW:dsigma}
\sigma_{i+f} - \sigma_i = (-1)^{\varepsilon_{f-1-i}} \cdot 
(1 + r_{f-1-i}).
\end{equation}
Plugging this in Eq.~\eqref{eq:SW:h}, we find that:
$$h \equiv \sum_{i=0}^{f-1} (-1)^{\varepsilon_i} (1+r_i) p^i
\pmod{q+1}.$$
Hence $s$ is an integer (see \ref{eq:ssigma}) and we can see the pair $(s, \underline r)$
as a Serre weight. Moreover, the congruences~\eqref{eq:ssigma} 
and~\eqref{eq:SW:h} show that it lies in $\DD(\rhobar)$.
In order to prove that it lies also in $\DD(\ttt)$, we have to
construct $\varepsilon'_0, \ldots, \varepsilon'_{f-1}$ exhibiting
the two following properties: first, the $r_i$'s and the $c_i$'s have 
to be related by the rules of the table~\eqref{eq:typetable} (page
\pageref{eq:typetable}) and, second, the congruence:
\begin{equation}
\label{eq:SW:s}
s \equiv 
\gamma' + \sum_{i=0}^{f-1} \varepsilon'_i (p-1-r_i) p^i
\pmod{q-1}
\end{equation}
must hold (see Lemma~\ref{lem:alts}). Here the $c_i$'s
are the integers introduced in \S \ref{ssec:serreweights}; we recall 
that they all belong to $\{0, \ldots p{-}1\}$ and that they satisfy 
the relation:
$$\gamma' - \gamma \equiv \sum_{i=0}^{f-1} c_i p^i
\pmod{q-1}.$$
The definition of the $\varepsilon'_i$'s goes as follows:
for those indices $i$ for which $\{\sigma_i, \sigma_{i+f}\} \neq 
\{0,p\}$, we set $\varepsilon'_{f-1-i} = 0$ if $\sigma_i = 0$ or 
$\sigma_{i+f} = 0$, and $\varepsilon'_{f-1-i} = 1$ if $\sigma_i 
= p$ or $\sigma_{i+f} = p$. Note that this dichotomy covers all 
cases thanks to the activity condition.
We then complete the sequence of $\varepsilon'_i$'s by setting
$\varepsilon'_{f-1-i} = \varepsilon'_{f-2-i}$ (where the indices
are considered modulo~$f$) when $\{\sigma_i, \sigma_{i+f}\} =
\{0,p\}$. This definition is not ambiguous except in the situation
where $\{\sigma_i, \sigma_{i+f}\} = \{0,p\}$ for all $i$; in this
case, we agree to define $\varepsilon'_i = 1$ for all $i$.

In order to prove that the $\varepsilon'_i$'s are convenient, we first 
observe that the equality:
$$\sigma_i + \sigma_{i+f} = p - (-1)^{\varepsilon'_{f-1-i}}
\big(p - 1 - r_{f-1-i}\big)$$
holds in all cases. 
Reducing the congruence~\eqref{eq:ssigma} modulo $q{-}1$ and
using the fact that the triple $(h, \gamma, \gamma')$ is coherent,
a simple calculation leaves us with the relation:
$$\sum_{i=0}^{f-1} c_i p^i \equiv 
\sum_{i=0}^{f-1} (-1)^{\varepsilon'_i} \big(p - 1 - r_i\big) p^i
\pmod{q-1}.$$
We set $s_i = p-1-r_i$ if
$\varepsilon'_i = 0$ and $s_i = r_i + 1$ otherwise. An easy 
computation then shows that $(-1)^{\varepsilon'_i} \big(p - 1 - 
r_i\big) = s_i - p \varepsilon'_i$, from what we finally deduce:
\begin{equation}
\label{eq:SW:ciri}
\sum_{i=0}^{f-1} c_i p^i \equiv 
\sum_{i=0}^{f-1} \big(s_i - \varepsilon'_{i-1}\big) p^i
\pmod{q-1}.
\end{equation}
An interesting feature of the latest equality is that
$s_i - \varepsilon'_{i-1}$ is in the range $\{0, \ldots, p{-}1\}$
for all $i$. Indeed, consider first the case where $\varepsilon'_i 
= 1$. With this additional assumption, we have $1 \leq s_i \leq p$ 
and the equality $s_i = p$ holds if and only if $\{\sigma_{f-1-i}, 
\sigma_{2f-1-i}\} = \{0,p\}$. However, in this case, we know that
$\varepsilon'_{i-1} = \varepsilon'_i$, \emph{i.e.}
$\varepsilon'_{i-1} = 1$. In all cases, we can then conclude
that $0 \leq s_i - \varepsilon'_{i-1} \leq p{-}1$. The case where
$\varepsilon'_i = 0$ is treated similarly.

From Eq.~\eqref{eq:SW:ciri}, we then deduce that $c_i = 
s_i - \varepsilon'_{i-1}$ for all $i$ except possibly if
$c_i = 0$ for all $i$ (\emph{i.e.} $\gamma = \gamma'$), in which
case the option $s_i - \varepsilon'_{i-1} = p{-}1$ for all $i$
is left open. 
On the one hand,  when $c_i = s_i - \varepsilon'_{i-1}$, we check by inspection 
that $c_i$ are $r_i$ are related by the rules of the 
table~\eqref{eq:typetable} as expected.  On the other hand, when
$c_i = 0$ and $s_i - \varepsilon'_{i-1} = p{-}1$ for all $i$, 
one proves by induction that
either $r_i = \varepsilon'_i = 0$ for all $i$,
or $r_i = p{-}1$ and $\varepsilon'_i = 1$ for all $i$.
In both cases, we verify that the rules of the 
table~\eqref{eq:typetable} are respected.

It remains to check that Eq.~\eqref{eq:SW:s} holds. For this, 
we observe that the congruence defining the $(h,\gamma, 
\gamma')$-compatibility of $\underline\sigma$ can be rewritten as 
follows:
$$\frac 1{q+1}\left(h -
\sum_{i=0}^{f-1} (\sigma_{i+f} - \sigma_i) p^{f-1-i}\right)
\equiv  \gamma' + \sum_{i=0}^{f-1} \sigma_i p^{f-1-i} \pmod{q-1}.$$
Noticing in addition that
$$\begin{array}{cl}
& \sigma_{i+f} - \sigma_i = 
(-1)^{\varepsilon_{f-1-i}} \cdot (1 + r_{f-1-i})
\smallskip \\ \text{and}
& \sigma_i 
 = \varepsilon_{f-1-i} \cdot (1 + r_{f-1-i}) \,+\,
   \varepsilon'_{f-1-i} \cdot (p - 1 - r_{f-1-i}) 
\end{array}$$
for all $i \in \{0, \ldots, f{-}1\}$, we end up with the relation:
\begin{align}
\label{eq:congrs}
s 
& \equiv
  \frac 1{q+1}
  \left(h - \sum_{i=0}^{f-1} (-1)^{\varepsilon_i} (1+r_i) p^i\right)
- \sum_{i=0}^{f-1} \varepsilon_i (1+r_i) p^i \\
& \equiv 
  \gamma' + \sum_{i=0}^{f-1} \varepsilon'_i (p-1-r_i) p^i
  \pmod{q-1} \nonumber
\end{align}
which is exactly what we had to prove.
\end{proof}

\begin{prop}
\label{prop:SWsurj}\label{prop:surjection}
The mapping $\SW_\SSS : \Sac \cap \Scomp \to \DD(\ttt, \rhobar)$
is surjective.
\end{prop}

\begin{proof}
Let $(s, \underline r)$ be the parameters of a Serre weight in
$\DD(\ttt, \rhobar)$ and write $\underline r = (r_0, \ldots, 
r_{f-1})$.
By definition of $\DD(\rhobar)$ and $\DD(\ttt)$, there exists
$\varepsilon_0, \ldots, \varepsilon_{f-1}, \varepsilon'_0, \ldots, 
\varepsilon'_{f-1}$ in $\{0, 1\}$ satisfying the 
congruences~\eqref{eq:congrs}.
For $i$ in $\{0, \ldots, f{-}1\}$, we set $\varepsilon_{i+f} = 1 - 
\varepsilon_i$ and $\varepsilon'_{i+f} = \varepsilon_i$.
We further extend the sequences $(\varepsilon_i)_i$ and 
$(\varepsilon'_i)_i$ to all $i$ in $\Z$ by $(2f)$-periodicity.
For all $i \in \Z$, we put:
$$\sigma_i 
 = \varepsilon_{f-1-i} \cdot (1 + r_{f-1-i}) \,+\,
   \varepsilon'_{f-1-i} \cdot (p - 1 - r_{f-1-i}).$$
Obviously the sequence $\underline \sigma = (\sigma_i)_{i\in \Z}$ 
lies in $\SSS$. We claim that it is actually an element of
$\Sac \cap \Scomp$, whose image under $\SW_\SSS$ equals
$\sigma_{s, \underline r}$. The fact that $\underline\sigma$
is $(h,\gamma,\gamma')$-compatible is proved by doing in the
reverse direction the final computation of the proof of
Proposition~\ref{prop:SW}. In order to check that $\underline\sigma$
is active, we build the following table in 
which we record the values of the pair $(\sigma_i, \sigma_{i+f})$ 
depending on the values of $\varepsilon_{f-1-i}$ and 
$\varepsilon'_{f-1-i}$.

\bigskip

\hfill%
\begin{tikzpicture}[xscale=4, yscale=0.7]
\draw (0,1)--(2,1);
\draw (-0.8,0)--(2,0);
\draw (-0.8,-1)--(2,-1);
\draw (-0.8,-2)--(2,-2);
\draw (-0.8,0)--(-0.8,-2);
\draw (0,1)--(0,-2);
\draw (1,1)--(1,-2);
\draw (2,1)--(2,-2);
\node at (-0.4,0.5) { $(\sigma_i, \sigma_{i+f})$ };
\node at (0.5,0.5) { $\varepsilon_{f-1-i} = 0$ };
\node at (1.5,0.5) { $\varepsilon_{f-1-i} = 1$ };
\node at (-0.4,-0.5) { $\varepsilon'_{f-1-i} = 0$ };
\node at (-0.4,-1.5) { $\varepsilon'_{f-1-i} = 1$ };

\node at (0.5,-0.5) { \ph $\big(0,\,1{+}r_{f-1-i}\big)$ };
\node at (0.5,-1.5) { \ph $\big(p{-}1{-}r_{f-1-i},\,p\big)$ };
\node at (1.5,-0.5) { \ph $\big(1{+}r_{f-1-i},\,0\big)$ };
\node at (1.5,-1.5) { \ph $\big(p,\,p{-}1{-}r_{f-1-i}\big)$ };
\end{tikzpicture}%
\hfill\null

\bigskip

\noindent
We observe that, for each case, both coordinates are between $0$
and $p$, they cannot be equal and one of them is equal to $0$ or $p$.
Hence the pair $(\sigma_i, \sigma_{i+f})$ is active in the sense of
Definition~\ref{def:active}. Since this holds for all $i$, we conclude
that $\underline \sigma \in \Sac$.

Finally, looking again at the above table,
we deduce that the values of $r_{f-1-i}$ and $\varepsilon_{f-1-i}$
can be recovered from the pair $(\sigma_i, \sigma_{i+f})$ thanks to 
the formulas~\eqref{eq:risigma} on the one hand and the fact that
$\varepsilon_{f-1-i} = 0$ if and only if $\sigma_i < \sigma_{i+f}$ 
on the other hand. This observation eventually ensures that
$\SW_\SSS$ takes $\underline\sigma$ to the Serre weight with
parameters $(s, \underline r)$ we started with.
\end{proof}

\subsubsection{Nonviable genes}\label{subsubempty}

As a first application of Propositions~\ref{prop:SW}
and~\ref{prop:SWsurj},
we prove that if $\bX$ is not viable (see
Definition~\ref{def:viable}), then the set of common
Serre weights is empty.
We start by recording a lemma that we shall use several
times in what follows.

\begin{lem}
\label{lem:zeroorp}
Let $\underline \sigma = (\sigma_i)_{i \in \Z}$ be in $\Sac \cap \Scomp$
and $i$ be in $\Z$. We assume $\sigma_i \in \{0, p\}$. Then
$X_i \neq \gO$.
\end{lem}

\begin{proof}
The compatibility condition tells us that
$\alpha_i \equiv \beta_i \pmod{q-1}$ where $\beta_i$ is defined by:
$$\beta_i =
\left\lfloor
\sum_{j=1}^{f} \frac{q \sigma_{i+j} + \sigma_{i+j+f}}{q + 1} \: p^{f-1-j}
\right\rfloor.$$
From the activity condition, we deduce that
$q \sigma_{i+j} + \sigma_{i+j+f} < p \cdot (q{+}1)$ for all $j$.
We then find
$0 \leq \beta_i < 1 + p + \cdots + p^{f-1}$, from what it follows
$X_i \neq \gO$.
\end{proof}

\begin{cor}\label{abovecor}
If the gene $\bX$ is not viable, then $\DD(\ttt, \rhobar)$
is empty.
\end{cor}

\begin{proof}
By Proposition~\ref{prop:SWsurj}, it is enough to prove that
$\Sac \cap \Scomp = \emptyset$. Arguing by contradiction, let us 
consider $\underline\sigma
= (\sigma_i)_{i \in \Z} \in \Sac \cap \Scomp$. From our assumption,
we know that there exists an index $i$ such that $X_i = X_{i+f}
= \gO$. From Lemma~\ref{lem:zeroorp}, it then follows that both
$\sigma_i$ and $\sigma_{i+f}$ do not belong to $\{0, p\}$, which
contradicts the definition of activity.
\end{proof}

Theorem~\ref{thm:main} follows from Corollary \ref{abovecor}
when $\bX$ is not viable (since $\WW(\bX)$ is then empty by
definition).  The case of viable genes is more complicated ; it is covered in the next subsections.

\subsection{Description of $\Sac \cap \Scomp$}
\label{ssec:ew}

In this subsection, we give a complete description of the set
$\Sac \cap \Scomp$ in terms of a novel combinatorial datum, that we
call \emph{enriched weights}. The link between enriched weights
and combinatorial weights as defined in \S \ref{ssec:weightgene} 
is precised in \S \ref{ssec:ewtow}.

\subsubsection{Mutations}

We recall that to any coherent triple $(h, \gamma, \gamma')$,
we have associated in \S \ref{ssec:defgene} a sequence $\underline
v = (v_i)_{i \in \Z} \in \SSS$ such that:
$$h - (q+1)\gamma' \equiv \sum_{i=0}^{2f-1} v_i \: p^{2f-1-i}
\pmod{q^2 - 1}$$
In other words, $\underline v$ is $(h,\gamma,\gamma')$-compatible. 
Even better, it is the \emph{unique} sequence in $\Scomp$ assuming 
values in $\{0, \ldots, p{-}1\}$. 
In order to describe all sequences in $\Scomp$, we introduce the 
\emph{mutation operators}. Given
$\underline \chi = (\chi_i)_{i \in \Z}$ in $\SSS$, we define:
$$\begin{array}{rcl}
\mut_{\underline\chi} : \quad \SSS & \longrightarrow & \SSS 
 \smallskip \\
(\sigma_i)_{i\in\Z} & \mapsto &
(\sigma_i - \chi_i + p \chi_{i-1})_{i\in\Z} 
\end{array}$$
One readily checks that 
$\mut_{\underline \chi} \circ \mut_{\underline \chi'} = 
\mut_{\underline \chi + \underline \chi'}$ for all $
\underline \chi, \underline \chi' \in \SSS$. In other words,
the association $\underline\chi \mapsto \mut_{\underline\chi}$ defines 
a group homomorphism. In particular, the $\mut_{\underline \chi}$'s 
pairwise commute.

\begin{prop}
\label{prop:Snc}
A sequence $\underline\sigma$ in $\SSS$ lies in $\Scomp$
if and only if there exists $\underline\chi$ in $\SSS$ such that
$\underline\sigma = \mut_{\underline\chi}(\underline v)$.
Moreover, when this occurs, $\underline\chi$ is uniquely determined.
\end{prop}

\begin{proof}
We consider $\underline\sigma \in \SSS$.
Solving a linear system, we find that $\underline \sigma = 
\mut_{\underline\chi}(\underline v)$ is equivalent to:
$$\chi_i = \frac 1{q^2 - 1}
\sum_{j=0}^{2f-1} (\sigma_{i+j} - v_{i+j}) \: p^{2f-1-j}
\quad (i \in \Z).$$
The proposition follows easily from this.
\end{proof}

We aim at describing all the sequences $\underline\sigma$ which
are at the same time $(h,\gamma,\gamma')$-compatible and active.
By Proposition~\ref{prop:Snc}, this amounts to characterize the
$\underline\chi$'s such that $\mut_{\underline\chi}(\underline v)$
is active.
The starting point is the following simple lemma that narrows the field 
of possibilities.

\begin{lem}
If $\mut_{\underline\chi}(\underline v)$ is active,
then $\underline\chi$ assumes values in $\{0, 1\}$.
\end{lem}

\begin{proof}
Set $\underline\sigma = (\sigma_i)_{i \in \Z} = 
\mut_{\underline\chi}(\underline v)$. By assumption,
$\sigma_i = v_i - \chi_i + p \chi_{i-1} \in [0,p]$ for all~$i$.
Let $i_0$ be the index for which $\chi_i$ is maximal. Then:
$$p \geq \sigma_{i_0+1} = v_{i_0+1} - \chi_{i_0+1} + p \chi_{i_0} 
\geq v_{i_0} + (p-1) \chi_{i_0} \geq (p-1)\chi_{i_0}.$$
We deduce that $\chi_{i_0} \leq 1$ and then, that $\chi_i
\leq 1$ for all~$i$. We prove similarly that $\chi_i$ is always
nonnegative, showing the lemma.
\end{proof}

\subsubsection{Enriched weights}

Instead of working with the sequence $\underline\chi$, it is 
convenient to slightly re-encode it. For this, we introduce the notion 
of \emph{enriched weights}.

\begin{definit}
An \emph{enriched weight} of length $f$ is a periodic sequence of 
period $2f$
assuming values in the finite set $\{\ga, \gb\}$.
\end{definit}

We introduce the function $\lambda : \{\gA, \gB, \gAB, \gO\}
\to \{\ga, \gb\}$ defined by
$\lambda(\gA) = \lambda(\gAB) = \ga$ and 
$\lambda(\gB) = \lambda(\gO) = \gb$. Notice that the equality
$\lambda(X) = \lambda(Y)$ is equivalent to $X \sim Y$ where $\sim$
is the equivalence relation introduced in \S \ref{ssec:recipe}.
To a numerical sequence $\underline\chi = (\chi_i)_{i \in \Z}$ 
assuming values in $\{0, 1\}$, we associate the enriched weight
$\underline \ew = (\ew_i)_{i \in \Z}$ which is uniquely determined by 
the following condition: for all $i \in \Z$,
$\chi_i = 1$ if and only if $\ew_i = \lambda(X_i)$.
More concretely, the following table shows what is the value of 
$\ew_i$ in terms of $\chi_i$ and $X_i$:

\vspace{-1em}

\begin{equation}
\label{eq:table}
\raisebox{-0.5\height}{%
\begin{tikzpicture}[xscale=2.8, yscale=0.6]
\draw (0.5,0.5)--(2.5,0.5);
\draw (-0.1,-0.5)--(2.5,-0.5);
\draw (-0.1,-2.5)--(2.5,-2.5);
\draw (-0.1,-0.5)--(-0.1,-2.5);
\draw (0.5,0.5)--(0.5,-2.5);
\draw (2.5,0.5)--(2.5,-2.5);
\draw (-0.1,-1.5)--(2.5,-1.5);
\draw (1.5,0.5)--(1.5,-2.5);
\node at (1,0) { \ph $X_i \in \{\gA, \gAB\}$ };
\node at (2,0) { \ph $X_i \in \{\gB, \gO\}$ };
\node at (0.2,-1) { \ph $\chi_i = 0$ };
\node at (1,-1) { \ph \gb };
\node at (2,-1) { \ph \ga };
\node at (0.2,-2) { \ph $\chi_i = 1$ };
\node at (1,-2) { \ph \ga };
\node at (2,-2) { \ph \gb };
\end{tikzpicture}}
\end{equation}

\noindent
There is obviously a bijection between the set of enriched weights and 
the set of $\underline\chi$'s. We are now going to translate the 
activity condition (introduced in Definition~\ref{def:active}) 
at the level of enriched weights.
For convenience, we introduce the following definition.

\begin{definit}\label{def:enrichedcompatible}
An enriched weight $\underline \ew$ is called 
\emph{$(h,\gamma,\gamma')$-active} (or just \emph{active} if no
confusion may arise) if the
sequence $\mut_{\underline\chi} (\underline v)$ lies in $\Sac$
(where $\underline \chi$ denotes as usual the numerical sequence 
associated to $\underline\ew$).

\vspace{-\parskip}

We denote by $\hat\WW(h, \gamma,\gamma')$ 
the set of $(h,\gamma,\gamma')$-active enriched weights.
\end{definit}

Our goal is now to characterize active enriched weights.

\begin{lem}
\label{lem:enrichedweights}
Let $\underline \ew$ be an enriched weight and let $\underline \chi$ be 
the numerical sequence associated to it. Set $\underline\sigma = 
(\sigma_i)_{i \in \Z} = \mut_{\underline\chi} (\underline v)$.
Then, for all $i$ in $\Z$, the following holds:
\begin{myenumerate}[(i)]
\item if $\ew_i = \ga$, then $p$ does not divide $\sigma_i$;
\item if $\ew_i = \gb$ and $X_i \neq \gO$, then $\sigma_i \in \{0, p\}$.
\end{myenumerate}
\end{lem}

\begin{proof}
We assume $\ew_i = \ga$. Looking up at the above table, we deduce that 
either $\chi_i = 0$ and $X_i \in \{\gB, \gO\}$ or $\chi_i = 1$ and $X_i 
\in \{\gA, \gAB\}$.
In the first case, from Lemma~\ref{lem:vi}, we find that
$v_i \in \{1, \ldots, p{-}1\}$. In the second case, 
we obtain $v_i = 0$. In both cases, it turns out that $v_i - \chi_i$ 
cannot be divisible by $p$. Hence $\sigma_i = v_i - \chi_i + p 
\chi_{i-1}$ is not divisible by $p$ either. This proves~(i).

Let us now assume $\ew_i = \gb$ and $X_i \neq \gO$. Then looking up
again at the table, we are faced to the following alternative:
either $\chi_i = 0$ and $X_i \in \{\gA, \gAB\}$ or  $\chi_i = 1$ and 
$X_i = \gB$. In both cases, Lemma~\ref{lem:vi} shows
that $\chi_i = v_i$. Hence $\sigma_i = p \chi_{i-1}$, from which (ii)~follows.
\end{proof}

After Lemma~\ref{lem:enrichedweights}, we have:

\begin{cor}
\label{cor:notaa}
If $\underline \ew$ is an active enriched weight,
then $(\ew_i, \ew_{i+f}) \neq (\ga, \ga)$ for all~$i$.
\end{cor}

We now come to the core proposition which characterizes active 
enriched weights ``outside~\gO''.

\begin{prop}
\label{prop:admoutsideO}
Let $\underline \ew = (\ew_i)_{i \in \Z}$ be an enriched weight
with $(\ew_i, \ew_{i+f}) \neq (\ga, \ga)$ for all $i$.
Let $\underline \chi$ be the numerical sequence associated to it
and set $\underline \sigma = (\sigma_i)_{i \in \Z} = 
\mut_{\underline \chi}(\underline v)$.
Let $i$ be an integer with $X_i \neq \gO$ and $X_{i+f} \neq \gO$.
Then the pair $(\sigma_i, \sigma_{i+f})$ is active if and only if
the following holds:
\begin{myenumerate}[(bb${}_1$)]
\item[\rm ($\gb\gb_1$)]
if $(\ew_i, \ew_{i+f}) = (\gb, \gb)$ and $\lambda(X_{i-1}) = \lambda(X_{i-1+f})$,
then $\ew_{i-1} \neq \ew_{i-1+f}$;
\smallskip
\item[\rm ($\gb\gb_2$)]
if $(\ew_i, \ew_{i+f}) = (\gb, \gb)$ and $\lambda(X_{i-1}) \neq \lambda(X_{i-1+f})$,
then $\ew_{i-1} = \ew_{i-1+f}$;

\smallskip

\item[\rm ($\ga\gb_1$)]
if $(\ew_i, \ew_{i+f}) = (\ga, \gb)$ and $\lambda(X_i) = \lambda(X_{i-1})$,
then $\ew_{i-1} = \ga$;
\smallskip
\item[\rm ($\ga\gb_2$)]
if $(\ew_i, \ew_{i+f}) = (\ga, \gb)$ and $\lambda(X_i) \neq \lambda(X_{i-1})$,
then $\ew_{i-1} = \gb$;

\smallskip

\item[\rm ($\gb\ga_1$)]
if $(\ew_i, \ew_{i+f}) = (\gb, \ga)$ and $\lambda(X_{i+f}) = \lambda(X_{i-1+f})$,
then $\ew_{i-1+f} = \ga$;
\smallskip
\item[\rm ($\gb\ga_2$)]
if $(\ew_i, \ew_{i+f}) = (\gb, \ga)$ and $\lambda(X_{i+f}) \neq \lambda(X_{i-1+f})$,
then $\ew_{i-1+f} = \gb$.
\end{myenumerate}
\end{prop}

\begin{proof}
We first assume that $(\ew_i, \ew_{i+f}) = (\gb, \gb)$. It then 
follows from Lemma~\ref{lem:enrichedweights} that both $\sigma_i$
and $\sigma_{i+f}$ lie in $\{0, p\}$. As a consequence, the pair
$(\sigma_i, \sigma_{i+f})$ is active if and only if
$\sigma_i \neq \sigma_{i+f}$, which is further equivalent to
$\chi_{i-1} \neq \chi_{i-1+f}$. Looking up at table~\eqref{eq:table}, 
we finally obtain the necessary and sufficient conditions~($\gb\gb_1$) 
and ($\gb\gb_2$).

Let us now assume $(\ew_i, \ew_{i+f}) = (\ga, \gb)$.
From Lemma~\ref{lem:enrichedweights}, we deduce that
$\sigma_{i+f} \in \{0, p\}$ whereas $p$ does not divide $\sigma_i$.
Consequently, the pair $(\sigma_i, \sigma_{i+f})$ is active if
and only if $1 \leq \sigma_i \leq p$ regardless the value of
$\sigma_{i+f}$. Since $\sigma_i = v_i - \chi_i + 
p \chi_{i-1}$ and $v_i \in \{0, 1\}$ (by Lemma~\ref{lem:vi}), 
this happens if and only if $\chi_{i-1} = \chi_i$, which is further 
equivalent to the conditions ~($\ga\gb_1$) and ($\ga\gb_2$) thanks to
the records of table \eqref{eq:table}.

Finally, the case $(\ew_i, \ew_{i+f}) = (\gb, \ga)$ is handled similarly.
\end{proof}

It now remains to understand the activity condition at positions
where the letter $\gO$ appears in the gene. This is the content of
the following proposition.

\begin{prop}
\label{prop:admatO}
Let $\underline \ew = (\ew_i)_{i \in \Z}$ be an enriched weight.
Let $\underline \chi$ be the numerical sequence associated to it
and set $\underline \sigma = (\sigma_i)_{i \in \Z} = 
\mut_{\underline \chi}(\underline v)$.
Then $\underline \ew$ is active if and only if:
\begin{myenumerate}[(i)]
\item[(a)]
for all $i$ such that $X_i \neq \gO$ and $X_{i+f} \neq \gO$,
the pair $(\sigma_i, \sigma_{i+f})$ is active, and

\smallskip

\item[(b)]
for all $i$ such that $X_i = \gO$, we have $\ew_{i-1} \neq 
\lambda(X_{i-1})$ and $\ew_{i+f} = \gb$.
\end{myenumerate}
\end{prop}

\begin{proof}
We first assume that $\underline \ew$ is active.
Then~(a) is clearly true. We now consider an index $i$ such that
$X_i = \gO$. From Lemma~\ref{lem:zeroorp},
we deduce that both $v_i$ and $\sigma_i$ are different from $0$
and $p$, \emph{i.e.} $v_i, \sigma_i \in [1, \ldots, p{-}1]$. It
follows that $|\sigma_i - v_i| \leq p{-}2$.
Since we have the equality $\sigma_i = v_i - \chi_i +
p \chi_{i-1}$,  we deduce that:
$$|p \chi_{i-1}| \leq |\sigma_i - v_i| + |\chi_i| \leq p-1.$$
Therefore $\chi_{i-1}$ has to vanish, implying by definition that 
$\ew_{i-1} \neq \lambda(X_{i-1})$.
Moreover, from the activity of the pair $(\sigma_i,
\sigma_{i+f})$, we derive that $\sigma_{i+f} = v_{i+f} - \chi_{i+f} + 
p \chi_{i+f-1} \in \{0,p\}$. Hence $\chi_{i+f} \equiv v_{i+f} \pmod
p$, which gives $\chi_{i+f} = v_{i+f}$ since both $\chi_{i+f}$
and $v_{i+f}$ are in the range $[0, p{-}1]$. Looking up at table
\eqref{eq:table}, we finally deduce that $\ew_{i+f} = \gb$.
We have then proved~(b).

Conversely, let us assume the conditions~(a) and~(b).
We need to prove that the pair $(\sigma_i, \sigma_{i+f})$ is
active as soon as $X_i = \gO$ or $X_{i+f} = \gO$. 
Since replacing $i$ by $i{+}f$ leaves unchanged the activity
condition, we may assume without loss of generality that $X_i = \gO$.
Combining the assumption~(b) with Lemma~\ref{lem:vi}, 
we derive $\chi_{i-1} = 0$ and $\chi_{i+f} = v_{i+f}$. 
Thus $\sigma_i = v_i - \chi_i$ and
$\sigma_{i+f} = p \chi_{i+f-1}$. Hence $0 \leq \sigma_i \leq p{-}1$
and $\sigma_{i+f} \in \{0,p\}$.
It is enough to exclude the option $\sigma_i = 0$. 
For this, we cannot apply directly 
Lemma~\ref{lem:zeroorp} because we do not know that $\underline\sigma$ 
is active (it is actually what we want to prove); however we can
mostly reuse the same argument. Indeed, looking at the proof of 
Lemma~\ref{lem:zeroorp}, we see that the conclusion follows if we
can ensure that
$q\sigma_{i+j} + \sigma_{i+j+f} < p\cdot(q{+}1)$ for all integer~$j$. If 
$X_{i+j} \neq \gO$ and $X_{i+j+f} \neq \gO$, this is a consequence of
the activity condition (as already noticed in the proof of
Lemma~\ref{lem:zeroorp}). If $X_{i+j} = \gO$, we 
have just proved that $\sigma_{i+j} < p$ and $\sigma_{i+j+f} = p 
\chi_{i+j+f-1} \leq p$. Hence 
$q \sigma_{i+j} + \sigma_{i+j+f} < p\cdot (q{+}1)$ as 
wanted. Similarly if $X_{i+j+f} = \gO$, we have $\sigma_{i+j} =
p \chi_{i+j-1} \leq p$ and 
$\sigma_{i+j+f} < p$, so the same conclusion follows.
\end{proof}

\subsubsection{Fragmentation and description of active enriched weights}
\label{sssec:descactive}

Propositions~\ref{prop:admoutsideO} and~\ref{prop:admatO} together
entirely elucidate the activity condition for enriched weights.
Besides, Proposition~\ref{prop:admatO} shows that the activity
condition can be checked independently on each fragment of the gene
as defined in \S \ref{ssec:weightgene}.
In order to be more precise, we introduce the following definition.

\begin{definit}
\label{def:enrichedfragment}
A \emph{fragmentary enriched weight} of length $\ell$ is a tuple
$\underline \ew = (\ew_0, \ldots, \ew_{\ell-1})$ with
$\ew_i = (\ew^\up_i, \ew^\down_i) \in \{\ga,\gb\}^2$ for all $i \in
\{0, 1, \ldots, \ell-1\}$.

\smallskip

Let $\underline F = (F_0, \ldots, F_{\ell-1})$ be a fragment of length 
$\ell$ with $F_i = (F^\up_i, F^\down_i)$.
We say that $\ew$ is a fragmentary enriched weight of $\underline F$ if
\begin{myenumerate}[(X)]
\item[(L)]
if $F^\up_0 = \gO$ (resp. $F^\down_0 = \gO$), then 
$\ew^\down_0 = \gb$ (resp. $\ew^\up_0 = \gb$)

\smallskip

\item[(R)]
if $F^\up_{\ell-1} = \gAB$ (resp. $F^\down_{\ell-1} = \gAB$), then 
$\ew^\up_{\ell-1} = \gb$ (resp. $\ew^\down_{\ell-1} = \gb$)

\noindent
if $\ell = 1$ and $F^\up_0 \in \{\gA, \gB\}$ (and thus $F^\down_0 = \gO$), 
then $\ew^\down_0 = \ga$

\noindent
if $\ell = 1$ and $F^\down_0 \in \{\gA, \gB\}$ (and thus $F^\up_0 = \gO$), 
then $\ew^\up_0 = \ga$
\end{myenumerate}

\vspace{-\parskip}

\noindent
and, for all $i \in \{1,\ldots, \ell{-}1\}$:

\begin{myenumerate}[(bb${}_1$)]
\item[($\ga\ga$)]
$(\ew^\up_i, \ew^\down_i) \neq (\ga,\ga)$

\smallskip

\item[($\gb\gb_1$)]
if $(\ew^\up_i, \ew^\down_i) = (\gb, \gb)$ 
and $\lambda(F^\up_{i-1}) = \lambda(F^\down_{i-1})$,
then $\ew^\up_{i-1} \neq \ew^\down_{i-1}$,
\smallskip
\item[($\gb\gb_2$)]
if $(\ew^\up_i, \ew^\down_i) = (\gb, \gb)$
and $\lambda(F^\up_{i-1}) \neq \lambda(F^\down_{i-1})$,
then $\ew^\up_{i-1} = \ew^\down_{i-1}$,

\smallskip

\item[($\ga\gb_1$)]
if $(\ew^\up_i, \ew^\down_i) = (\ga, \gb)$ 
and $\lambda(F^\up_i) = \lambda(F^\up_{i-1})$,
then $\ew^\up_{i-1} = \ga$,
\smallskip
\item[($\ga\gb_2$)]
if $(\ew^\up_i, \ew^\down_i) = (\ga, \gb)$ 
and $\lambda(F^\up_i) \neq \lambda(F^\up_{i-1})$,
then $\ew^\up_{i-1} = \gb$,

\smallskip

\item[($\gb\ga_1$)]
if $(\ew^\up_i, \ew^\down_i) = (\gb, \ga)$
and $\lambda(F^\up_i) = \lambda(F^\down_{i-1})$,
then $\ew^\down_{i-1} = \ga$,
\smallskip
\item[($\gb\ga_2$)]
if $(\ew^\up_i, \ew^\down_i) = (\gb, \ga)$ 
and $\lambda(F^\up_i) \neq \lambda(F^\down_{i-1})$,
then $\ew^\down_{i-1} = \gb$.
\end{myenumerate}

\noindent
We denote by $\hat \WW(\underline F)$ the set of all fragmentary 
enriched weights of $\underline F$.
\end{definit}

After Propositions~\ref{prop:admoutsideO} 
and~\ref{prop:admatO}, we obtain:

\begin{prop}
\label{prop:cartesian}
We have:
$$\hat\WW(h, \gamma,\gamma') \,\simeq\, 
  \prod_{\underline F}\, \hat\WW(\underline F)$$
where the product runs of all fragments $\underline F$ of $\bX$.
\end{prop}

Proposition~\ref{prop:cartesian} shows in particular that the 
set~$\hat\WW(h,\gamma,\gamma')$ depends only on the gene $\bX$. In what
follows, we will often denote it by $\hat\WW(\bX)$.
Besides, describing
$\hat\WW(\bX)$ reduces to compute the sets $\hat\WW(\underline F)$.
This can actually be easily achieved by induction on $i$ as stated
in the following proposition.

\begin{prop}
\label{prop:fragmentweights}
Let $\underline F = (F_0, \ldots, F_{\ell-1})$ be a fragment and
write $F_i = (F^\up_i, F^\down_i)$. Then:
$$\begin{array}{r@{\hspace{0.5ex}}l@{\qquad}l}
\hat\WW(\underline F)
 & = \EW^{(\gb,\gb)}_{\ell-1} \cup \EW^{(\ga,\gb)}_{\ell-1}
 & \text{if } F^\down_{\ell-1} = \gAB \smallskip \\
 & = \EW^{(\gb,\gb)}_{\ell-1} \cup \EW^{(\gb,\ga)}_{\ell-1}
 & \text{if } F^\up_{\ell-1} = \gAB \smallskip \\
 & = \EW^{(\gb,\gb)}_0 \cup \EW^{(\ga,\gb)}_0 \cup \EW^{(\gb,\ga)}_0
 & \text{otherwise.}
\end{array}$$

\noindent
where the sequences $(\EW^{(\gb,\gb)}_i)_{0 \leq i < \ell}$,
$(\EW^{(\ga,\gb)}_i)_{0 \leq i < \ell}$,
$(\EW^{(\gb,\ga)}_i)_{0 \leq i < \ell}$ are defined by:

\medskip

$\begin{array}{cr@{\hspace{0.5ex}}l@{\qquad}l}
\bullet
 & \EW^{(\gb,\gb)}_0
 & = \emptyset
 & \text{if } \ell = 1 \text{ and } \big(F^\up_0 \in \{\gA, \gB\}
                         \text{ or } F^\down_0 \in \{\gA, \gB\}\big) \smallskip \\
&& = \big\{(\gb,\gb)\big\}
 & \text{otherwise}
\end{array}$

\smallskip

$\begin{array}{cr@{\hspace{0.5ex}}l@{\qquad}l}
\bullet
 & \EW^{(\ga,\gb)}_0
 & = \emptyset
 & \text{if } F^\down_0 = \gO \smallskip \\
&& = \big\{(\ga,\gb)\big\}
 & \text{otherwise}
\end{array}$

\smallskip

$\begin{array}{cr@{\hspace{0.5ex}}l@{\qquad}l}
\bullet
 & \EW^{(\gb,\ga)}_0
 & = \emptyset
 & \text{if } F^\up_0 = \gO \smallskip \\
&& = \big\{(\gb,\ga)\big\}
 & \text{otherwise}
\end{array}$

\medskip

\noindent
and the following recurrence formulas (for $1 \leq i \leq \ell-1$):

\medskip

$\begin{array}{cr@{\hspace{0.5ex}}l@{\qquad}l}
\bullet
 & \EW^{(\gb,\gb)}_i 
 & = \big(\EW^{(\ga,\gb)}_{i-1} \cup \EW^{(\gb,\ga)}_{i-1}\big) 
     \times \big\{(\gb,\gb)\big\}
 & \text{if } \lambda(F^\up_{i-1}) = \lambda(F^\down_{i-1}) \smallskip \\
&& = \EW^{(\gb,\gb)}_{i-1}
    \times \big\{(\gb,\gb)\big\}
 & \text{otherwise}
\end{array}$

\smallskip

$\begin{array}{cr@{\hspace{0.5ex}}l@{\qquad}l}
\bullet
 & \EW^{(\ga,\gb)}_i 
 & = \EW^{(\ga,\gb)}_{i-1}
     \times \big\{(\ga,\gb)\big\}
 & \text{if } \lambda(F^\up_i) = \lambda(F^\up_{i-1}) \smallskip \\
&& = \big(\EW^{(\gb,\ga)}_{i-1} \cup \EW^{(\gb,\gb)}_{i-1}\big) 
    \times \big\{(\ga,\gb)\big\}
 & \text{otherwise}
\end{array}$

\smallskip

$\begin{array}{cr@{\hspace{0.5ex}}l@{\qquad}l}
\bullet
 & \EW^{(\gb,\ga)}_i 
 & = \EW^{(\gb,\ga)}_{i-1}
     \times \big\{(\gb,\ga)\big\}
 & \text{if } \lambda(F^\down_i) = \lambda(F^\down_{i-1}) \smallskip \\
&& = \big(\EW^{(\ga,\gb)}_{i-1} \cup \EW^{(\gb,\gb)}_{i-1}\big) 
    \times \big\{(\gb,\ga)\big\}
 & \text{otherwise.}
\end{array}$
\end{prop}

\begin{proof}
This is a direct consequence of the definitions.
\end{proof}

\subsection{From enriched weights to combinatorial weights}
\label{ssec:ewtow}

Before continuing, we
do a brief recap of what we have done. On the
one hand, we have seen in~\S \ref{ssec:ew} that active and 
$(h,\gamma,\gamma')$-compatible sequences are described by enriched 
weights. Precisely, we have constructed a bijection:
$$\begin{array}{rcl}
\hat\WW(h,\gamma,\gamma') 
  & \stackrel\sim\longrightarrow  & \Sac \cap \Scomp \smallskip \\
\underline \ew 
  & \mapsto & \underline\sigma= \mut_{\underline\chi}(\underline v)
\end{array}$$
where $\underline\chi$ denotes the numerical sequence associated
to $\underline\ew$ by the rules of the table~\eqref{eq:table}.
On the other hand, in \S \ref{ssec:desccommon}, we 
have constructed a surjection:
$$\SW_\SSS : \Sac \cap \Scomp \longrightarrow \DD(\ttt,\rhobar)$$
(see Proposition~\ref{prop:SWsurj}).
Composing these two functions, we get a surjective map:
$$\hat\SW : \,\, \hat\WW(h,\gamma,\gamma') 
\stackrel\sim\longrightarrow \Sac \cap \Scomp
\stackrel{\SW_\SSS}\longrightarrow \DD(\ttt,\rhobar).$$
The final step in the proof of Theorem~\ref{thm:main}
consists in establishing a link between $\hat\WW(h,\gamma,
\gamma')$ and the set $\WW(\bX)$ of combinatorial weights of 
$\bX$ introduced in Definitions~\ref{def:fragmentweight}
and~\ref{def:geneweight}.
Precisely, we are going to prove that the map $\SW$ considered
in Theorem~\ref{thm:main} and the map $\hat\SW$ introduced above
sit in a commutative diagram of the form:
$$\xymatrix @C=5em {
\hat\WW(h,\gamma,\gamma') 
\ar[d]_-{\comb} \ar@{->>}[r]^-{\hat\SW} \ar@{->>}[d] &
\DD(\ttt,\rhobar) \\
\WW(\bX) \ar[ru]_-{\SW}}$$
where $\bX$ denotes the gene of $(h, \gamma,\gamma')$.
The vertical map $\comb$ is defined as follows: it takes an
enriched weight $\underline\ew = (\ew_i)_{i\in \Z}$ to the 
combinatorial weight:
$$\comb(\ew) = \big(\delta(\ew_i, \ew_{i+f})\big)_{i \in \Z}$$
where the $\delta$ function is defined by:
$$\begin{array}{r@{\hspace{0.5ex}}cl@{\qquad}l}
\delta(x,y) & = 1 & \text{if } x = y \smallskip \\
\delta(x,y) & = 0 & \text{otherwise.}
\end{array}$$
The fact that $\comb$ takes $\hat\WW(h,\gamma,\gamma')$ to 
$\WW(\bX)$ follows from
Propositions~\ref{prop:cartesian} and~\ref{prop:fragmentweights}.
Moreover, we deduce from the constructions that
$\comb$ is surjective.

Before getting to the heart of the matter, we underline that 
the $\delta$ notation allows us to write down a simple formula 
summarizing the table~\eqref{eq:table}, that is:
\begin{equation}
\label{eq:deltachii}
\chi_i = \delta\big(\ew_i, \lambda(X_i)\big)
\end{equation}
where $(\ew_i)_{i \in \Z}$ is an enriched weight and 
$(\chi_i)_{i \in \Z}$ is its associated numerical sequence.
Similarly, the parameter $\delta_i$, which appears in the table
of Figure~\ref{fig:ri}, is simply equal to $\delta(\lambda(X_i),
\lambda(X_{i+f}))$.
Another useful remark on the $\delta$ function is given by the next
lemma.

\begin{lem}
\label{lem:delta}
If $E$ is a set with two elements and $x_1$, $x_2$, $y_1$ and 
$y_2$ are elements of $E$, we have the identity:
\begin{equation}
\label{eq:delta}
\delta\big(\delta(x_1,x_2), \delta(y_1,y_2)\big) =
\delta\big(\delta(x_1,y_1), \delta(x_2,y_2)\big).
\end{equation}
\end{lem}

\begin{proof}
Without loss of generality, we may take $E = \{0, 1\}$.
With this further assuption,
the congruence $\delta(x,y) = x + y + 1 \pmod 2$ holds for all
$x, y$ in $E$. Therefore, the left hand side and the right hand
side of~\eqref{eq:delta} are both congruent to $x_1 + x_2 + y_1 + 
y_2 + 1$ modulo $2$. Since they both belong to $E$, they need
to be equal.
\end{proof}

\begin{prop}
\label{prop:diagcomm}
We have $\hat \SW = \SW \circ \comb$.
\end{prop}

\begin{proof}
Let $\underline\ew = (\ew_i)_{i \in \Z}$ be an enriched weight in
$\hat\WW(h,\gamma,\gamma')$, let $(\chi_i)_{i \in \Z}$ be the 
associated sequence and set $\underline \sigma = (\sigma_i)_{i \in \Z}
= \mut_{\underline\chi}(\underline v)$. 
Let further $(s, \underline r)$ be the
parameters of the Serre weight $\hat \SW(\underline\ew)$. As usual, we
write $\underline r = (r_0, \ldots, r_{f-1})$ and define $r_i
= r_{i \text{ mod } f}$ for $i$ in $\Z$.
We let $\varepsilon'_0, \ldots, \varepsilon'_{f-1}$ in
$\{0, 1\}$ be the parameters describing the Serre weight $\hat 
\SW(\underline\ew)$ inside $\DD(\ttt)$ (see \S \ref{ssec:serreweights}
for more details).
By Remark~\ref{rem:uniqueepsp}, we know that they are uniquely 
determined. By the proof of Proposition~\ref{prop:SW}, we even 
have a formula for their values. In particular, when $\{\sigma_i,
\sigma_{i+f}\} \neq \{0, p\}$, we have
$\varepsilon'_{f-1-i} = 0$ if $\sigma_i = 0$ or $\sigma_{i+f} = 0$
and $\varepsilon'_{f-1-i} = 1$ otherwise.

We first focus on the parameter $\underline r$:
we fix an index $i$ in $\{0, \ldots, f{-}1\}$ and aim at proving
that $r_{f-1-i}$ is given by the rules of the table of 
Figure~\ref{fig:ri}.
To start with, we consider the case where $X_i = \gO$. 
By Proposition~\ref{prop:admatO}, we know that $\ew_{i-1} \neq 
\lambda(X_{i-1})$ and $\ew_{i+f} = \gb$. These properties allow us
to find the values of $\chi_{i-1}$ and $\chi_i$. Indeed, 
after~\eqref{eq:deltachii}, it is clear that the former one means 
that $\chi_{i-1} = 0$ whereas the latter one gives $\chi_i = w_i$
since $\ew_{i+f} = \lambda(X_i) = \gb$. Consequently $\sigma_i =
v_i - w_i$.
Besides, since $X_{i+f} \neq \gO$ and $\ew_{i+f} = 
\gb$, the proof of Lemma~\ref{lem:enrichedweights} indicates that 
$\sigma_{i+f} = p \chi_{i+f-1}$. 
Applying Lemma~\ref{lem:delta} with the inputs $x_1 = \ew_{i+f-1}$, 
$x_2 = \lambda(X_{i+f-1})$, $y_1 = \ew_{i-1}$ and $y_2 = \lambda(X_{i-1})$, 
we find moreover that the condition $\chi_{i+f-1} = 0$ is equivalent to 
$w_{i-1} = \delta_{i-1}$. Putting all together and coming back to the
definition of $\SW_\SSS$ (see in particular Eq.~\eqref{eq:risigma},
page \pageref{eq:risigma}), we find that:
$$\begin{array}{r@{\hspace{0.5ex}}l@{\qquad}l}
r_{f-1-i} & = v_i - w_i - 1 & \text{if } w_{i-1} = \delta_{i-1},
  \smallskip \\
r_{f-1-i} & = p - 1 - v_i + w_i & \text{if } w_{i-1} \neq \delta_{i-1}
\end{array}$$
as recorded in the table of Figure~\ref{fig:ri}. We have then proved
that the value of $r_{f-1-i}$ is correct when $X_i = \gO$. In this
setting, we can moreover determine the value of 
$\varepsilon'_{f-1-i}$ (which will be useful for later use). 
Indeed, from Lemma~\ref{lem:zeroorp}, we deduce that
$\sigma_i \not\in\{0,p\}$. As a consequence, we can find the value of 
$\varepsilon'_i$ using the recipe we have recalled earlier; in our 
setting, we obtain $\varepsilon'_{f-1-i} = \chi_{i+f-1}$, that is:
$$\begin{array}{r@{\hspace{0.5ex}}l@{\qquad}l}
\varepsilon'_{f-1-i} & = 0 & \text{if } w_{i-1} = \delta_{i-1},
  \smallskip \\
\varepsilon'_{f-1-i} & = 1 & \text{if } w_{i-1} \neq \delta_{i-1}.
\end{array}$$

The case where $X_{i+f} = \gO$ is treated similarly. We then 
move to the case where $X_i \neq \gO$ and $X_{i+f} \neq \gO$.
If $w_i = 1$, we must have $\ew_i = \ew_{i+f} = \gb$ thanks to
Corollary~\ref{cor:notaa}. 
From Lemma~\ref{lem:enrichedweights}, it then follows that
$\{\sigma_i, \sigma_{i+f}\} = \{0, p\}$ and hence $r_{f-1-i} =
p-1$. This again agrees with the table of Figure~\ref{fig:ri}.
It remains to examine the case where $w_i = 0$. By symmetry, one 
may assume that $\ew_i = \ga$ and $\ew_{i+f} = \gb$. In
this setting, we have
$\sigma_i = v_i - \chi_i + p \chi_{i-1}$ and $\sigma_{i+f} =
p\chi_{i+f-1}$. Moreover $v_i$ and $\chi_i$ are both in $\{0,
1\}$ (see Lemma~\ref{lem:vi}) and $v_i - \chi_i$ is not divisible 
by $p$ (see Lemma~\ref{lem:enrichedweights}). Hence, we must have 
$v_i = 1 - \chi_i$ and we get $\sigma_i = 1 - 2\chi_i + p \chi_{i-1}$. 
Since $\sigma_i$ must be in addition between $0$ and $p$, we find that 
$\chi_i = \chi_{i-1}$ necessarily. So $\sigma_i = 1 + (p-2) \chi_{i-1}$.
A simple calculation then shows that $r_{f-1-i} = 0$ if
$\chi_{i-1} = \chi_{i+f-1}$ and $r_{f-1-i} = p{-}2$ otherwise.
Finally, applying Lemma~\ref{lem:delta} with the inputs $x_1 = 
\ew_{i+f-1}$, $x_2 = \lambda(X_{i+f-1})$, $y_1 = \ew_{i-1}$ and $y_2 = 
\lambda(X_{i-1})$, we find that the condition $\chi_{i+f-1} = 
\chi_{i-1}$ is equivalent to $w_{i-1} = \delta_{i-1}$. 
The result of our computations then again agrees with the table of 
Figure~\ref{fig:ri}.

To summarize, we have proved the tuple $\underline r$ is correct
in all cases. It remains to prove that $s$ is correct, \emph{i.e.} 
that it is given by the recipe presented at the beginning of~\S 
\ref{ssec:recipe}. By the second part of Lemma~\ref{lem:alts}, it
is sufficient to show that $\varepsilon_{i_0}$ is correct. This
follows from the computation we have carried out earlier when 
$X_{i_0} = \gO$. When $c_{i_0} \neq \frac{p-1} 2$, this follows
by looking at the table of Figure~\ref{fig:ri}.
\end{proof}

Proposition~\ref{prop:diagcomm} shows that $\SW$ takes its values 
in $\DD(\ttt, \rhobar)$ (since $\comb$ is surjective) on the one 
hand, and that $\SW$ is surjective onto $\DD(\ttt, \rhobar)$ (since 
$\SW_{\hat\WW}$ is surjective) on the other hand.
It then only remains to prove that $\SW$ is injective. It is the
content of the next proposition, which concludes the proof of
Theorem~\ref{thm:main}.

\begin{prop}
\label{prop:SWinj}
The mapping $\SW : \WW(\bX) \to \DD(\ttt,\rhobar)$ is injective. 
\end{prop}

\begin{proof}
We consider $\underline w = (w_i)_{i\in \Z}$ and $\underline w' = 
(w'_i)_{i\in \Z}$ in
$\WW(\bX)$ and assume that $\SW(\underline w) = \SW(\underline w')$.
We denote by $s$ and $\underline r = (r_0, \ldots, r_{f-1})$ the
parameters of this Serre weight.
Let $i$ be in $\{0, \ldots, f{-}1\}$. If $X_i = \gO$ then it follows
from the definition of $\SW$ that:
$$\begin{array}{r@{\hspace{0.5ex}}l@{\qquad}l}
r_{f-1-i} & = v_i - w_i - 1 & \text{if } w_{i-1} = \delta_{i-1},
  \smallskip \\
r_{f-1-i} & = p - 1 - v_i + w_i & \text{if } w_{i-1} \neq \delta_{i-1}
\end{array}$$
where we recall that $\delta_{i-1} = \delta(\lambda(X_{i-1}, X_{i+f-1}))$.
Of course, the same result holds when $w_{i-1}$ and $w_i$ are replaced 
by $w'_{i-1}$ and $w'_i$ respectively. Examining all options, we find 
that, if $w_i \neq w'_i$, we must have $v_i = \frac{p+1}2$ and $w_{i-1} 
\neq w'_{i-1}$ as well.
Coming back to the proof of Proposition~\ref{prop:diagcomm}, we realize 
that the latter condition implies that the $\varepsilon'_{f-1-i}$ 
associated to $w$ and $w'$ differ. This contradicts the fact that
$\SW(\underline w)$ and $\SW(\underline w')$ share the same $s$. 
Hence $w_i = w'_i$ when $X_i = \gO$.

Obviously, the assumption $X_{i+f} = \gO$ leads to the same 
conclusion that $w_i = w'_i$. To finish with, we need to examine
the case where $X_i \neq \gO$ and $X_{i+f} \neq \gO$. In this
situation we have $r_{f-1-i} \in \{w_i (p-1), p-2 + w_i\}$ and
similarly $r_{f-1-i} \in \{w'_i (p-1), p-2 + w'_i\}$. These two
sets must then meet, which is only possible when $w_i = w'_i$.

In conclusion, we have shown that $w_i = w'_i$ in all cases.
Hence $\underline w = \underline w'$ and injectivity is established.
\end{proof}


\appendix

\section{The case of degenerated genes}
\label{app:degenerate}

The object of this first appendix is to extend the results of \S 
\ref{sec: Results} (including Theorem~\ref{thm:main}) to the case where 
the gene $\bX = (X_i)_{i \in \Z}$ is degenerate, \emph{i.e.} when $X_i 
\geq \gO$ for all~$i$.

\subsection{Set of combinatorial weights associated to a gene} 
\label{sssec:withoutO}

In this new setting, fragmentation no longer makes sense and the 
definition of $\WW(\bX)$ needs to be adapted.
Precisely, we now need to define nine recursive sequence 
$(W_i^{\square, \square'})_{-1 \leq i \leq f-1}$ for $\square$ and $\square'$ 
varying in $\{(\gb,\gb), (\ga,\gb), (\gb,\ga)\}$. The initial values 
of these sequences are given by:

$\begin{array}{cr@{\hspace{0.5ex}}l@{\qquad}l}
\bullet
 & W^{\square, \square'}_{-1}
 & = \emptyset
 & \text{if } \square \neq \square' \\
&& = \{(\:)\}
 & \text{otherwise}
\end{array}$

\noindent
where $(\:)$ denotes the empty tuple
(which is the unique element of $\{0,1\}^0$).
The next values (for $0 \leq i \leq f{-}1$) are given by the formulas:

\medskip

$\begin{array}{cr@{\hspace{0.5ex}}l@{\qquad}l}
\bullet
 & W^{\square,(\gb,\gb)}_i 
 & = \big(W^{\square,(\ga,\gb)}_{i-1} \cup W^{\square,(\gb,\ga)}_{i-1}\big) 
     \times \{1\}
 & \text{if } X_{i-1} = X_{i-1+f} \smallskip \\
&& = W^{\square,(\gb,\gb)}_{i-1}
    \times \{1\}
 & \text{otherwise}
\end{array}$

\smallskip

$\begin{array}{cr@{\hspace{0.5ex}}l@{\qquad}l}
\bullet
 & W^{\square,(\ga,\gb)}_i 
 & = W^{\square,(\ga,\gb)}_{i-1}
     \times \{0\}
 & \text{if } X_i = X_{i-1} \smallskip \\
&& = \big(W^{\square,(\gb,\ga)}_{i-1} \cup W^{\square,(\gb,\gb)}_{i-1}\big) 
    \times \{0\}
 & \text{otherwise}
\end{array}$

\smallskip

$\begin{array}{cr@{\hspace{0.5ex}}l@{\qquad}l}
\bullet
 & W^{\square,(\gb,\ga)}_i 
 & = W^{\square,(\gb,\ga)}_{i-1}
     \times \{0\}
 & \text{if } X_{i+f} = X_{i-1+f} \smallskip \\
&& = \big(W^{\square,(\ga,\gb)}_{i-1} \cup W^{\square,(\gb,\gb)}_{i-1}\big) 
    \times \{0\}
 & \text{otherwise.}
\end{array}$

\medskip

\noindent
The set of combinatorial weights of $\bX$ is finally defined by:
$$\WW(\bX) =
  W^{(\gb,\gb), (\gb,\gb)}_{f-1} \, \cup \,
  W^{(\ga,\gb), (\gb,\ga)}_{f-1} \, \cup \,
  W^{(\gb,\ga), (\ga,\gb)}_{f-1}.$$

\begin{ex}
Let us consider the following simple gene:

\medskip

\noindent\hfill%
\begin{tikzpicture}[scale=0.8]
\draw [fill=yellow!20] (-0.5,-0.5) rectangle (1.5,1.5);
\node at (0, 0) { $\gA$ };
\node at (0, 1) { $\gB$ };
\node at (1, 0) { $\gA$ };
\node at (1, 1) { $\gA$ };
\end{tikzpicture}%
\hfill\null

The values of the sequences $(W_i^{\square, \square'})$ are recorded
in the tables of Figure~\ref{fig:weights}.
\begin{figure}
\noindent\hfill%
\begin{tikzpicture}[xscale=2.7, yscale=0.8]
\draw (0,1)--(3,1);
\draw (-1,0)--(3,0);
\draw (-1,-1)--(3,-1);
\draw (-1,-2)--(3,-2);
\draw (-1,-3)--(3,-3);
\draw (-1,0)--(-1,-3);
\draw (0,1)--(0,-3);
\draw (1,1)--(1,-3);
\draw (2,1)--(2,-3);
\draw (3,1)--(3,-3);
\node at (-0.5,0.5) { $W^{\square,\square'}_{-1}$ };
\node at (0.5,0.5) { $\square' = (\gb,\gb)$ };
\node at (1.5,0.5) { $\square' = (\ga,\gb)$ };
\node at (2.5,0.5) { $\square' = (\gb,\ga)$ };
\node at (-0.5,-0.5) { $\square = (\gb,\gb)$ };
\node at (-0.5,-1.5) { $\square = (\ga,\gb)$ };
\node at (-0.5,-2.5) { $\square = (\gb,\ga)$ };
\node at (0.5,-0.5) { $\big\{ (\:) \big\}$ };
\node at (0.5,-1.5) { $\emptyset$ };
\node at (0.5,-2.5) { $\emptyset$ };
\node at (1.5,-0.5) { $\emptyset$ };
\node at (1.5,-1.5) { $\big\{ (\:) \big\}$ };
\node at (1.5,-2.5) { $\emptyset$ };
\node at (2.5,-0.5) { $\emptyset$ };
\node at (2.5,-1.5) { $\emptyset$ };
\node at (2.5,-2.5) { $\big\{ (\:) \big\}$ };
\end{tikzpicture}%
\hfill\null

\vspace{3ex}

\noindent\hfill%
\begin{tikzpicture}[xscale=2.7, yscale=0.8]
\draw (0,1)--(3,1);
\draw (-1,0)--(3,0);
\draw (-1,-1)--(3,-1);
\draw (-1,-2)--(3,-2);
\draw (-1,-3)--(3,-3);
\draw (-1,0)--(-1,-3);
\draw (0,1)--(0,-3);
\draw (1,1)--(1,-3);
\draw (2,1)--(2,-3);
\draw (3,1)--(3,-3);
\node at (-0.5,0.5) { $W^{\square,\square'}_0$ };
\node at (0.5,0.5) { $\square' = (\gb,\gb)$ };
\node at (1.5,0.5) { $\square' = (\ga,\gb)$ };
\node at (2.5,0.5) { $\square' = (\gb,\ga)$ };
\node at (-0.5,-0.5) { $\square = (\gb,\gb)$ };
\node at (-0.5,-1.5) { $\square = (\ga,\gb)$ };
\node at (-0.5,-2.5) { $\square = (\gb,\ga)$ };
\node at (0.5,-0.5) { $\emptyset$ };
\node at (0.5,-1.5) { $\big\{ (1) \big\}$ };
\node at (0.5,-2.5) { $\big\{ (1) \big\}$ };
\node at (1.5,-0.5) { $\big\{ (0) \big\} $ };
\node at (1.5,-1.5) { $\emptyset$ };
\node at (1.5,-2.5) { $\big\{ (0) \big\} $ };
\node at (2.5,-0.5) { $\emptyset$ };
\node at (2.5,-1.5) { $\emptyset$ };
\node at (2.5,-2.5) { $\big\{ (0) \big\}$ };
\end{tikzpicture}%
\hfill\null

\vspace{3ex}

\noindent\hfill%
\begin{tikzpicture}[xscale=2.7, yscale=0.8]
\draw (0,1)--(3,1);
\draw (-1,0)--(3,0);
\draw (-1,-1)--(3,-1);
\draw (-1,-2)--(3,-2);
\draw (-1,-3)--(3,-3);
\draw (-1,0)--(-1,-3);
\draw (0,1)--(0,-3);
\draw (1,1)--(1,-3);
\draw (2,1)--(2,-3);
\draw (3,1)--(3,-3);
\node at (-0.5,0.5) { $W^{\square,\square'}_1$ };
\node at (0.5,0.5) { $\square' = (\gb,\gb)$ };
\node at (1.5,0.5) { $\square' = (\ga,\gb)$ };
\node at (2.5,0.5) { $\square' = (\gb,\ga)$ };
\node at (-0.5,-0.5) { $\square = (\gb,\gb)$ };
\node at (-0.5,-1.5) { $\square = (\ga,\gb)$ };
\node at (-0.5,-2.5) { $\square = (\gb,\ga)$ };
\node at (0.5,-0.5) { $\emptyset$ };
\node at (0.5,-1.5) { $\big\{ (1,1) \big\}$ };
\node at (0.5,-2.5) { $\big\{ (1,1) \big\}$ };
\node at (1.5,-0.5) { $\emptyset$ };
\node at (1.5,-1.5) { $\big\{ (1,0) \big\}$ };
\node at (1.5,-2.5) { $\big\{ (0,0),(1,0) \big\}$ };
\node at (2.5,-0.5) { $\emptyset$ };
\node at (2.5,-1.5) { $\emptyset$ };
\node at (2.5,-2.5) { $\big\{ (0,0) \big\}$ };
\end{tikzpicture}%
\hfill\null
\caption{Computation of a set of combinatorial weights step by step}
\label{fig:weights}
\end{figure}
From this calculation, we find that our gene has $2$ combinatorial
weights, which are $(0,0)$ and $(1,0)$ (both coming from $\square
= (\gb,\ga)$ and $\square' = (\ga,\gb)$).
\end{ex}

 \begin{ex}
The set of combinatorial
weights of the fragment $\underline{F}$:

\medskip

\noindent\hfill%
\begin{tikzpicture}[scale=0.8]
\node at (0, 0) { $\gA$ };
\node at (0, 1) { $\gB$ };
\node at (1, 0) { $\gB$ };
\node at (1, 1) { $\gB$ };
\node at (2, 0) { $\gA$ };
\node at (2, 1) { $\gB$ };
\node at (3, 0) { $\gA$ };
\node at (3, 1) { $\gA$ };
\end{tikzpicture}%
\hfill\null

\medskip

\noindent
is $\WW(\underline{F})=\big\{(0, 0, 0, 0),(0, 0, 1, 0),(0, 0, 1, 1),(1, 
0, 1, 0),(1, 1, 0, 0)\big\}$.
\end{ex}

\subsection{Viability and non-emptiness}

Theorem~\ref{thm:empty} extends to degenerate genes without any
modification.

\begin{thm}
\label{Athm:empty}
Let $\bX$ be a degenerate gene. Then $\WW(\bX)$ is not empty if and 
only if $\bX$ is viable.
\end{thm}

The rest of this subsection is devoted to the proof of
Theorem~\ref{Athm:empty}.
We argue by contradiction, assuming that $\WW(\bX) = \emptyset$. 
Then:
$$W_{f-1}^{(\gb,\gb),(\gb,\gb)} =
W_{f-1}^{(\ga,\gb),(\gb,\ga)} =
W_{f-1}^{(\gb,\ga),(\ga,\gb)} = \emptyset.$$
Moreover, by the third condition of Definition~\ref{def:gene}, 
the sequence $(X_i)_{i \in \Z}$ cannot be constant. As a consequence, 
there exists an index $i \in \{0, \ldots, f{-}1\}$ for which $X_{i-1} 
\neq X_i$ or $X_{i-1+f} \neq X_{i+f}$. Let $j$ be the maximal such 
index. For all $i$ between $j{+}1$ and $f{-}1$, we then have
$W_i^{(\ga,\gb),(\gb,\ga)} = W_{i-1}^{(\ga,\gb),(\gb,\ga)}$ and
$W_i^{(\gb,\ga),(\ga,\gb)} = W_{i-1}^{(\gb,\ga),(\ga,\gb)}$. 
Hence:
$$W_{j}^{(\ga,\gb),(\gb,\ga)} = W_{j}^{(\ga,\gb),(\gb,\ga)} =
\emptyset.$$

\begin{claim}
\label{claim:empty}
There exists 
$\square' \in \{(\gb,\gb),(\ga,\gb),(\gb,\ga)\}$ such that
$W_j^{\square, \square'} = \emptyset$ for all
$\square \in \{(\gb,\gb),(\ga,\gb),(\gb,\ga)\}$.
\end{claim}

In order to prove this claim, we distinguish between two cases.

\paragraph*{First case}

We assume that $X_{i-1} \neq X_{i-1+f}$ for all
$i \in \{j{+}1, \ldots, f{-}1\}$. Under this additional
assumption, we have:
$$W_{j+1}^{(\gb,\gb),(\gb,\gb)} = \cdots =
W_{f-2}^{(\gb,\gb),(\gb,\gb)} = \cdots =
W_{f-1}^{(\gb,\gb),(\gb,\gb)} = \emptyset.$$

If $X_{j-1} \neq X_{j}$ but $X_{j-1+f} = X_{j+f}$, we get
$X_{j-1} = X_{j-1+f}$. Therefore:
\begin{align*}
\emptyset =
W_j^{(\gb,\gb),(\gb,\gb)} 
  & = \big(W_{j-1}^{(\gb,\gb),(\ga,\gb)} \cup
           W_{j-1}^{(\gb,\gb),(\gb,\ga)} \big) \times \{ 1\} \\
\emptyset =
W_j^{(\ga,\gb),(\gb,\ga)} 
  & = W_{j-1}^{(\ga,\gb),(\gb,\ga)} \times \{ 0\} \\
\emptyset =
W_j^{(\gb,\ga),(\ga,\gb)} 
  & = \big(W_{j-1}^{(\gb,\ga),(\gb,\ga)} \cup
           W_{j-1}^{(\ga,\gb),(\gb,\gb)} \big) \times \{ 0\}
\end{align*}
In particular, we deduce that $W_{j-1}^{\square, (\gb,\ga)} = 
\emptyset$ for all $\square \in \{(\gb,\gb),(\ga,\gb),(\gb,\ga)\}$
and our claim is proved.
Similarly, if $X_{j-1} = X_{j}$ but $X_{j-1+f} \neq X_{j+f}$, we 
find that $W_{j-1}^{\square, (\ga,\gb)} = 
\emptyset$ for all $\square \in \{(\gb,\gb),(\ga,\gb),(\gb,\ga)\}$,
which proves the claim as well.

Let us now examine the situation where both
$X_{j-1} \neq X_{j}$ and $X_{j-1+f} \neq X_{j+f}$. In this case,
we have $X_{j-1} \neq X_{j-1+f}$ and thus:
\begin{align*}
\emptyset =
W_j^{(\gb,\gb),(\gb,\gb)}
  & = W_{j-1}^{(\gb,\gb),(\gb,\gb)} \times \{ 1\} \\
\emptyset =
W_j^{(\ga,\gb),(\gb,\ga)}
  & = \big(W_{j-1}^{(\ga,\gb),(\ga,\gb)} \cup
           W_{j-1}^{(\ga,\gb),(\gb,\gb)} \big) \times \{ 0\} \\
\emptyset =
W_j^{(\gb,\ga),(\ga,\gb)}
  & = \big(W_{j-1}^{(\gb,\ga),(\gb,\ga)} \cup
           W_{j-1}^{(\ga,\gb),(\gb,\gb)} \big) \times \{ 0\}
\end{align*}
and the claim holds with $\square' = (\gb,\gb)$.

\paragraph*{Second case}

Here, we assume that there exists an integer $j' \in 
\{j{+}1, \ldots, f{-}1\}$ such that $X_{j'-1} = X_{j'-1+f}$. We 
choose $j'$ maximal. Then:
$$W_{j'+1}^{(\gb,\gb),(\gb,\gb)} = \cdots =
W_{f-1}^{(\gb,\gb),(\gb,\gb)} = \emptyset$$
and:
$$W_{j'+1}^{(\gb,\gb),(\gb,\gb)}
  = \big(W_{j'}^{(\gb,\gb),(\ga,\gb)} \cup
         W_{j'}^{(\gb,\gb),(\gb,\ga)} \big) \times \{ 1\}.$$
Consequently $W_{j'}^{(\gb,\gb),(\ga,\gb)} = W_{j'}^{(\gb,\gb),
(\gb,\ga)} = \emptyset$, from what we deduce by decreasing induction
that:
$$W_{j}^{(\gb,\gb),(\ga,\gb)} = W_{j}^{(\gb,\gb),
(\gb,\ga)} = \emptyset.$$
Now, as in first case, we examine the three subcases depending on
the truth values of the assertions ``$X_{j-1} = X_j$'' and 
``$X_{j-1+f} = X_{j+f}$'' and conclude in all settings that
our claim indeed holds.

\paragraph*{End of the proof}

We observe that, if Claim~\ref{claim:empty} holds for the index $j$, then 
it also holds for the index $j{-}1$. Indeed, either it holds for the 
same $\square'$ if $W_j^{\square,\square'} = W_{j-1}^{\square,\square'} 
\times \{\star\}$, or it holds for $\square'_1$ and $\square'_2$ if
$W_j^{\square,\square'} = \big(W_{j-1}^{\square,\square'_1} \cup 
W_{j-1}^{\square,\square'_2}\big)\times \{\star\}$
with $\{\square', \square'_1, \square'_2\} = \{(\gb,\gb),(\ga,\gb),
(\gb,\ga)\}$.

By descending induction we deduce that Claim~\ref{claim:empty} holds
with $j = -1$, which contradicts the definition of the
$W_{-1}^{\square,\square'}$'s. Theorem~\ref{Athm:empty} is proved.

\subsection{Counting weights}
\label{ssec:countdegenerate}

We now aim at extending Corollary~\ref{cor:countwithO} to degenerate 
genes. Let then $\bX$ be a degenerate gene and
define $c_i^{\square, 
\square'} = \Card W_i^{\square, \square'}$ when $\square$ and
$\square'$ are elements of $\{(\gb,\gb), (\ga,\gb), (\gb,\ga)\}$.
As in \S \ref{ssec:count}, we have recurrence relations
from which we easily derive the values of the
$c_i^{\square, \square'}$'s. Unfortunately, this is not sufficient to
calculate the cardinality of $\WW(\bX)$ because the union defining it,
namely
$$\WW(\bX) =
  W^{(\gb,\gb), (\gb,\gb)}_{f-1} \, \cup \,
  W^{(\ga,\gb), (\gb,\ga)}_{f-1} \, \cup \,
  W^{(\gb,\ga), (\ga,\gb)}_{f-1}.$$
\noindent
is usually \emph{not} a disjoint union.
To tackle this issue, we introduce more sequences: given 
$\square, \square', \Delta, \Delta' \in \{(\gb,\gb), (\ga,\gb), 
(\gb,\ga)\}$, we set:
$$c_i^{\square, \square', \Delta, \Delta'} = 
\Card \big(W_i^{\square, \square'} \cup W_i^{\Delta, \Delta'}\big).$$
It then turns out that the following set of $14$ sequences:
$$\begin{array}{l@{\qquad}l}
\big(c_i^{\square, \square'}\big)_{-1 \leq i \leq f-1} 
  & \text{for } \square, \square' \in \{(\gb,\gb), (\ga,\gb), (\gb,\ga)\} \medskip \\
\big(c_i^{(\ga,\gb),\square', (\gb,\ga), \Delta'}\big)_{-1 \leq i \leq f-1} 
  & \text{for } \square', \Delta' \in \{(\ga,\gb), (\gb,\ga)\} \medskip \\
\big(c_i^{(\ga,\gb),(\gb,\gb), (\gb,\ga), (\gb,\gb)}\big)_{-1 \leq i \leq f-1}
\end{array}$$
satisfy a full collection of recurrence relations, allowing for their complete
determination. For example, from the equalities:
$$\begin{array}{r@{\hspace{0.5ex}}l@{\qquad}l}
W_i^{(\ga,\gb),(\gb,\gb)} \cup W_i^{(\gb,\ga), (\gb,\gb)}
  & = \big(W_{i-1}^{(\ga,\gb),(\ga,\gb)} \cup W_{i-1}^{(\ga,\gb), (\gb,\ga)} \\
  & \hspace{5ex}{}\cup W_{i-1}^{(\gb,\ga),(\ga,\gb)} \cup W_{i-1}^{(\gb,\ga), (\gb,\ga)}\big) \times \{1\}
  & \text{if } X_{i-1} = X_{i-1+f}
\smallskip \\
  & = \big(W_{i-1}^{(\ga,\gb),(\gb,\gb)} \cup W_{i-1}^{(\gb,\ga), (\gb,\gb)}\big) \times \{1\}
  & \text{if } X_{i-1} \neq X_{i-1+f}.
\end{array}$$
we derive using a straightforward analogue of Lemma~\ref{lem:inclweight}:
$$\begin{array}{r@{\hspace{0.5ex}}l@{\qquad}l}
c_i^{(\ga,\gb),(\gb,\gb), (\gb,\ga), (\gb,\gb)}
  & = \max \, c_{i-1}^{(\ga,\gb),\square', (\gb,\ga), \Delta'}
  & \text{if } X_{i-1} = X_{i-1+f}
\smallskip \\
  & = c_{i-1}^{(\ga,\gb),(\gb,\gb), (\gb,\ga), (\gb,\gb)}
  & \text{if } X_{i-1} \neq X_{i-1+f}.
\end{array}$$
where the maximum is taken over $\square'$ and $\Delta'$ varying
in $\{(\ga,\gb), (\gb,\ga)\}$.
In a similar fashion, we obtain recurrence relations for the $13$ 
other sequences. Note that for $c_i^{(\ga,\gb),\square', (\gb,\ga), 
\Delta'}$, we have to distinguish between four cases depending on the
values of $X_i$, $X_{i-1}$, $X_{i+f}$ and $X_{i-1+f}$.
The cardinality of $\WW(\bX)$ is finally given by the formula:
\begin{equation}\label{equacardwithoutO}
\Card \WW(\bX) = 
c_{f-1}^{(\gb,\gb),(\gb,\gb)} + 
c_{f-1}^{(\ga,\gb),(\gb,\ga), (\gb,\ga), (\ga,\gb)}
\end{equation}
since the two corresponding set of weights are now disjoint.

As a summary, although the computation becomes more sophisticated and
unpleasant, it remains possible to have access to the cardinality of 
$\WW(\bX)$ without computing the set $\WW(\bX)$ itself. 
This conclusion will be important when we will design fast algorithms 
for enumerating and counting elements of $\WW(\bX)$ in \S 
\ref{sssec:algocount}.

Previous constructions are also important for deriving the next
theorem.

\begin{thm}
\label{Athm:fibowithoutO}
Let $\bX$ be a gene of length $f$ without any occurrence of $\gO$.
Then $\Card \WW(\bX) < \Fib_{f+2}$.
\end{thm}

\begin{proof}
One checks by induction on $i$ that:
$$\begin{array}{r@{\hspace{0.5ex}}l@{\quad;\quad}r@{\hspace{0.5ex}}l@{\quad;\quad}r@{\hspace{0.5ex}}l}
c_i^{\square,(\gb,\gb)} &\leq \Fib_{i+1} &
c_i^{(\gb,\gb), \square'} &\leq \Fib_{i+1} &
c_i^{(\ga,\gb),(\gb,\gb), (\gb,\ga), (\gb,\gb)} &\leq \Fib_{i+1} \smallskip \\
c_i^{(\gb,\gb),(\gb,\gb)} &\leq \Fib_{i+2} &
c_i^{\square, \square'} &\leq \Fib_{i+2} &
c_i^{(\ga,\gb),\square', (\gb,\ga), \Delta'} &\leq \Fib_{i+2} 
\end{array}$$
for $i \in \{-1, 0, \ldots, f{-}1\}$ and $\square, \square',
\Delta' \in \{(\ga,\gb), (\gb,\ga)\}$. Knowing this, we deduce that
$\Card \WW(\bX) \leq \Fib_{f+1} + \Fib_f = \Fib_{f+2}$. The fact that the
inequality is strict is proved by driving out the equality cases; 
we left it to the reader.
\end{proof}

\subsection{Combinatorial weights and Serre weights}

We fix an irreducible representation $\rhobar : G_F \to \GL_2(k_E)$
and a tame inertial type $\ttt : I_F \to \GL_2(\oE)$ and choose a
coherent triple $(h, \gamma, \gamma')$ parametrizing them (see
Definition~\ref{def:triple}). We let $\bX = (X_i)_{i \in \Z}$ be
the gene associated to $(h, \gamma, \gamma')$. In what follows, we
do not make any assumption on $\bX$; it can be degenerate or not.

To any combinatorial weight $\underline w \in \WW(\bX)$, one can 
associate a Serre weight $\SW(\underline w)$ as in \S 
\ref{ssec:recipe}. For this, we follow exactly the same recipe;
the only difference is that we now need to pay attention to the
fact that the two cases we have considered for defining~$s$
exhaust all possibilites. This is the content of the following 
lemma.

\begin{lem}
With the notations of \S \ref{ssec:recipe}, 
if $c_i = \frac{p-1}2$ for all $i$, the gene $\bX$ is nondegenerate.
\end{lem}

\begin{proof}
The assumption means that $\gamma' = \gamma + \frac{q-1} 2$.
Using in addition that the triple $(h, \gamma, \gamma')$ is coherent,
we deduce that:
$$h - (q{+}1)\gamma' \equiv \frac{q-1}2 + \frac{q-1}{p-1} 
\pmod {q-1}.$$
Moreover, we derive from the definition of the $v_i$'s
that:
\begin{equation}\label{defiSs}
h - (q{+}1)\gamma' \equiv \sum_{i=0}^{f-1} (v_i + v_{i+f})\:p^{f-1-i}
\pmod {q-1}.\end{equation}
If the gene $\bX$ did not contain any occurrence of $\gO$, we would
derive from Lemma~\ref{lem:vi} that $v_i \in \{0, 1\}$ for all $i$.
Therefore, comparing the two congruences above, we would get
$v_i + v_{i+f} = \frac{p+1} 2$ for all $i$. Using again that
$v_i \in \{0, 1\}$ for all $i$, we would deduce that $p = 3$ and 
$v_i = 1$ for all $i$. Applying Lemma~\ref{lem:vi} again, we would 
finally find that $X_i = \gB$ for all $i$, which contradicts the 
definition of a gene.
\end{proof}

After this remark, Theorem~\ref{thm:main} extends \emph{verbatim}.

\begin{thm}
\label{Athm:main}
The construction $\SW$ induces a bijection $\WW(\bX) 
\stackrel\sim\longrightarrow \DD(\ttt,\rhobar)$.
\end{thm}

\begin{proof}
It is exactly the same that that of Theorem~\ref{thm:main}
(see \S \ref{sec:proof}), except that, when $\bX$ is degenerate,
the developments carried out in \S \ref{sssec:descactive} are no 
longer needed (actually, they do not make sense) and we
need to invoke Proposition~\ref{prop:admoutsideO} instead of
Propositions~\ref{prop:cartesian} and~\ref{prop:fragmentweights}
in order to guarantee that $\comb$ maps $\hat\WW(h,\gamma,\gamma')$
to $\WW(\bX)$.
\end{proof}

\section{Algorithms}
\label{sec:algo}

In this appendix, we discuss algorithmical solutions for manipulating
the mathematical objects considered in this article: irreductible
$2$-dimensional mod $p$ Galois representations, tamely ramified
Galois types, Serre weights and, of course, genes and combinatorial
weights.

We always encode a Galois representation
$$\rhobar \,\simeq\,
\Ind_{G_{F'}}^{G_F}\big(\omega_{2f}^h \otimes \nr'(\theta)\big)$$
by the tuple $(p, f, h, \theta)$ with $0 \leq h < q^2 - 1$. Although
the tuples $(p, f, h, \theta)$ and $(p, f, p^f h, \theta)$ correspond
to isomorphic Galois representations, we make the distinction
between them. The main reason is that the definition of the gene is
sensible to the choice of $h$. In what follows, we shall always assume
that $h$ is given by its sequence of digits in base $p$.

In the same fashion, we encode a Galois type $\ttt = \omega_f^\gamma
\oplus \omega_f^{\gamma'}$ by the quadruple $(p, f, \gamma, \gamma')$.
As before, we make the distinction between $(p, f, \gamma,
\gamma')$ and $(p, f, \gamma', \gamma)$ and we assume that $\gamma$
and $\gamma'$ are given by their sequences of digits in radix $p$.

In what follows, we estimate the efficiency of our algorithms
by bounding their bit complexity defined as the number of operations
on bits they perform.
For example, the bit complexity of an addition on integers less 
than $p$ is $O(\log p)$. For integers written in base $p$ with at
most $n$ digits, it is $O(n \log p)$.

\subsection{About genes}
\label{ssec:algogene}

In this first subsection, we focus on genes.
We give fast algorithms for computing the gene associated to a
pair $(\ttt, \rhobar)$ and conversely, given a gene $\bX$,
we design an algorithm that samples a uniformely distributed pair
$(\ttt, \rhobar)$ having gene $\bX$.

\subsubsection{Preliminaries}

We consider a coherent triple in the sense of 
Definition~\ref{def:triple} and denote by $\bX = (X_i)_{i\in\Z}$
its associated gene. We let $(v_i)_{i\in\Z}$ be the sequence
introduced at the end of \S \ref{ssec:defgene}.
We recall that it is defined by the fact that it is $(2f)$-periodic,
 takes values in $\{0, \ldots, p{-}1\}$ and makes the 
congruence~\eqref{eq:vi} hold.

We have already seen in Lemma~\ref{lem:vi}, that the $v_i$'s are
closely related to the gene $\bX$. In what follows, we make
these relationships even tighter.
The forthcoming results are the key for designing our algorithms.

\begin{lem}
\label{lem:equivO}
For an integer $i\in \Z$, we have $X_i = \gO$ if and only if:
$$\sum_{j=0}^{2f-1} p^{2f-1-j} v_{i+j}
\geq \frac{q^2 - 1}{p-1}.$$
\end{lem}

\begin{proof}
From the
definition of the $v_i$'s and the $\alpha_i$'s, we derive that
$$\alpha_i = \left\lfloor \frac 1{q+1} \cdot
\sum_{j=0}^{2f-1} p^{2f-1-j} v_{i+j} \right\rfloor.$$
By definition, $X_i = \gO$ if and only if 
$\alpha_i \geq \frac{q-1}{p-1}$. Plugging the value of $\alpha_i$
and noticing that $\frac{q-1}{p-1}$ is an integer, we find that 
$X_i = \gO$ if and only if
$$\frac 1{q+1} \cdot \sum_{j=0}^{2f-1} p^{2f-1-j} v_{i+j}
\geq \frac{q-1}{p-1}$$
which proves the lemma.
\end{proof}

\begin{cor}
\label{cor:vi}
We assume that $v_i \in \{0,1\}$ for all $i$.
Then $X_i \neq \gO$ for all $i$.
\end{cor}

\begin{proof}
We fix an index $i$. It follows from the assumption that:
\begin{equation}
\label{eq:vi01}
\sum_{j=0}^{2f-1} p^{2f-1-j} v_{i+j} \leq \frac{q^2 - 1}{p-1}.
\end{equation}
Besides, the equality case occurs if and only if $v_i = 1$ for all~$i$.
But this cannot happen because if would imply that $h$ is divisible by
$q{+}1$, which is assumed to be false. 
Therefore, the inequality~\eqref{eq:vi01} is always strict and we 
conclude by Lemma~\ref{lem:equivO}.
\end{proof}

\begin{lem}
\label{lem:vi2}
For any integer $i$ in $\Z$, the following holds:
\begin{myenumerate}[(0x)]
\item[(0a)] if $v_i = 0$ and $X_{i+1} = \gO$, then $X_i = \gAB$;
\item[(0b)] if $v_i = 0$ and $X_{i+1} \neq \gO$, then $X_i = \gA$;
\item[(1a)] if $v_i = 1$ and $X_{i+1} = \gO$, then $X_i = \gO$;
\item[(1b)] if $v_i = 1$ and $X_{i+1} \neq \gO$, then $X_i = \gB$;
\item[(2)] if $v_i \geq 2$, then $X_i = \gO$.
\end{myenumerate}
\end{lem}

\begin{proof}
We first assume $v_i = 0$. From Lemma~\ref{lem:vi}, we deduce that
$X_i \in \{\gA, \gAB\}$. Moreover, if $X_{i+1} = \gO$, we
know that $X_i \in \{\gAB, \gO\}$; hence $X_i = \gAB$ and we have
proved~(0a). Conversely, from $X_i = \gAB$ we deduce $X_{i+1} =
\gO$, which proves~(0b).

We assume $v_i = 1$ and $X_{i+1} = \gO$. 
By Lemma~\ref{lem:vi}, we find $X_i \in \{\gB, \gO\}$. However
$X_i = \gB$ is not compatible with $X_{i+1} = \gO$. Therefore
$X_i = \gO$, which proves~(1a).

We consider the case where $v_i = 1$ and $X_{i+1} \neq \gO$.
As before, we have $X_i \in \{\gB, \gO\}$.
Applying Lemma~\ref{lem:equivO}
with $i{+}1$, we obtain:
$$\sum_{j=0}^{2f-1} p^{2f-1-j} v_{i+j+1}
\geq \frac{q^2 - 1}{p-1}.$$
A simple computation using the fact that $v_i = 1$ gives:
\begin{align*}
p \sum_{j=0}^{2f-1} p^{2f-1-j} v_{i+j} 
 & = q^2 - 1 + \sum_{j=0}^{2f-1} p^{2f-1-j} v_{i+j} \\
 & \geq q^2 + 1 + \frac{q^2 - 1}{p-1} = p \cdot \frac{q^2-1}{p-1}.
\end{align*}
Applying again Lemma~\ref{lem:equivO}, we find $X_i \neq \gO$.
Therefore $X_i = \gB$ and (1b)~is proved.

Finally, the last assertion follows directly from Lemma~\ref{lem:vi}.
\end{proof}

\subsubsection{Computation of the gene}

As above, we consider a coherent triple $(h, \gamma, \gamma')$ and aim 
at designing a fast algorithm for computing its associated gene. 
Before proceeding, let us observe that the naive algorithm (consisting 
in computing the $\alpha_i$'s and spotting in which intervals they 
fall) has quadratic complexity in $f$. Indeed the 
computation of a single $\alpha_i$ requires $O(f)$ operations on 
integers less than $p$ and this calculation needs to be repeated $f$ 
times.
However, after Lemma~\ref{lem:vi2}, it becomes possible to 
significantly speed up this computation.

\begin{thm}
\label{thm:algogene}
There exists an algorithm that takes as input a coherent triple
$(h, \gamma, \gamma')$ (with $h$, $\gamma$ and $\gamma'$ written in
base $p$) and outputs its gene $(X_0, \ldots, X_{2f-1})$ for a cost 
of $O(f \log p)$ bit operations.
\end{thm}

\begin{rem}
If the inputs $h$, $\gamma$ and $\gamma'$ are not written in base 
$p$, it is of course always possible to compute these writings as
a preliminary. However this calculations requires a number of bit
operations which does not stay in $O(f \log p)$, although fast
multiplication techniques allow for quasi-linear algorithms in
$f \log p$.
\end{rem}

\begin{proof}[Proof of Theorem~\ref{thm:algogene}]
The main ingredients of our algorithm are the formulas of
Corollary~\ref{cor:vi} and Lemma~\ref{lem:vi2}, which connect
the $X_i$'s to the values of the $v_i$'s.

As a first step, we then compute the $v_i$'s. For this, it is
enough to compute 
$$(h - q\gamma' - \gamma') \mod q^2-1$$ 
and observe its digits in base $p$. We claim that, if we put all 
operations in base $p$, the whole computation can be achieved for a 
cost of $O(f \log p)$ bit operations. Indeed, additions can be done 
with this complexity (using the naive algorithm), whereas reduction
modulo $q^2{-}1$ and multiplication boils down to reorganizing 
the digits.

Once the $v_i$'s have been computed, the second step of our
algorithm consists in checking if all of them are in $\{0,1\}$
(which can be obviously done in the desired complexity).
If this occurs, we know from Corollary~\ref{cor:vi}, that the
$X_i$'s are all different from $\gO$. 
Applying now Lemma~\ref{lem:vi2}, we find that $X_i = \gA$ when
$v_i = 0$ and $X_i = \gB$ when $v_i = 1$. This completes the
computation of the gene.

To conclude with, we have to consider the case where we have
found an index $i_0$ with $v_{i_0} \geq 2$. In this situation, 
it follows from Lemma~\ref{lem:vi2} that $X_{i_0} = \gO$. Then,
applying again Lemma~\ref{lem:vi2}, we can deduce the value of
$X_{i_0-1}$: it is $\gAB$ if $v_{i_0-1}$ vanishes and $\gO$
otherwise. Repeating this procedure again and again, we find
iteratively all the values of the $X_i$'s for a cost which
remains in $O(f\log p)$ bit operations.
\end{proof}

\begin{ex}
\label{ex:algogene}
We use the algorithm described above to compute the gene of
the triple $(h, \gamma,\gamma')$ of Example~\ref{ex:rtweights}.
As seen in this example, the values of the $v_i$'s are
$(v_0, \ldots, v_{13}) = (4, 0, 1, 0, 0, 3, 0, 1, 0, 0, 4, 2, 1, 0)$.
We observe that there do exist indices $i_0$ with $v_{i_0} \geq 2$, 
\emph{e.g.} $i_0 = 0$. We thus have $X_0 = \gO$.
Now, applying Lemma~\ref{lem:vi2} with $i = 13$, we obtain 
$X_{13} = \gAB$. Continuing this way, we find $X_{12} = \gB$,
$X_{11} = \gO$, $X_{10} = \gO$, $X_9 = \gAB$, $X_8 = \gA$, 
\emph{etc.} 
Finally, we discover the gene of $(h, \gamma, \gamma')$ which is:

\medskip

\noindent\hfill%
\begin{tikzpicture}[scale=0.8]
\draw [fill=yellow!20] (-0.5,-0.5) rectangle (6.5,1.5);
\node at (0, 0) { $\gB$ };
\node at (0, 1) { $\gO$ };
\node at (1, 0) { $\gA$ };
\node at (1, 1) { $\gA$ };
\node at (2, 0) { $\gAB$ };
\node at (2, 1) { $\gB$ };
\node at (3, 0) { $\gO$ };
\node at (3, 1) { $\gA$ };
\node at (4, 0) { $\gO$ };
\node at (4, 1) { $\gAB$ };
\node at (5, 0) { $\gB$ };
\node at (5, 1) { $\gO$ };
\node at (6, 0) { $\gAB$ };
\node at (6, 1) { $\gA$ };
\end{tikzpicture}%
\hfill\null

\medskip

\noindent
that is exactly the gene of Example~\ref{ex:fragments} as we claimed.
\end{ex}

\subsubsection{Sampling $(h,\gamma,\gamma')$ with a prescribed gene}

Conversely, Lemmas~\ref{lem:vi} and~\ref{lem:vi2} allows for finding 
all the coherent triples $(h, \gamma, \gamma')$ associated to a given 
gene $\bX$. Indeed, they together readily imply the following
proposition.

\begin{prop}
\label{prop:vi}
Let $(h, \gamma, \gamma')$ be a coherent triple and let 
$(v_i)_{i \in \Z}$ be its associated sequence.
Then, the gene of $(h, \gamma, 
\gamma')$ is $\bX = (X_i)_{i\in \Z}$ if and only if the following 
conditions hold for all $i$:

\begin{myenumerate}[(xx)]
\item[(a)] 
  if $X_i = \gA$, then $v_i = 0$;
\item[(ab)] 
  if $X_i = \gAB$, then $v_i = 0$;
\item[(b)]
  if $X_i = \gB$, then $v_i = 1$;
\item[(oo)]
  if $X_i = \gO$ and $X_{i+1} = \gO$, then $1 \leq v_i \leq p{-}1$;
\item[(o$\star$)]
  if $X_i = \gO$ and $X_{i+1} \neq \gO$, then $2 \leq v_i \leq p{-}1$.
\end{myenumerate}
\end{prop}

\begin{thm}
There exists a Las Vegas algorithm that takes as input a gene $\bX$ and 
outputs a coherent triple $(h, \gamma, \gamma')$, uniformly distributed 
among all possibilities, with gene $\bX$.
This algorithm fails with probability at most $\frac 1{p^f} + 
\min(\frac 1 2, \frac 1{p-2})$ and performs at most $O(f \log p)$ 
bit operations.
\end{thm}

\begin{proof}
We start by sampling a sequence $v_0, \ldots, v_{2f-1}$ satisfying 
the requirements of Proposition~\ref{prop:vi}.
We then sample an integer $\gamma'$, uniformly distributed in the
range $\{0, 1, \ldots, q-2\}$, by sampling independently its $f$ 
digits in base $p$ and rejecting the value if all digits are $p{-}1$. 

If the integer $v = \sum_{i=0}^{2f-1} v_i p^{2f-1-i}$
is divisible by $q{+}1$, we reject the value and the algorithm
fails, except if there was only one possibility for the $v_i$'s in
which case the algorithm raises an error and answers that there is
no coherent triple whose gene is $\bX$.
(By Proposition~\ref{prop:gene}, this case can show up only when
$p = 3$.)

If $v$ is not divisible by $q{+}1$, the algorithm
computes:
\begin{align*}
h & = v - (q+1)\gamma' \mod {q^2-1} \\
\text{and} \qquad
\gamma & = (1 + p + \cdots + p^{f-1}) - \gamma' - h \mod {q-1}
\end{align*}
and outputs $(h, \gamma, \gamma')$.
All the previous computations can be done for a cost of $O(f \log p)$ 
bit operations (by writing down all operations in base $p$). 
Moreover, it is clear after Proposition~\ref{prop:vi} 
that the output is a coherent triple with gene $\bX$ and that it is 
uniformly distributed.

It then only remains to bound the probability of failure of our
algorithm. Note that failures can happen in
two places. First, it happens if all the digits of $\gamma'$ are
$p{-}1$; this case occurs with probability $\frac 1{p^f}$. 

The second source of failure occurs when $h$ is divisible by $q{+}1$,
which is equivalent to the fact that $v_i = v_{i+f}$ for all $i$.
Remember that, in this special situation, we know moreover that 
there are multiple choices for 
the $v_i$'s. This means that there exists a particular index 
$i_0$ for which $v_{i_0}$ can take at least two values. Coming
back to the definition, we find more precisely that it can take at 
least $n$ values
with $n = \max(2, p{-}2)$. Since modifying the value of $v_{i_0}$, 
while keeping the $v_i$'s unchanged for all $i \neq i_0$, leads to 
another acceptable set 
of values of the $v_i$'s, we conclude that the probability to
have $v_i = v_{i+f}$ for all $i$ is at most $\frac 1 n$.
The probability of rejection at this second place is then at most
$\frac 1 n$.

Adding to it the first probability of rejection we
found, we deduce that our algorithm fails with probability at most 
$\frac 1{p^f} + \frac 1 n = \frac 1{p^f} + \min(\frac 1 2, 
\frac 1{p{-}2})$, as wanted.
\end{proof}

\subsection{About combinatorial weights}

In this subsection, we assume that we are given a gene $\bX$ and
we aim at designing efficient algorithms for describing its set 
of combinatorial weights $\WW(\bX)$. 
Combining this with Theorem~\ref{thm:algogene} and the recipe
of~\S \ref{ssec:recipe}, we end up with fast algorithms 
for the computation of $\DD(\ttt, \rhobar)$.

\subsubsection{Computing and enumerating weights}
\label{sssec:algoweights}

The first question we address is the complete computation of the
set $\WW(\bX)$. In what follows, we shall prove the following
theorem.

\begin{thm}
\label{thm:algoweights}
There exists an algorithm which takes as input a gene $\bX$ of
length~$f$ and outputs the set $\WW(\bX)$ for a cost of 
$O\big(f + f \cdot \Card \WW(\bX) \big)$ bit operations.
\end{thm}

Given that the bit size of $\WW(\bX)$ is obviously $f \cdot \Card 
\WW(\bX)$ (since each combinatorial weight consists of $f$ bits),
the complexity announced in Theorem~\ref{thm:algoweights} is 
optimal up to a constant factor.

We now concentrate on the proof of Theorem~\ref{thm:algoweights} and 
the description of the underlying algorithm. 
One first checks whether the input gene $\bX$ is viable or not, which
can obviously be done in $O(f)$ bit operations. If $\bX$ is not viable
we output the empty set. From now on, we then assume that $\bX$ is viable.
For the sake of simplicity, 
we only consider the case where $\bX$ is nondegenerate, the opposite 
case being similar (but more technical). 
Keeping in mind the definition of $\WW(\bX)$ (see 
Definition~\ref{def:geneweight}), it is enough to explain how to compute 
$\WW(\underline F)$ for a fragment $\underline F$.

For this, we rely on the definitions and on Lemma~\ref{lem:inclweight}
which guarantees that the unions appearing in the definition of the 
$W_i^{(\ga,\gb)}$'s and $W_i^{(\gb,\ga)}$'s are all disjoint unions, 
whereas the unions of type $W_i^{(\ga,\gb)} \cup W_i^{(\gb,\ga)}$
are in fact supremums.
Looking at the proof of this lemma, one can even figure our for 
which indices $i$, one has $W_i^{(\ga,\gb)} \subset W_i^{(\gb,\ga)}$
(and so $W_i^{(\ga,\gb)} \cup W_i^{(\gb,\ga)} = W_i^{(\gb,\ga)}$)
and for which indices $i$, the inclusion goes in the reverse
direction.

\begin{algorithm}
  \caption{Computation of combinatorial weights}
  \label{algo:weights}%
  \SetKwInOut{Input}{Input} %
  \SetKwInOut{Output}{Output} %
  \SetKwProg{W}{\tt combinatorial\_weights}{}{}%
  \SetKwProg{Wab}{\tt Wab}{}{}%
  \SetKwProg{Wba}{\tt Wba}{}{}%
  \SetKwProg{Wbb}{\tt Wbb}{}{}%
  \newcommand{\Wxy}{\textrm{\tt Wxy}}
  \DontPrintSemicolon

  \smallskip

  \textbf{Global variable:} $\Wxy$

  \smallskip

  \W{$(\underline F)$}{

    \smallskip

    \Input {A fragment $\underline F$}
    \Output{The set $\WW(\underline F)$}

    \medskip

    $\ell \leftarrow \text{length of } \underline F$\;
    \lIf{$F^\up_0 = \gO$}{$\Wxy \leftarrow \texttt{[ Wab ]}$}
    \lIf{$F^\down_0 = \gO$}{$\Wxy \leftarrow \texttt{[ Wba ]}$}
    \For{$i = 1, 2, \ldots, \ell-1$}
      {\lIf{$F^\up_{i-1} \sim F^\up_i$ and $F^\down_{i-1} \sim F^\down_i$}
         {$\Wxy[i] \leftarrow \Wxy[i{-}1]$}
       \lIf{$F^\up_{i-1} \sim F^\up_i$ and $F^\down_{i-1} \not\sim F^\down_i$}
         {$\Wxy[i] \leftarrow \texttt{Wba}$}
       \lIf{$F^\up_{i-1} \not\sim F^\up_i$ and $F^\down_{i-1} \sim F^\down_i$}
         {$\Wxy[i] \leftarrow \texttt{Wab}$}
       \If{$F^\up_{i-1} \not\sim F^\up_i$ and $F^\down_{i-1} \not\sim F^\down_i$}
         {\lIf{$\Wxy[i{-}1] = \textrm{\tt Wab}$}
            {$\Wxy[i] \leftarrow \texttt{Wba}$
          \textbf{else}
             $\Wxy[i] \leftarrow \texttt{Wab}$}}
      }

    \smallskip

    \lIf{$F^\up_{\ell-1} = \gAB$}
      {\KwRet $\texttt{Wab}(\underline F, \ell{-}1) 
             + \texttt{Wbb}(\underline F, \ell{-}1)$}
    \lIf{$F^\down_{\ell-1} = \gAB$}
      {\KwRet $\texttt{Wba}(\underline F, \ell{-}1) 
             + \texttt{Wbb}(\underline F, \ell{-}1)$}
    \KwRet $\texttt{[ }(0)\texttt{ ]}$
  }

  \bigskip

  \Wab{$(\underline F, i)$}{
    \If{$i = 0$}
      {\lIf{$F^\down_0 = \gO$}{\KwRet $\texttt{[ ]}$ }
       \lElse{\KwRet $\texttt{[ }(0,0,\ldots, 0)\texttt{ ]}$}}
    \lIf{$F^\up_{i-1} \sim F^\up_i$}
      {\KwRet $\texttt{Wab}(\underline F, i{-}1)$}
    \lIf{$F^\up_{i-1} \not\sim F^\up_i$}
      {\KwRet $\texttt{Wba}(\underline F, i{-}1) 
             + \texttt{Wbb}(\underline F, i{-}1)$}
  }

  \medskip

  \Wba{$(\underline F, i)$}{
    \If{$i = 0$}
      {\lIf{$F^\up_0 = \gO$}{\KwRet $\texttt{[ ]}$ }
       \lElse{\KwRet $\texttt{[ }(0,0,\ldots, 0)\texttt{ ]}$}}
    \lIf{$F^\down_{i-1} \sim F^\down_i$}
      {\KwRet $\texttt{Wba}(\underline F, i{-}1)$}
    \lIf{$F^\down_{i-1} \not\sim F^\down_i$}
      {\KwRet $\texttt{Wab}(\underline F, i{-}1) 
             + \texttt{Wbb}(\underline F, i{-}1)$}
  }

  \medskip

  \Wbb{$(\underline F, i)$}{
    \lIf{$i = 0$}
      {\KwRet $\texttt{[ }(1,0,\ldots, 0)\texttt{ ]}$}
    \lIf{$F^\up_{i-1} \sim F^\down_{i-1}$}
       {$W \leftarrow \Wxy[i{-}1](\underline F, i{-}1)$}
    \lIf{$F^\up_{i-1} \not\sim F^\down_{i-1}$}
      {$W \leftarrow \texttt{Wbb}(\underline F, i{-}1)$}
    \lFor{$w \in W$}{ $w[i] \leftarrow 1$ }
    \KwRet $W$
  }
\end{algorithm}

Translating these definitions and observations into algorithms, we end 
up with the procedure \texttt{combinatorial\_weights}, which calls the 
recursive subroutines \texttt{Wab}, \texttt{Wba}, \texttt{Wbb} presented 
in Algorithm~\ref{algo:weights}.
In this implementation, the sets of weights are represented by lists 
and the addition of lists means concatenation. The notation
\texttt{[ ]} refers to the empty list.

The correction of the algorithm is easily proved.
Indeed, if $\ell$ denotes the length of the fragment $\underline F$, 
one checks by induction on $i$ that a call to
$\texttt{Wab}(\underline F, i)$ returns the set
$W_i^{(\ga,\gb)} \times \{0\}^{\ell-i}$ (represented by its list of
elements), and similarly for \texttt{Wba} and \texttt{Wbb}.
Regarding the complexity, we first observe that the preparation part of 
the algorithm \textrm{\tt combinatorial\_weights} (consisting of the $9$ 
first lines) requires no more than $O(\ell)$ bit operations.
Now, for $i$ in $\{0, \ldots, \ell{-}1\}$ and $\square$ in
$\{(\ga,\gb), (\gb,\ga), (\gb,\gb)\}$, we set $c_i^\square = 
\Card W_i^\square$.
In order to bound the complexity of the recursive part of the algorithm, 
the key observation is that there exists an absolute constant $C$ such 
that the following holds: when they are called on the input $(\underline 
F, i)$, the routines $\texttt{Wab}$, $\texttt{Wba}$ and $\texttt{Wbb}$ 
perform at most $C \cdot (\ell+i) \cdot c_i^\square$ bit operations with 
$\square = (\ga,\gb)$, $(\gb,\ga)$ and $(\gb,\gb)$ respectively.
This fact is proved by induction on $i$ without difficulty. 
Besides, it readily implies that the bit complexity of the recursive
part of \textrm{\tt combinatorial\_weights} is bounded by
$2C \cdot \ell \cdot \Card \WW(\underline F)$. This finally establishes 
the complexity
bound announced in Theorem~\ref{thm:algoweights}.

\begin{rem}
\label{rem:iterative}
One can also design an iterative version of 
Algorithm~\ref{algo:weights} by running over the integers $i$
between $0$ and $\ell{-}1$ and computing, for each new value of $i$, 
the sets $W_i^{(\ga,\gb)}$, $W_i^{(\gb,\ga)}$ and $W_i^{(\gb,\gb)}$
taking advantage of the previous calculations.
However, proceeding this way, it is quite possible to perform useless 
computations. For example, if $F^\up_{i-1} = F^\up_i =
F^\down_{i-1} = F^\down_i$, we observe that $W_i^{(\ga,\gb)}$, 
$W_i^{(\gb,\ga)}$ and $W_i^{(\gb,\gb)}$ depend only on 
$W_{i-1}^{(\ga,\gb)}$ and $W_{i-1}^{(\gb,\ga)}$ but not on 
$W_{i-1}^{(\gb,\gb)}$; therefore, the computation of 
$W_{i-1}^{(\gb,\gb)}$ is not needed in this case.
Algorithm~\ref{algo:weights} does see this fact, while its iterative 
counterpart does not. 
Nevertheless, it is \emph{true} that the bit complexity of the 
iterative version of Algorithm~\ref{algo:weights} stays within
$O(\ell + \ell \cdot \Card \WW(\bX))$.
This result is obtained by noticing that, for each $i$, at most
one set among $W_i^{(\ga,\gb)}$, $W_i^{(\gb,\ga)}$ and 
$W_i^{(\gb,\gb)}$ can be discarded and then by proving by induction 
on $i$ that the three inequalities
$$c_i^{(\ga,\gb)} \leq  c_i^{(\gb,\ga)} + c_i^{(\gb,\gb)}
  \quad ; \quad
  c_i^{(\gb,\ga)} \leq  c_i^{(\ga,\gb)} + c_i^{(\gb,\gb)}
  \quad ; \quad
  c_i^{(\gb,\gb)} \leq  c_i^{(\ga,\gb)} + c_i^{(\gb,\ga)}$$
hold for all $i$ in $\{0, \ldots, \ell{-}1\}$.
\end{rem}

\begin{rem}
A slight modification of Algorithm~\ref{algo:weights} provides an
algorithm that \emph{enumerates} the elements of $\WW(\bX)$ in such 
a way that each new weight is generated for a cost of $O(f)$ bit 
operations.
This modification can be interesting when $\WW(\bX)$ is large but
we are only interested in computing a small number of weights.
\end{rem}

\subsubsection{Counting weights}
\label{sssec:algocount}

Another related interesting question is the calculation of the
cardinality of $\WW(\bX)$. A naive solution for this consists in
generating the set $\WW(\bX)$ and then counting its elements;
however, this is far for being optimal. 
In this subsection, we rely on the techniques introduced in \S 
\ref{ssec:count} (see also \S \ref{ssec:countdegenerate}) to 
design fast algorithms for performing this task. Precisely, we shall 
prove the following theorem.

\begin{thm}
\label{thm:algocount}
There exists an algorithm which takes as input a gene $\bX$ of
length~$f$ and outputs the cardinality of $\WW(\bX)$ for a cost 
of $O(f^2)$ bit operations.
\end{thm}

Since the cardinality of $\WW(\bX)$ is generally much bigger
than $f$ (it could even be~$2^f$), the complexity announced in 
Theorem~\ref{thm:algocount} is in general much better than that
of Theorem~\ref{thm:algoweights}: counting weights can be done
more efficiently that enumerating them, which is of course not
surprising.

We move to the proof of Theorem~\ref{thm:algocount}. Again, we
give it only in the case where the gene $\bX$ is viable and 
nondegenerate. In the opposite case, 
the proof follows the same pattern but it is more technical as it uses 
all the material introduced in \S \ref{ssec:countdegenerate}.

\newcommand{\cab}{\textrm{\tt cab}}
\newcommand{\cba}{\textrm{\tt cba}}
\newcommand{\cbb}{\textrm{\tt cbb}}
\begin{algorithm}
  \caption{Computation of the number of combinatorial weights}
  \label{algo:count}%
  \SetKwInOut{Input}{Input} %
  \SetKwInOut{Output}{Output} %
  \SetKwProg{nW}{\tt number\_of\_combinatorial\_weights}{}{}%
  \DontPrintSemicolon

  \smallskip

  \nW{$(\underline F)$}{

    \smallskip

    \Input {A fragment $\underline F$}
    \Output{The cardinality of $\WW(\underline F)$}

    \medskip

    \lIf{$F^\down_0 = \gO$}
      {$\cab[0] \leftarrow 0$ \textbf{else} $\cab[0] \leftarrow 1$}
    \lIf{$F^\up_0 = \gO$}
      {$\cba[0] \leftarrow 0$ \textbf{else} $\cba[0] \leftarrow 1$}
    $\cbb[0] \leftarrow 1$\;

    \smallskip

    $\ell \leftarrow \text{length of } \underline F$\;
    \For{$i = 1, \ldots, \ell{-}1$}{
      \lIf{$F^\up_{i-1} \sim F^\up_i$}
        {$\cab[i] \leftarrow \cab[i{-}1]$}
      \lIf{$F^\up_{i-1} \not\sim F^\up_i$}
        {$\cab[i] \leftarrow \cba[i{-}1] + \cbb[i{-}1]$}
      \lIf{$F^\down_{i-1} \sim F^\down_i$}
        {$\cba[i] \leftarrow \cba[i{-}1]$}
      \lIf{$F^\down_{i-1} \not\sim F^\down_i$}
        {$\cba[i] \leftarrow \cab[i{-}1] + \cbb[i{-}1]$}
      \lIf{$F^\up_{i-1} \sim F^\down_{i-1}$}
        {$\cbb[i] \leftarrow \max(\cab[i{-}1], \cba[i{-}1])$}
      \lIf{$F^\up_{i-1} \not\sim F^\down_{i-1}$}
        {$\cbb[i] \leftarrow \cbb[i{-}1]$}
    }

    \lIf{$F^\up_{\ell-1} = \gAB$}
      {\KwRet $\cab[\ell{-}1] + \cbb[\ell{-}1]$}
    \lIf{$F^\down_{\ell-1} = \gAB$}
      {\KwRet $\cba[\ell{-}1] + \cbb[\ell{-}1]$}
    \KwRet $1$
  }
\end{algorithm}

To begin with, we consider Algorithm~\ref{algo:count} that takes as 
input a fragment of length $\ell$ and computes the cardinality of 
$\WW(\underline F)$ for a cost of $O(\ell^2)$ bit operations.
It actually follows closely the formulas of \S \ref{ssec:count}.
We notice that, contrary to what we did in \S \ref{sssec:algoweights}, 
it is more pleasant here to work iteratively (otherwise, we have to 
implement a cache in order to guarantee that the complexity is the 
correct one).

The fact that Algorithm~\ref{algo:count} is correct is clear after
the results of \S \ref{ssec:count}. Moreover, we readily 
see that it performs $O(\ell)$ additions and comparisons on 
integers of the form $\cab[i]$, $\cba[i]$ or $\cbb[i]$. 
Besides, we know from Theorem~\ref{thm:fibo} and its proof
that $\cab[i] \leq \Fib_{i+2}$, $\cba[i] \leq \Fib_{i+2}$ and
$\cbb[i] \leq \Fib_{i+1}$. Consequently $\cab[i]$, $\cba[i]$ and
$\cbb[i]$ have at most $O(\ell)$ digits in their writings in base
$2$. Adding and comparing them can then be achieved for a cost of
$O(\ell)$ bit operations. Putting all together, we find that the
bit complexity of Algorithm~\ref{algo:count} is within $O(\ell^2)$
as wanted.

To complete the computation of the cardinality of $\WW(\bX)$, it
only remains to combine all the contributions of the fragments.
If $\ell_1, \ldots, \ell_m$ denote the respective lengths, this
amounts to multiply integers whose bitsizes are $O(\ell_i)$.
Using the naive multiplication algorithm, this can be done for 
a cost of $O\big(\sum_{i \neq j} \ell_i \ell_j\big)$ bit operations. 
Adding to this the cost of the computation of the cardinality of the 
fragments, which is $O\big(\sum_i \ell_i^2\big)$, we find that the 
complete algorithm runs within
$$\textstyle 
O\Big(\sum_{i=1}^m \sum_{j=1}^m \ell_i \ell_j\Big) \subset 
O\Big(\big(\sum_{i=1}^m \ell_i\big)^2 \Big) \subset O(f^2)$$
bit operations. Theorem~\ref{thm:algocount} is then proved.

\subsection{About Serre weights}

After Theorem~\ref{thm:main}, the results of the previous subsection 
have direct consequences on the enumeration and the counting of common 
Serre weights.

\begin{thm}
\label{thm:algoenum} ~
\begin{myenumerate}[(1)]
\item
There exists an algorithm that takes as input a coherent triple
$(h, \gamma, \gamma')$ with $h$, $\gamma$ and $\gamma'$ written in
base $p$ and outputs the set $\DD(\ttt, \rhobar)$ for a 
cost of 
$$O\big(f \log p + f \cdot \Card \DD(\ttt,\rhobar) \cdot \log p \big)$$
bit operations.

\smallskip

\item
There exists an algorithm that takes as input a coherent triple
$(h, \gamma, \gamma')$ with $h$, $\gamma$ and $\gamma'$ written in
base $p$ and outputs the cardinality of the set $\DD(\ttt, \rhobar)$ 
for a cost of $O(f \log p + f^2)$ bit operations.
\end{myenumerate}
\end{thm}

\begin{proof}
Given a coherent triple $(h, \gamma, \gamma')$, one can compute
its gene $\bX$ using the algorithm of Theorem~\ref{thm:algogene}
for a cost of $O(f \log p)$ bit operations. After this, one can use
Algorithm~\ref{algo:weights} to compute the set $\WW(\bX)$ of
combinatorial weights of $\bX$ for a supplementary cost of:
$$O\big(f + f \cdot \Card \WW(\bX) \big)
= O\big(f + f \cdot \Card \DD(\ttt,\rhobar) \big)$$
bit operations.
It then remains to transform those combinatorial weights into actual
Serre weights using the recipe of \S \ref{ssec:recipe}. Looking at it,
we find that each such transformation requires $O(f)$ operations on
integers less than $p$, corresponding to $O(f \log p)$ bit
operations. Since this operation has to be repeated for each
combinatorial weight, the total complexity of this part amounts to
$$O\big(f \cdot \Card \WW(\bX) \cdot \log p \big)
= O\big(f \cdot \Card \DD(\ttt,\rhobar) \cdot \log p \big)$$
bit operations.
Putting all these inputs together, we deduce the first point of 
the theorem.

The second point is proved in a similar fashion except that we refer to 
Algorithm~\ref{algo:count} instead of Algorithm~\ref{algo:weights} (and 
Theorem~\ref{thm:algocount} instead of
Theorem~\ref{thm:algoweights} for the complexity analysis).
\end{proof}

\subsection{Implementation}

All the algorithms presented in the previous subsections have been
implemented in the SageMath package \texttt{pbtdef}~\cite{CDM3}.
Below, we present an overview of the capabilities of this package.
First of all, we need to import the package. This is done as follows
(after having installed the package, of course):

\begin{Verbatim}[commandchars=\\\{\}]
{\color{incolor}In [{\color{incolor}1}]:} \PY{k+kn}{from} \PY{n+nn}{pbtdef}\PY{n+nn}{.}\PY{n+nn}{all} \PY{k}{import} \PY{o}{*}
\end{Verbatim}

We can now create a $2$-dimensional absolutely irreducible Galois representation
by passing in the relevant parameters:

    \begin{Verbatim}[commandchars=\\\{\}]
{\color{incolor}In [{\color{incolor}2}]:} \PY{n}{p} \PY{o}{=} \PY{l+m+mi}{5}; \PY{n}{f} \PY{o}{=} \PY{l+m+mi}{7}
        \PY{n}{h} \PY{o}{=} \PY{l+m+mi}{4865171564}
        \PY{n}{rhobar} \PY{o}{=} \PY{n}{IrreducibleRepresentation}\PY{p}{(}\PY{n}{p}\PY{p}{,} \PY{n}{f}\PY{p}{,} \PY{n}{h}\PY{p}{)}; \PY{n}{rhobar}
\end{Verbatim}

\vspace{-1ex}

\noindent
\texttt{\color{outcolor}Out[{\color{outcolor}2}]:} %
    \begin{math}
\newcommand{\Bold}[1]{\mathbf{#1}}\verb"Ind"(\omega_{14}^{4865171564})
\end{math}

The method \texttt{weights} returns the set of weights of 
$\rhobar$
(in order to save space, the output has been voluntarily truncated in
the cell below; the complete set of weights has $96$ elements):

    \begin{Verbatim}[commandchars=\\\{\}]
{\color{incolor}In [{\color{incolor}3}]:} \PY{n}{rhobar}\PY{o}{.}\PY{n}{weights}\PY{p}{(}\PY{p}{)}
\end{Verbatim}

\vspace{-1ex}

\noindent
\texttt{\color{outcolor}Out[{\color{outcolor}3}]:} %

\begin{math}
\newcommand{\Bold}[1]{\mathbf{#1}} 
\begin{array}{rlll}
\big\{ &
  \verb"Sym"^{[3, 3, 3, 4, 3, 3, 2]}\otimes \verb"det"^{46544}, &
  \verb"Sym"^{[0, 2, 0, 3, 0, 4, 0]}\otimes \verb"det"^{61648}, & \ldots, \smallskip \\
& \verb"Sym"^{[4, 0, 1, 3, 3, 0, 2]}\otimes \verb"det"^{12264}, &
  \verb"Sym"^{[3, 0, 1, 0, 4, 3, 0]}\otimes \verb"det"^{62139} & \big\}
\end{array}
\end{math}

Similarly, one can manipulate types:

    \begin{Verbatim}[commandchars=\\\{\}]
{\color{incolor}In [{\color{incolor}4}]:} \PY{n}{gamma} \PY{o}{=} \PY{l+m+mi}{58923}; \PY{n}{gammap} \PY{o}{=} \PY{l+m+mi}{77258}
        \PY{n}{t} \PY{o}{=} \PY{n}{Type}\PY{p}{(}\PY{n}{p}\PY{p}{,} \PY{n}{f}\PY{p}{,} \PY{n}{gamma}\PY{p}{,} \PY{n}{gammap}\PY{p}{)}; \PY{n}{t}
\end{Verbatim}

\vspace{-1ex}

\noindent
\texttt{\color{outcolor}Out[{\color{outcolor}4}]:} %
    \begin{math}
\newcommand{\Bold}[1]{\mathbf{#1}}\omega_{7}^{58923} \oplus \omega_{7}^{77258}
\end{math}

    \begin{Verbatim}[commandchars=\\\{\}]
{\color{incolor}In [{\color{incolor}5}]:} \PY{n}{t}\PY{o}{.}\PY{n}{weights}\PY{p}{(}\PY{p}{)}
\end{Verbatim}

\vspace{-1ex}

\noindent
\texttt{\color{outcolor}Out[{\color{outcolor}5}]:}

\begin{math}
\newcommand{\Bold}[1]{\mathbf{#1}}
\begin{array}{rlll}
\big\{ &
  \verb"Sym"^{[0, 2, 0, 3, 0, 4, 0]}\otimes \verb"det"^{61648}, &
  \verb"Sym"^{[3, 1, 0, 3, 3, 3, 0]}\otimes \verb"det"^{62274}, & \ldots, \\
& \verb"Sym"^{[0, 2, 0, 0, 4, 0, 1]}\otimes \verb"det"^{59023}, &
  \verb"Sym"^{[3, 1, 0, 3, 3, 0, 1]}\otimes \verb"det"^{59149} & \big\}
\end{array}
\end{math}

One can easily compute the gene associated to the above inputs,
by using the constructor \texttt{Gene} as follows:

    \begin{Verbatim}[commandchars=\\\{\}]
{\color{incolor}In [{\color{incolor}6}]:} \PY{n}{G} \PY{o}{=} \PY{n}{Gene}\PY{p}{(}\PY{n}{rhobar}\PY{p}{,}\PY{n}{t}\PY{p}{)}; \PY{n}{G}
\end{Verbatim}

\vspace{-1ex}

\noindent
\texttt{\color{outcolor}Out[{\color{outcolor}6}]:}

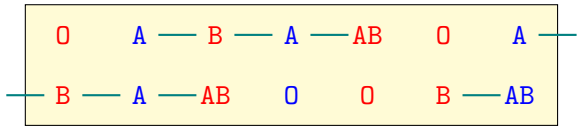
    \begin{math}
\newcommand{\Bold}[1]{\mathbf{#1}}\begin{tikzpicture}[yscale=0.8]
\draw [fill=yellow!20] (-0.5,-0.5) rectangle (6.5,1.5);
\node[color=red] at (0, 0) { \texttt{B} };
\node[color=red] at (0, 1) { \texttt{O} };
\node[color=blue] at (1, 0) { \texttt{A} };
\node[color=blue] at (1, 1) { \texttt{A} };
\node[color=red] at (2, 0) { \texttt{AB} };
\node[color=red] at (2, 1) { \texttt{B} };
\node[color=blue] at (3, 0) { \texttt{O} };
\node[color=blue] at (3, 1) { \texttt{A} };
\node[color=red] at (4, 0) { \texttt{O} };
\node[color=red] at (4, 1) { \texttt{AB} };
\node[color=red] at (5, 0) { \texttt{B} };
\node[color=red] at (5, 1) { \texttt{O} };
\node[color=blue] at (6, 0) { \texttt{AB} };
\node[color=blue] at (6, 1) { \texttt{A} };
\draw[color=teal,thick] (-0.75,0)--(-0.25,0);
\draw[color=teal,thick] (0.25,0)--(0.75,0);
\draw[color=teal,thick] (1.25,1)--(1.75,1);
\draw[color=teal,thick] (1.25,0)--(1.75,0);
\draw[color=teal,thick] (2.25,1)--(2.75,1);
\draw[color=teal,thick] (3.25,1)--(3.75,1);
\draw[color=teal,thick] (5.25,0)--(5.75,0);
\draw[color=teal,thick] (6.25,1)--(6.75,1);
\end{tikzpicture}
\end{math}

The colors in the output indicate the dominant letter at each position (blue for
$\gA$ and red for $\gB$), while the lines drawn between the nucleotides
above correspond to the so-called decorations; they has 
been introduced in \cite{CDM2} and are useful to
read the equation of the associated Kisin variety and those of
its shape stratification.

We can now ask for the combinatorial weights of the gene using the
method \texttt{weights}:

    \begin{Verbatim}[commandchars=\\\{\}]
{\color{incolor}In [{\color{incolor}7}]:} \PY{n}{G}\PY{o}{.}\PY{n}{weights}\PY{p}{(}\PY{p}{)}
\end{Verbatim}

\vspace{-1ex}

\noindent
\texttt{\color{outcolor}Out[{\color{outcolor}7}]:}

    \begin{math}
\newcommand{\Bold}[1]{\mathbf{#1}}
\begin{array}{rl@{\hspace{1.5ex}}l@{\hspace{1.5ex}}l@{\hspace{1.5ex}}ll}
\big\{ &
  \left(0, 0, 1, 0, 1, 0, 1\right), & \left(1, 0, 0, 0, 0, 1, 0\right), &
  \left(1, 0, 1, 0, 0, 0, 1\right), & \left(0, 1, 0, 0, 1, 0, 1\right), \\
& \left(0, 0, 1, 0, 0, 1, 0\right), & \left(1, 0, 0, 0, 1, 1, 0\right), &
  \left(0, 0, 0, 0, 0, 0, 1\right), & \left(1, 0, 1, 0, 1, 1, 0\right), \\
& \left(0, 1, 0, 0, 0, 1, 0\right), & \left(0, 0, 1, 0, 1, 1, 0\right), &
  \left(1, 0, 0, 0, 0, 0, 1\right), & \left(0, 0, 0, 0, 1, 1, 0\right), \\
& \left(1, 0, 1, 0, 0, 1, 0\right), & \left(0, 1, 0, 0, 1, 1, 0\right), &
  \left(0, 0, 1, 0, 0, 0, 1\right), & \left(0, 0, 0, 0, 0, 1, 0\right), \\
& \left(1, 0, 1, 0, 1, 0, 1\right), & \left(0, 1, 0, 0, 0, 0, 1\right), &
  \left(1, 0, 0, 0, 1, 0, 1\right), & \left(0, 0, 0, 0, 1, 0, 1\right) & \big\}
\end{array}
\end{math}

\noindent
We can verify that the weights computed by the software are exactly those
we computed by hand in Example~\ref{ex:geneweights} (except that they do
not appear in the same order).

If we are only interested in the number of combinatorial weights of a
given gene, one should use preferably the method 
\texttt{number\_of\_weights}, which implements Algorithm~\ref{algo:count}
and is then much faster than computing the set of weights in full.

    \begin{Verbatim}[commandchars=\\\{\}]
{\color{incolor}In [{\color{incolor}8}]:} \PY{n}{G}\PY{o}{.}\PY{n}{number\PYZus{}of\PYZus{}weights}\PY{p}{(}\PY{p}{)}
\end{Verbatim}

\vspace{-1ex}

\noindent
\texttt{\color{outcolor}Out[{\color{outcolor}8}]:} %
    \begin{math}
\newcommand{\Bold}[1]{\mathbf{#1}}20
\end{math}

One can ask for the computation of the set of common Serre weights 
of $\rhobar$ and $\ttt$ as follows:

    \begin{Verbatim}[commandchars=\\\{\}]
{\color{incolor}In [{\color{incolor}9}]:} \PY{n}{rhobar}\PY{o}{.}\PY{n}{weights}\PY{p}{(}\PY{n}{t}\PY{p}{)}  \PY{c+c1}{\PYZsh{} or equivalently t.weights(rhobar)}
\end{Verbatim}

\vspace{-1ex}

\noindent
\texttt{\color{outcolor}Out[{\color{outcolor}9}]:}

    \begin{math}
\begin{array}{rlll}
\big\{ &
  \verb"Sym"^{[0, 2, 0, 3, 0, 4, 0]}\otimes \verb"det"^{61648}, &
  \verb"Sym"^{[0, 1, 1, 0, 4, 3, 0]}\otimes \verb"det"^{62138}, & \ldots, \smallskip \\
& \verb"Sym"^{[4, 2, 1, 0, 4, 3, 3]}\otimes \verb"det"^{77758}, &
  \verb"Sym"^{[0, 2, 0, 0, 4, 0, 1]}\otimes \verb"det"^{59023}  & \big\}
\end{array}
\end{math}

It is possible to create a gene by passing in its sequence of 
nucleotides. For example, the following creates the ``Fibonacci gene'' 
which appears in Theorem~\ref{thm:fibo}:

    \begin{Verbatim}[commandchars=\\\{\}]
{\color{incolor}In [{\color{incolor}10}]:} \PY{n}{FibG} \PY{o}{=} \PY{n}{Gene}\PY{p}{(}\PY{p}{[} \PY{l+s+s1}{\PYZsq{}}\PY{l+s+s1}{O}\PY{l+s+s1}{\PYZsq{}}\PY{p}{,} \PY{l+s+s1}{\PYZsq{}}\PY{l+s+s1}{A}\PY{l+s+s1}{\PYZsq{}}\PY{p}{,} \PY{l+s+s1}{\PYZsq{}}\PY{l+s+s1}{B}\PY{l+s+s1}{\PYZsq{}}\PY{p}{,} \PY{l+s+s1}{\PYZsq{}}\PY{l+s+s1}{A}\PY{l+s+s1}{\PYZsq{}}\PY{p}{,} \PY{l+s+s1}{\PYZsq{}}\PY{l+s+s1}{B}\PY{l+s+s1}{\PYZsq{}}\PY{p}{,} \PY{l+s+s1}{\PYZsq{}}\PY{l+s+s1}{A}\PY{l+s+s1}{\PYZsq{}}\PY{p}{,} \PY{l+s+s1}{\PYZsq{}}\PY{l+s+s1}{B}\PY{l+s+s1}{\PYZsq{}}\PY{p}{,} \PY{l+s+s1}{\PYZsq{}}\PY{l+s+s1}{A}\PY{l+s+s1}{\PYZsq{}}\PY{p}{,} \PY{l+s+s1}{\PYZsq{}}\PY{l+s+s1}{B}\PY{l+s+s1}{\PYZsq{}}\PY{p}{,}
                       \PY{l+s+s1}{\PYZsq{}}\PY{l+s+s1}{B}\PY{l+s+s1}{\PYZsq{}}\PY{p}{,} \PY{l+s+s1}{\PYZsq{}}\PY{l+s+s1}{A}\PY{l+s+s1}{\PYZsq{}}\PY{p}{,} \PY{l+s+s1}{\PYZsq{}}\PY{l+s+s1}{B}\PY{l+s+s1}{\PYZsq{}}\PY{p}{,} \PY{l+s+s1}{\PYZsq{}}\PY{l+s+s1}{A}\PY{l+s+s1}{\PYZsq{}}\PY{p}{,} \PY{l+s+s1}{\PYZsq{}}\PY{l+s+s1}{B}\PY{l+s+s1}{\PYZsq{}}\PY{p}{,} \PY{l+s+s1}{\PYZsq{}}\PY{l+s+s1}{A}\PY{l+s+s1}{\PYZsq{}}\PY{p}{,} \PY{l+s+s1}{\PYZsq{}}\PY{l+s+s1}{B}\PY{l+s+s1}{\PYZsq{}}\PY{p}{,} \PY{l+s+s1}{\PYZsq{}}\PY{l+s+s1}{A}\PY{l+s+s1}{\PYZsq{}}\PY{p}{,} \PY{l+s+s1}{\PYZsq{}}\PY{l+s+s1}{AB}\PY{l+s+s1}{\PYZsq{}} \PY{p}{]}\PY{p}{)}
         \PY{n}{FibG}
\end{Verbatim}

\vspace{-1ex}

\noindent
\texttt{\color{outcolor}Out[{\color{outcolor}10}]:}

    \begin{math}
\newcommand{\Bold}[1]{\mathbf{#1}}\begin{tikzpicture}[yscale=0.8]
\draw [fill=yellow!20] (-0.5,-0.5) rectangle (8.5,1.5);
\node[color=red] at (0, 0) { \texttt{B} };
\node[color=red] at (0, 1) { \texttt{O} };
\node[color=blue] at (1, 0) { \texttt{A} };
\node[color=blue] at (1, 1) { \texttt{A} };
\node[color=red] at (2, 0) { \texttt{B} };
\node[color=red] at (2, 1) { \texttt{B} };
\node[color=blue] at (3, 0) { \texttt{A} };
\node[color=blue] at (3, 1) { \texttt{A} };
\node[color=red] at (4, 0) { \texttt{B} };
\node[color=red] at (4, 1) { \texttt{B} };
\node[color=blue] at (5, 0) { \texttt{A} };
\node[color=blue] at (5, 1) { \texttt{A} };
\node[color=red] at (6, 0) { \texttt{B} };
\node[color=red] at (6, 1) { \texttt{B} };
\node[color=blue] at (7, 0) { \texttt{A} };
\node[color=blue] at (7, 1) { \texttt{A} };
\node[color=red] at (8, 0) { \texttt{AB} };
\node[color=red] at (8, 1) { \texttt{B} };
\draw[color=red,thick] (-0.8,0.2)--(-0.2,0.8);
\draw[color=teal,thick] (0.25,0)--(0.75,0);
\draw[color=teal,thick] (1.25,1)--(1.75,1);
\draw[color=teal,thick] (1.25,0)--(1.75,0);
\draw[color=teal,thick] (2.25,1)--(2.75,1);
\draw[color=teal,thick] (2.25,0)--(2.75,0);
\draw[color=teal,thick] (3.25,1)--(3.75,1);
\draw[color=teal,thick] (3.25,0)--(3.75,0);
\draw[color=teal,thick] (4.25,1)--(4.75,1);
\draw[color=teal,thick] (4.25,0)--(4.75,0);
\draw[color=teal,thick] (5.25,1)--(5.75,1);
\draw[color=teal,thick] (5.25,0)--(5.75,0);
\draw[color=teal,thick] (6.25,1)--(6.75,1);
\draw[color=teal,thick] (6.25,0)--(6.75,0);
\draw[color=teal,thick] (7.25,1)--(7.75,1);
\draw[color=teal,thick] (7.25,0)--(7.75,0);
\draw[color=red,thick] (8.2,0.8)--(8.8,0.2);
\end{tikzpicture}
\end{math}

\noindent
We can compute its number of combinatorial weights and check that it
is indeed in the Fibonacci sequence:

    \begin{Verbatim}[commandchars=\\\{\}]
{\color{incolor}In [{\color{incolor}11}]:} \PY{n}{FibG}\PY{o}{.}\PY{n}{number\PYZus{}of\PYZus{}weights}\PY{p}{(}\PY{p}{)}
\end{Verbatim}

\vspace{-1ex}

\noindent
\texttt{\color{outcolor}Out[{\color{outcolor}11}]:} %
    \begin{math}
\newcommand{\Bold}[1]{\mathbf{#1}}89
\end{math}

Finally, one can generate a pair $(\ttt, \rhobar)$ exhibiting this
particular gene using the method \texttt{random\_individual} (and
providing a value of $p$):

    \begin{Verbatim}[commandchars=\\\{\}]
{\color{incolor}In [{\color{incolor}12}]:} \PY{n}{FibG}\PY{o}{.}\PY{n}{random\PYZus{}individual}\PY{p}{(}\PY{n}{p}\PY{o}{=}\PY{l+m+mi}{5}\PY{p}{)}
\end{Verbatim}

\vspace{-1ex}

\noindent
\texttt{\color{outcolor}Out[{\color{outcolor}12}]:} %
    \begin{math}
\newcommand{\Bold}[1]{\mathbf{#1}}\left(\verb"Ind"(\omega_{18}^{1136706441368}),\,\, \omega_{9}^{500613} \oplus \omega_{9}^{956342}\right)
\end{math}

\noindent
Notice that the output is not deterministic; more precisely, it is 
uniformly distributing among all possibilities.

    \begin{Verbatim}[commandchars=\\\{\}]
{\color{incolor}In [{\color{incolor}13}]:} \PY{n}{FibG}\PY{o}{.}\PY{n}{random\PYZus{}individual}\PY{p}{(}\PY{n}{p}\PY{o}{=}\PY{l+m+mi}{5}\PY{p}{)}
\end{Verbatim}

\vspace{-1ex}

\noindent
\texttt{\color{outcolor}Out[{\color{outcolor}13}]:} %
    \begin{math}
\newcommand{\Bold}[1]{\mathbf{#1}}\left(\verb"Ind"(\omega_{18}^{654399553802}),\,\, \omega_{9}^{253672} \oplus \omega_{9}^{709401}\right)
\end{math}

\end{document}